    \documentclass[11pt]{article}
    \usepackage[margin=1in, a4paper]{geometry}
    \usepackage{amsthm,amsfonts,amsmath,amssymb,bbm}
    \usepackage{mleftright}
    \usepackage{etoolbox}
    \usepackage[dvipsnames]{xcolor}
    \usepackage{hyperref}
    \usepackage{graphicx}
    \usepackage{authblk}
    
    \hypersetup{
        colorlinks=true,
        linkcolor=blue,
        urlcolor=blue,
        citecolor=blue
    }
    
    \makeatletter
    \let\ams@starttoc\@starttoc
    \makeatother
    \usepackage[parfill]{parskip}
    \makeatletter
    \let\@starttoc\ams@starttoc
    \patchcmd{\@starttoc}{\makeatletter}{\makeatletter\parskip\z@}{}{}
    \makeatother
    
    \usepackage[numbers,square]{natbib}
    \usepackage{enumerate}
    
    \numberwithin{equation}{section}
    
    \newtheorem{theorem}{Theorem}[section]
    \newtheorem{proposition}[theorem]{Proposition}
    \newtheorem{lemma}[theorem]{Lemma}
    \newtheorem{corollary}[theorem]{Corollary}
    
    \theoremstyle{definition}

    \newtheorem{remark}[theorem]{Remark}
    
    \newcommand{\re}{\textnormal{Re}}   
    \newcommand{\lb}{\mleft(}
    \newcommand{\rb}{\mright)}
    \newcommand{\lbrb}[1]{\lb #1 \rb}
        \newcommand{\gammaEuler}{\gamma_{\mathrm E}}

    \newcommand{\Eb}{\mathsf E}
    \renewcommand{\P}{\mathsf P}
    
    \newcommand{\Cb}{\mathbb C}
    \newcommand{\Nb}{\mathbb N}
    \newcommand{\Rb}{\mathbb R}
    \newcommand{\ind}[1]{\mathbbm{1}_{\mleft\{#1\mright\}}}
    
    \newcommand{\Var}{\operatorname{Var}}
    \newcommand{\Cov}{\operatorname{Cov}}
    \newcommand{\D}{\mathrm{d}}
    \newcommand{\Pois}{\operatorname{Poisson}}
    \newcommand{\Gumbel}{\operatorname{Gumbel}}
    \newcommand{\bo}{\mathrm{O}}
\newcommand{\so}{\mathrm{o}}
    
        \renewcommand{\Re}{\operatorname{Re}}
        \renewcommand{\Im}{\operatorname{Im}}
        
    \newenvironment{enumeratei}{\begin{enumerate}[(i)]}{\end{enumerate}}

\begin{document}
    \title{Asymptotics for Beta-Splitting Trees via Homogeneous
    Fragmentations and Meromorphic Potential Theory}
    
    \author{
    Yoana R. Chorbadzhiyska$^{*}$,
    Martin Minchev$^{\dagger, *}$,
    and Mladen Savov$^{*,\ddagger}$
    }   
    
    \date{}
    \maketitle
    \thispagestyle{empty}
    
    \begingroup
    \renewcommand{\thefootnote}{\fnsymbol{footnote}}
    \footnotetext[1]{Faculty of Mathematics and Informatics,
    Sofia University ``St. Kliment Ohridski'', Bulgaria.}
    \footnotetext[2]{Institute of Mathematics,
    University of Zurich, Switzerland.}
    \footnotetext[3]{Institute of Mathematics and Informatics,
    Bulgarian Academy of Sciences, Bulgaria.}
    \endgroup
    
    \renewcommand{\thefootnote}{\arabic{footnote}}
    \setcounter{footnote}{0}
    
    \begin{abstract}
    Inspired by recent work of Aldous, Janson, and Pittel on the critical
    beta-splitting model, we study the full beta-splitting family for
    $\beta>-2$ by employing a canonical continuous-time embedding into a homogeneous
    exchangeable fragmentation. In this representation, the asymptotic frequency
    of a tagged fragment is $e^{-\xi_t}$, where $\xi$ is a subordinator with
    Laplace exponent
    \[
            \phi_\beta(z)
            =
            \int_0^1(1-s^z)s^{\beta+1}(1-s)^\beta\D s.
    \]
    We express the continuous height of a typical leaf, its occupation
    probabilities, the discrete height, and the total continuous-time length in
    terms of the potential measure of $\xi$. Renewal theory yields first-order
    asymptotics and a central limit theorem for the continuous-time height. Next, a regenerative-composition representation gives Gaussian limits for the discrete
height above and at the critical value, and a non-Gaussian power-law limit below it. 

We obtain residue
    expansions for the potential measure of the subordinator $\xi$, which transfers to the mean continuous height of the tree. For the case
    $-2<\beta\leq0$, this is facilitated by the property that $\xi$ belongs to the family of {meromorphic} subordinators, and for $\beta>0$, we employ the theory of generalised
    Nevanlinna functions, augmented by the specific structure of these subordinators, which can be roughly viewed as extended meromorphic subordinators. 
    
    We finally investigate the maximum continuous-time height $D_n^*$.  Combining
    additive martingales for homogeneous fragmentations with a conditional
    Poisson approximation of sums of weakly dependent Bernoulli variables that account for collisions, we prove
    \[
            \frac{D_n^*}{\log n}
            \longrightarrow
            \frac{2}{\phi_\beta(1)},
            \qquad\text{in probability},
    \]
    and establish a mixed Gumbel limit for
    $\phi_\beta(1)D_n^*-2\log n$. At the critical value $\beta=-1$, this
    identifies the limiting constant as $2$ and resolves an open problem of
    Aldous and Janson.

    \end{abstract}
    
    {\small
    
    }
    \tableofcontents
    \section{Introduction}
    
    The beta-splitting model was introduced in \cite{Aldous-1996} as a
    one-parameter family of distributions on rooted binary cladograms. In phylogenetics, a cladogram is a rooted tree whose leaves represent species,
    or other biological groups, and whose branches record their successive
    separation from common ancestors. A group consisting of a common ancestor and
    all its descendants is called a \emph{clade}, from the Greek word \emph{klados}, meaning ``branch''. Such parametric models make it
    possible to compare the balance of observed tree shapes with predictions from
    theoretical models.
    
    In the beta-splitting model, a clade containing $m$ leaves is split into two
    daughter clades containing $i$ and $m-i$ leaves, with a probability depending
    on a parameter $\beta$. This parameter controls the balance of the resulting
    tree: as $\beta\downarrow-2$, the trees become increasingly unbalanced and
    approach the comb; $\beta=-3/2$ gives the
    \emph{proportional to distinguishable arrangements} (PDA) model;
    $\beta=0$ gives the Yule, or
    \emph{equal-rates Markov} (ERM), model; and increasing $\beta$ favours more
    balanced splits. We refer to
    \cite{Aldous-2001-Phylogenetic,Sainudiin-Veber-2016} for background and
    biological interpretations of these models.
    
    In a recent series of papers, Aldous and Pittel
    \cite{Aldous-Pittel-2025} and Aldous and Janson
    \cite{beta2-arxiv,Aldous-Janson-Exchangeable,
    Aldous-Janson-2025-Mellin} studied the model at the critical value
    $\beta=-1$. Background, related quantities, and many open questions are
    discussed in the survey \cite{beta2-arxiv}. This critical model can be
    described briefly as follows:
    \begin{itemize}
        \item Start with $m$ labels in a single block.
        \item Split a block of size $m$ into blocks of sizes $i$ and $m-i$ with
              probability proportional to $1/(i(m-i))$.
        \item Repeat recursively until all blocks are singletons.
    \end{itemize}
    The resulting singletons are the leaves of the tree. This model has a natural embedding in continuous time through the rule
    \begin{equation}
        \label{eq: rule critical}
    \text{split independently a block of size $m$ after an
    $\operatorname{Exp}(h_{m-1})$ time, where
    $h_m:=\sum_{r=1}^m\frac1r$.}
    \end{equation}
    We refer to quantities of the discrete model as \textit{discrete} or
    \textit{discrete-time} and for the embedded one as \textit{continuous}
    or \textit{continuous-time}.
    
    A basic question is to understand the behaviour of a uniformly chosen leaf,
    for example its continuous-time height, its discrete-time height, their moments and
    central limit theorems, and the expected total length of the tree. Aldous and
    Pittel \cite{Aldous-Pittel-2025} study several of these questions through
    recurrences. Aldous and Janson
    \cite{Aldous-Janson-2025-Mellin} obtain sharper results in the critical case
    through a Mellin-transform analysis of the function
    \[
            z\mapsto\psi(z+1)-\psi(1),
            \qquad
            \text{where}\qquad
            \psi(z):=(\log\Gamma(z))'
    \]
    is the digamma function.
    
    In this work, we use the framework of homogeneous exchangeable
    fragmentations to study these questions for all admissible
    $\beta>-2$. More precisely, the canonical continuous-time beta-splitting
    tree on $n$ leaves is the restriction to $[n]$ of a homogeneous
    fragmentation of $\Nb$, and the frequency of a tagged fragment is described
    by a canonical subordinator.
    
    The connection between Markov branching trees and fragmentation processes
    has been studied before. In particular,
    \cite{Haas-Miermont-Pitman-Winkel-2008} relate sampling-consistent Markov
    branching trees to dislocation measures and obtain continuum-tree scaling
    limits for discrete fragmentation trees, including beta-splitting models.
    Our use of fragmentation is different: rather than taking a scaling limit of
    the discrete tree, we use the homogeneous fragmentation as an exact
    continuous-time representation of every finite beta-splitting tree.
    
    Let us briefly sketch this connection; see
    Section~\ref{sec:exchangeable} for a more thorough treatment. Start with an
    interval fragmentation $I(t)$ of $[0,1]$, in which different intervals split
    independently at rates described by a dislocation measure $\nu(\D x)$ on
    $[1/2,1)$, where $x$ denotes the larger relative mass. Let
    $U_1,U_2,\ldots$ be iid uniform random variables on $(0,1)$. At time $t$,
    place two labels $i$ and $j$ in the same block when $U_i$ and $U_j$ belong to
    the same interval of $I(t)$. This gives an exchangeable partition
    $\Pi(t)$ of $\Nb$. Its restriction
    \[
            \Pi_n(t):=\Pi(t)_{|[n]}
    \]
    is a fragmentation process on the partitions of $[n]$.
    The process $\Pi_n$ is the continuous-time embedding of a Markov branching
    scheme whose split probabilities satisfy
    \[
            q(n,i)
            \propto
            \binom ni
            \int_0^1x^i(1-x)^{n-i}\nu(\D x),
            \qquad
            \text{for }1\le i\le n-1.
    \]
    The corresponding continuous-time rule is
    \begin{equation}
        \label{eq: rule general}
    \text{split independently a block of size $m$ after an
    $\operatorname{Exp}(\phi_\nu(m-1))$ time,}
    \end{equation}
    where
    \[
            \phi_\nu(k)
            :=
            \int_{1/2}^1
            \left[1-x^{k+1}-(1-x)^{k+1}\right]\nu(\D x).
    \]
    
    Taking
    \[
            \nu_\beta(\D x)
            =
            x^\beta(1-x)^\beta\D x,
            \qquad \text{on }x\in[1/2,1),
            \qquad
            \text{for }\beta>-2,
    \]
    we recover the beta-splitting model from \cite{Aldous-1996}. After a change
    of variables, the exponent of the tagged fragment becomes
    \begin{equation}\label{eq:phi-beta-intro}
            \phi_\beta(z)
            =
            \int_0^1
            (1-s^z)s^{\beta+1}(1-s)^\beta\D s.
    \end{equation}
    For $\beta>-1$, this can be written using the beta function
    \[
            B(z_1,z_2)
            :=
            \int_0^1s^{z_1-1}(1-s)^{z_2-1}\D s
            =
            \frac{\Gamma(z_1)\Gamma(z_2)}
            {\Gamma(z_1+z_2)},
            \qquad
            \text{for }\Re z_1>0\text{ and }\Re z_2>0,
    \]
    as
    \[
            \phi_\beta(z)
            =
            B(\beta+2,\beta+1)
            -
            B(z+\beta+2,\beta+1).
    \]
    For $\beta\in(-2,-1)$, the same identity is understood through analytic
    continuation of the beta function.
    
    At the critical value $\beta=-1$, a direct calculation gives
    \[
            \phi_{-1}(z)
            =
            \psi(z+1)-\psi(1),
            \qquad\text{and}\qquad
            \phi_{-1}(k)=h_k.
    \]
    In particular,
    \[
            \phi_{-1}(m-1)=h_{m-1},
    \]
    so \eqref{eq: rule general} reduces exactly to
    \eqref{eq: rule critical}. Thus the continuous-time rule introduced in the
    critical model is the specification of the general fragmentation
    construction.
    
    The value $\beta=-1$ is called \emph{critical} because it separates
    different regimes: first,
    \[
            \nu_\beta([1/2,1))<\infty
            \qquad\text{if and only if}\qquad
            \beta>-1,
    \]
    so $\beta=-1$ is the boundary between finite and infinite dislocation
    activity. Second, the same transition appears in the expected discrete-time height of a
    typical leaf $L_n$:
    \[
            \Eb L_n\asymp\log n
            \quad(\beta>-1),
            \qquad
            \Eb L_n\asymp(\log n)^2
            \quad(\beta=-1),
            \qquad
            \Eb L_n\asymp n^{-\beta-1}
            \quad(-2<\beta<-1).
    \]
    These three orders already appear in the original analysis of
    \cite[Proposition~4]{Aldous-1996}. We recover them through the tagged
    subordinator and give the corresponding exact constants in
    Corollary~\ref{cor:discrete-height-phase-results}. Figure~\ref{fig:beta-splitting-trees} illustrates the effect of $\beta$ on
    the continuous-time tree and on its discrete shape.
    
    \begin{figure}[t]
    \centering
    \includegraphics[width=\textwidth]{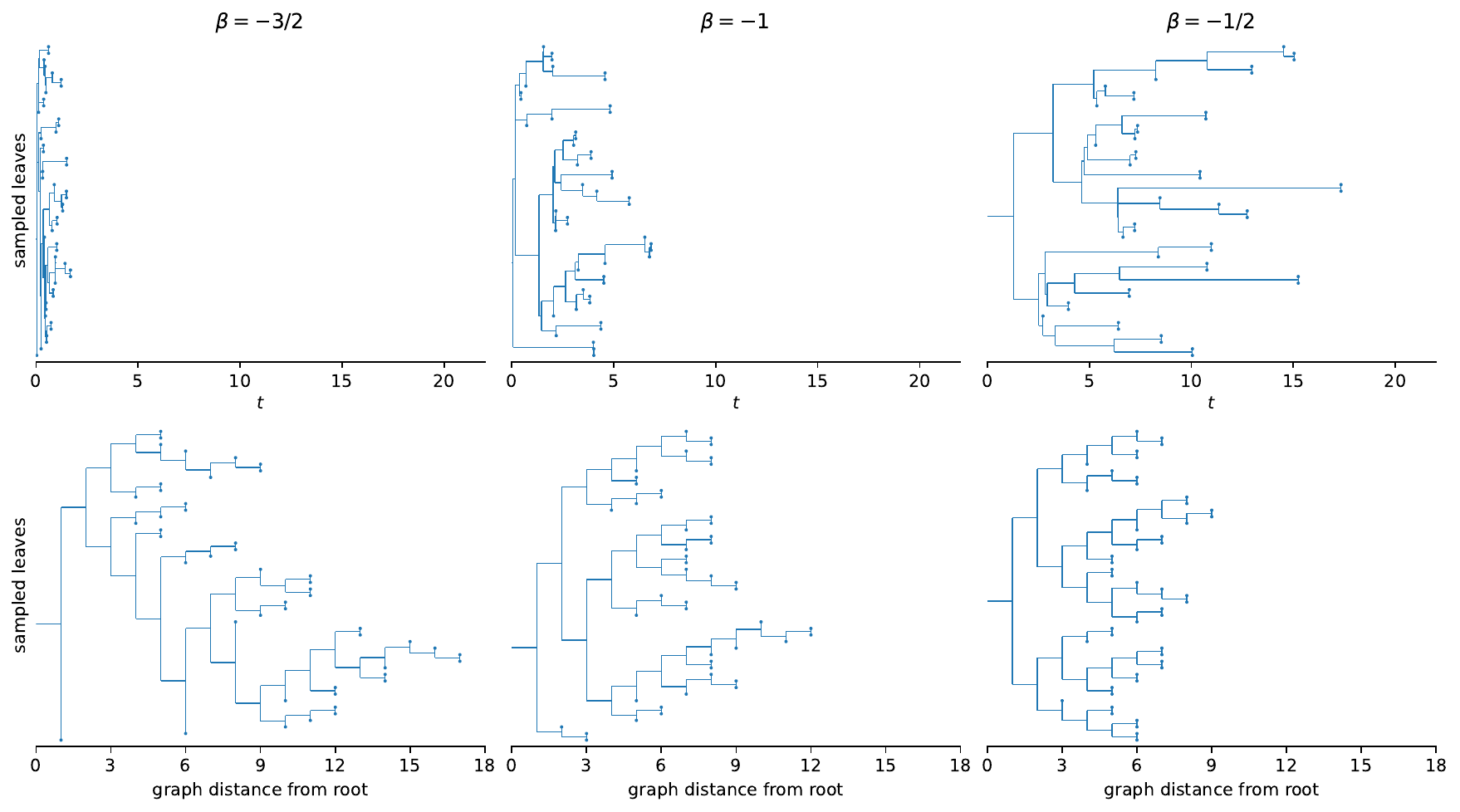}
    \caption{Simulations of beta-splitting trees with $n=48$ leaves for
    $\beta=-3/2$, $\beta=-1$, and $\beta=-1/2$. The top row shows the canonical
    continuous-time trees on the common time range $t\in[0,22]$. The bottom row
    shows the same realizations with unit edge lengths, and hence only their
    discrete tree shapes.}
    \label{fig:beta-splitting-trees}
    \end{figure}
    
   By exchangeability, $\Eb L_n$ is also the expected number of splits along
the spine leading to label~$1$. The number of labels in the block containing
label~$1$ evolves as a decreasing Markov chain. At $\beta=-1$, this is the
\textit{harmonic descent chain} introduced in
\cite{Aldous-Janson-Li-2024}; see also
\cite[Section~3.1]{beta2-arxiv}. The name refers to the harmonic numbers
appearing in its transition probabilities in \eqref{eq: rule critical}.
The successive numbers of labels which leave the tagged block form a
composition of $n-1$, that is, an ordered sequence of positive integers with
sum $n-1$. Gnedin and Pitman introduced the class of
\textit{regenerative composition structures} in
\cite{Gnedin-Pitman-2005}. Iksanov
\cite{Iksanov-2025-Regenerative} identified the critical harmonic descent
chain with the decrement chain of such a composition, and this connection was
further developed in
\cite[Proposition~2.1]{Iksanov-Nikitin-Yakymiv-2026}. We observe that the same
representation applies throughout the beta-splitting family and use it to
obtain the corresponding limit laws in
Corollary~\ref{cor:Ln-clt-appendix} and
Remark~\ref{rem:Ln-power-limit} of
Appendix~\ref{app:discrete-height}.

    The importance of $\phi_\beta$ stems from the tagged fragment. If
    $\Pi_1(t)$ is the block containing label $1$, then
    \[
            |\Pi_1(t)|=e^{-\xi_t},
            \qquad \text{for a subordinator $\xi$ with } 
            \Eb e^{-z\xi_t}=e^{-t\phi_\beta(z)},
    \]
    that is, $\phi_\beta$ is the Laplace exponent of some subordinator $\xi$.
    Expectations of several functionals of a typical leaf can consequently be
    written in terms of the potential measure
    of $\xi$, defined by
    \begin{equation}\label{eq: U_beta}
            U_\beta(\D x)
            :=
            \int_0^\infty\P(\xi_t\in\D x)\D t,
            \qquad
            \text{where}\qquad
            \int_{[0,\infty)}e^{-zx}U_\beta(\D x)
            =
            \frac1{\phi_\beta(z)}.
    \end{equation}
    Standard renewal theory for $\xi$ then gives the first-order behaviour and the central limit
    theorem for the continuous height. For sharper results, we study the
    meromorphic continuation of $1/\phi_\beta$.
    
    For $-2<\beta<0$, the L\'evy density of $\xi$ is a positive mixture of
    exponentials, so $\xi$ is a meromorphic subordinator in the sense of
    \cite{Kuznetsov-Kyprianou-Pardo-2012}. In this case, $\phi$ is \textit{special Bernstein function}, and thus has a tractable Mittag--Leffler expansion. This gives exact residue expansions
    for the potential measure and the mean height. The case $\beta=0$ is
    explicit. For $\beta>0$, the positive-mixture property is lost. However,
    we still provide an infinite series expansion as well error bounds for its partial sums, using inversion techniques. One of the main structural differences from the case $\beta<0$ is the appearance of nonreal poles of the meromorphic extension of
    $1/\phi$. A link with with the theory of generalised Nevanlinna functions ensures their finite number, and we get an even more precise understanding 
    of their structure in the specific case of beta-splitting via complex-analytical arguments in Appendix~\ref{app:extmeromorphic-potential-proof}.

    Taking advantage of the same fragmentation representation, we also study the
    maximum continuous height $D_n^*$. If $N_i(t)$ is the number of sampled
    labels in the $i$th fragment and
    \[
            C_n(t)
            =
            \sum_i\binom{N_i(t)}2,
    \]
    then
    \[
            D_n^*\le t
            \qquad\text{if and only if}\qquad
            C_n(t)=0.
    \]
    Thus the maximum height is a shattering question for the whole
    fragmentation. Since
    \[
            \Eb C_n(t)
            =
            \binom n2
            \Eb\sum_iP_i(t)^2
            =
            \binom n2e^{-t\phi_\beta(1)},
    \]
    the natural scale is
    $2\log n/\phi_\beta(1)$. We prove in Theorem \ref{thm:max-height-results} that
    \[
            \frac{D_n^*}{\log n}
            \longrightarrow
            \frac2{\phi_\beta(1)},
            \qquad\text{in probability},
    \]
    and that $\phi_\beta(1)D_n^*-2\log n$ has a mixed Gumbel limit. At
    $\beta=-1$, this resolves \cite[Open Problem 3]{beta2-arxiv} and identifies the constant as $2$.
    \section{Results}\label{sec:results}
    
    We first state some identities linking quantities from the beta-splitting model for all
    admissible $\beta$ and the subordinator $\xi$ similar to \cite[Proposition~4.1]{Aldous-Janson-2025-Mellin}. Next, we state the discrete height
phase transition and refer to the corresponding fluctuation results. We then give the
    meromorphic expansions separately for $-2<\beta<0$, $\beta=0$, and
    $\beta>0$. The main analytic result is the residue description of the
    potential measure, including the positive-beta case. The final result concerns
    the maximum continuous-time height and is independent of the potential
    analysis.
    
    Keeping the notation of \cite{beta2-arxiv}, let
    \begin{itemize}
        \item $D_n$ be the continuous height of a uniformly chosen leaf,
        \item $L_n$ its discrete height,
        \item $a_\beta(n,j)$ the probability that the tagged leaf is ever
              contained in a clade of size $j$,
        \item $\Lambda_n$ the total continuous length of the tree.
    \end{itemize}
    These quantities admit formulas through $\xi$, and therefore their expectations can be expressed through
    its potential measure $U_\beta$.  The proofs use mainly the paintbox construction and
    are, in essence, the same as in the critical case; see
    \cite[Proposition~4.1]{Aldous-Janson-2025-Mellin}.
    For every nonnegative Borel function $f$ on $[0,\infty)$, introduce
    \begin{equation}\label{eq:potential-operator-results}
            \mathcal U_\beta f
            :=
            \Eb\left[\int_0^\infty f(\xi_t)\D t\right]
            =
            \int_{[0,\infty)}f(x)U_\beta(\D x),
    \end{equation}
    and for $n\ge2$ and $2\le j\le n$, define
    \begin{equation}
        \begin{split}
            \label{eq:def F_n}
            F_n(x)
            &:=
            1-(1-e^{-x})^{n-1},\\
            G_{n,j}(x)
            &:=
            e^{-(j-1)x}(1-e^{-x})^{n-j},\\
            H_n(x)
            &:=
            e^x\left[1-(1-e^{-x})^n\right]
            -n(1-e^{-x})^{n-1}.
        \end{split}
    \end{equation}
    \begin{theorem}[Potential identities]
    \label{thm:potential-identities-results}
    For every $\beta>-2$ and $n\ge2$, the following identities hold.
    \begin{enumeratei}
    
    \item For every $t\ge0$, the continuous height and the expected total length satisfy
    \begin{equation}\label{eq:potential-identities-Dn-Lambda}
            \P(D_n>t)=\Eb\bigl(F_n(\xi_t)\bigr),
            \qquad
            \Eb D_n=\mathcal U_\beta F_n,
            \qquad\text{and}\qquad
            \Eb\Lambda_n=\mathcal U_\beta H_n.
    \end{equation}
    
    \item The expected discrete height satisfies
    \begin{equation}\label{eq:an-j-results}
            \Eb L_n=\sum_{j=2}^n a_\beta(n,j),
            \qquad\text{where}\qquad
            a_\beta(n,j)
            =
            \phi_\beta(j-1)\binom{n-1}{j-1}
            \mathcal U_\beta G_{n,j}.
    \end{equation}
    
    \item Let $M_{n-1}$ be the maximum of $n-1$ independent exponential random
    variables of rate $1$, independent of $\xi$. Then
    \begin{equation}\label{eq:Dn-passage-results}
            D_n\stackrel{d}=T_{M_{n-1}},
            \qquad\text{where}\qquad
            T_x:=\inf\{t\ge0:\xi_t\ge x\}.
    \end{equation}
    
    \end{enumeratei}
    \end{theorem}
    
    Put
    \begin{equation}\label{eq:mu-tau-results}
            \mu_\beta:=\phi_\beta'(0)
            =-\int_0^1\log(s)s^{\beta+1}(1-s)^\beta\D s,
            \quad\text{and}\quad
            \tau_\beta^2:=-\phi_\beta''(0)
            =\int_0^1(\log s)^2s^{\beta+1}(1-s)^\beta\D s.
    \end{equation}
    Based on the representation of $D_n$ as a passage time in
    \eqref{eq:Dn-passage-results}, standard renewal theory for $\xi$ gives 
    another proof of the central limit theorem for $D_n$, considered at
    $\beta=-1$ in \cite[Theorem~1.7]{Aldous-Pittel-2025},
    \cite[Theorem~2]{beta2-arxiv},
    \cite[Theorem~1.5]{Aldous-Janson-2025-Mellin}, and
    \cite[Theorem~9]{Kolesnik-2025}.
    
    \begin{corollary}[Typical continuous height]
    \label{cor:Dn-renewal-results}
    For every $\beta>-2$, as $n\to\infty$,
    \begin{equation}\label{eq:Dn-renewal-results}
            \Eb D_n\sim\frac{\log n}{\mu_\beta},
            \qquad\text{and}\qquad
            \frac{D_n-\mu_\beta^{-1}\log n}
            {\sqrt{\tau_\beta^2\mu_\beta^{-3}\log n}}
            \xrightarrow{d}\mathcal N(0,1).
    \end{equation}
    \end{corollary}
    By exchangeability, $\Eb L_n=\Eb\overline D_n$ in the notation of
    \cite{Aldous-1996}. Proposition~4 of that reference gives the following
    first-order asymptotics. The constants below are its constants rewritten in
    terms of the Laplace exponent of the tagged subordinator. An alternative derivation, using Theorem~\ref{thm:potential-identities-results}, is given in
    Appendix~\ref{app:discrete-height}.
    
    \begin{corollary}[Discrete-height phase transition]
    \label{cor:discrete-height-phase-results}
    For every $\beta>-2$, as $n\to\infty$,
    \begin{equation}\label{eq:discrete-height-phase-results}
            \Eb L_n
            \sim
            \begin{cases}
            \displaystyle
            \frac{\phi_\beta(\infty)}{\mu_\beta}\log n,
            &\beta>-1,\\[1em]
            \displaystyle
            \frac1{2\mu_{-1}}(\log n)^2,
            &\beta=-1,\\[1em]
            \displaystyle
            \frac{\Gamma(\beta+2)}{(-\beta-1)\phi_\beta(-\beta-1)}
            n^{-\beta-1},
            &-2<\beta<-1,
            \end{cases}
    \end{equation}
    where
    \[
            \phi_\beta(\infty)
            :=
            \lim_{a\to\infty}\phi_\beta(a)
            =
            B(\beta+2,\beta+1),
            \qquad\text{for }\beta>-1.
    \]
    \end{corollary}
    The corresponding fluctuation results, based on the theory of regenerative
compositions, are collected in Appendix~\ref{app:discrete-height}. For $\beta>-1$, $L_n$ satisfies a
central limit theorem with fluctuations of order $(\log n)^{1/2}$, while at
$\beta=-1$ its fluctuations are of order $(\log n)^{3/2}$. For
$-2<\beta<-1$, the limit is instead non-Gaussian: $n^{\beta+1}L_n$ converges
almost surely to an explicit exponential functional of the tagged
subordinator. At $\beta=-1$, the occupied-gap representation recovers
Iksanov's construction \cite{Iksanov-2025-Regenerative}; a stronger joint
formulation is given in
\cite[Proposition 2.1]{Iksanov-Nikitin-Yakymiv-2026}.

    \subsection{Meromorphic and residue expansions}\label{sec:meromorphic}
    
    For later use, define, for every $\beta>-2$,   \begin{equation}\label{eq:eta-results}
            \eta_\beta
            :=U_\beta(\{0\})
            =\lim_{a\to\infty}\frac1{\phi_\beta(a)}
            =
            \begin{cases}
            B(\beta+2,\beta+1)^{-1},
            &\beta>-1,\\
            0,
            &-2<\beta\le-1.
            \end{cases}
    \end{equation}
    
    \subsubsection{The range \texorpdfstring{$-2<\beta<0$}{-2<beta<0}}
    
    For $-2<\beta<0$, the L\'evy density of the tagged subordinator is, for $x>0$,
    \begin{equation}\label{eq:levy-mixture-negative-results}
            \lambda_\beta(x)
            =e^{-(\beta+2)x}(1-e^{-x})^\beta
            =\sum_{m=0}^\infty
            \frac{(-\beta)_m}{m!}e^{-(m+\beta+2)x},
    \end{equation}
    where $(a)_m$ denotes the rising factorial.  The coefficients in
    \eqref{eq:levy-mixture-negative-results} are positive, so $\xi$ is a
    meromorphic subordinator in the sense of
    \cite{Kuznetsov-Kyprianou-Pardo-2012}.  Therefore the zeros of $\phi_\beta$
    are real and simple and interlace with its simple poles; write them as
    \[
            0=s_{\beta,0}>s_{\beta,1}>s_{\beta,2}>\cdots.
    \]
    We then have the following representation of the
    potential measure.
    
    \begin{theorem}[Potential measure for $-2<\beta<0$]
    \label{thm:negative-potential-results}
    For $-2<\beta<0$,
    \begin{equation}\label{eq:U-negative-results}
            U_\beta(\D x)
            =\eta_\beta\delta_0(\D x)
            +\left[
            \frac1{\mu_\beta}
            +\sum_{k\ge1}
            \frac{e^{s_{\beta,k}x}}{\phi_\beta'(s_{\beta,k})}
            \right]\D x.
    \end{equation}
    The residues $1/\phi_\beta'(s_{\beta,k})$ are positive, and the series
    converges locally uniformly for $x>0$.  Equivalently, for $\re z>0$,
    \begin{equation}\label{eq:reciprocal-negative-results}
            \frac1{\phi_\beta(z)}
            =\eta_\beta+\frac1{\mu_\beta z}
            +\sum_{k\ge1}
            \frac1{\phi_\beta'(s_{\beta,k})(z-s_{\beta,k})}.
    \end{equation}
    In particular,
    \begin{equation}\label{eq:finite-part-negative-results}
            \eta_\beta
            -\sum_{k\ge1}
            \frac1{s_{\beta,k}\phi_\beta'(s_{\beta,k})}
            =-\frac{\phi_\beta''(0)}{2\mu_\beta^2}.
    \end{equation}
    \end{theorem}
    Applying the last result in \eqref{eq:potential-identities-Dn-Lambda} we obtain a series expansion of $\Eb D_n$. Explicit error bounds are available as a by-product of the proof, see \eqref{eq:error_bound}.
    \begin{theorem}[Mean height expansion for $-2<\beta<0$]
    \label{thm:negative-height-results}
    For every $-2<\beta<0$ and every fixed $r\in\mathbb N_0$, as $n\to\infty$,
    \begin{equation}\label{eq:EDn-negative-exact-results}
    \begin{split}
            \Eb D_n
            &=
            \frac{h_{n-1}}{\mu_\beta}
            -\frac{\phi_\beta''(0)}{2\mu_\beta^2}
            -\sum_{k\ge1}
            \frac{\Gamma(-s_{\beta,k})\Gamma(n)}
            {\phi_\beta'(s_{\beta,k})\Gamma(n-s_{\beta,k})}
            \\
            &=
            \frac{h_{n-1}}{\mu_\beta}
            -\frac{\phi_\beta''(0)}{2\mu_\beta^2}
            -\sum_{k=1}^{r}
            \frac{\Gamma(-s_{\beta,k})\Gamma(n)}
            {\phi_\beta'(s_{\beta,k})\Gamma(n-s_{\beta,k})}
            +O_{\beta,r}\bigl(n^{s_{\beta,r+1}}\bigr).
    \end{split}
    \end{equation}
    \end{theorem}
    
    \begin{remark}\label{rem:harmonic-number-results}
    Keeping $h_{n-1}$ unexpanded in
    \eqref{eq:EDn-negative-exact-results} separates the residue error from the
    ordinary integer-power corrections.  If desired, one may use
    $h_{n-1}=\log n+\gamma+\bo(n^{-1})$ afterwards.
    For $\beta=-1$,
\eqref{eq:EDn-negative-exact-results} recovers, and through its exact first
equality strengthens, \cite[Theorem~7.3]{Aldous-Janson-2025-Mellin}.
    \end{remark}
    
    \begin{remark}[The endpoint $\beta=0$]\label{rem:beta-zero-results}
    At $\beta=0$,
    \[
            \phi_0(z)=\frac{z}{2(z+2)},
            \qquad
            \text{and}\qquad
            \frac1{\phi_0(z)}=2+\frac4z.
    \]
    Consequently,
    \begin{equation}\label{eq:beta-zero-results}
            U_0(\D x)=2\delta_0(\D x)+4\D x,
            \qquad
            \text{and}\qquad
            \Eb D_n=2+4h_{n-1}.
    \end{equation}
    \end{remark}

    \subsubsection{The positive-beta residue expansion}\label{subsec:MerbetaPos}
    
    Recall that $\phi_\beta$ and $\eta_\beta$ were defined in
    \eqref{eq:phi-beta-intro} and \eqref{eq:eta-results}, respectively.
    The beta-function representation
    \begin{equation}\label{eq:phi-beta-meromorphic-results}
            \phi_\beta(z)
            =
            B(\beta+2,\beta+1)
            -
            B(z+\beta+2,\beta+1)
    \end{equation}
    gives a meromorphic continuation of $\phi_\beta$ on $\Cb$.
    
    For $\beta>0$, let $\mathcal Z_\beta$ denote the set of distinct nonzero
    zeros of $\phi_\beta$. For $\rho\in\mathcal Z_\beta$, let $m_\rho$ be its
    multiplicity and write the principal part of $1/\phi_\beta$ at $\rho$ as
    \begin{equation}\label{eq:principal-part-positive-results}
            \mathcal P_{\beta,\rho}(z)
            :=
            \sum_{\ell=1}^{m_\rho}
            \frac{c_{\rho,\ell}}{(z-\rho)^\ell}.
    \end{equation}
    Define also, for $x>0$ and $n\ge2$,
    \begin{equation}\label{eq:principal-part-inverses-positive-results}
            u_{\beta,\rho}(x)
            :=
            e^{\rho x}
            \sum_{\ell=1}^{m_\rho}
            c_{\rho,\ell}\frac{x^{\ell-1}}{(\ell-1)!},
            \qquad\text{and}\qquad
            \mathcal B_{\beta,\rho}(n)
            :=
            \sum_{\ell=1}^{m_\rho}
            \frac{c_{\rho,\ell}}{(\ell-1)!}
            \frac{\partial^{\ell-1}}{\partial\rho^{\ell-1}}B(-\rho,n).
    \end{equation}
    
    \begin{theorem}[Zero structure for $\beta>0$]
    \label{thm:positive-zero-results}
    Let $\beta>0$. Every nonzero zero of $\phi_\beta$ lies in the open left
    half-plane. There are only finitely many nonreal zeros, counted with
    multiplicity, and they occur in conjugate pairs.
    
If $\beta=n_\beta$ is a positive integer, then $1/\phi_{n_\beta}$ is rational,
$\mathcal Z_{n_\beta}$ consists of exactly $n_\beta$ simple zeros, and
\[
\begin{cases}
\text{all these zeros are nonreal}, & \text{if $n_\beta$ is even},\\
-3n_\beta-4\text{ is the unique real zero}, & \text{if $n_\beta$ is odd}.
\end{cases}
\]

If $\beta$ is noninteger, let
\[
        \beta=n_\beta+\theta_\beta,
        \qquad 
        \text{with }
        n_\beta=\lfloor\beta\rfloor \text{ and }
\theta_\beta\in(0,1).
\]
If $n_\beta$ is even, then $\phi_\beta$ has exactly $n_\beta$ nonreal
zeros, counted with multiplicity. If $n_\beta$ is odd, then
$\phi_\beta$ has at most $n_\beta+1$ nonreal zeros, counted with
multiplicity. If this upper bound is attained, then all real zeros are
simple.
    
    Moreover, for every $\sigma>0$ such that
    $\phi_\beta$ has no zero on $\re z=-\sigma$, 
    \begin{equation}\label{eq:Z-strip-results}
        \text{the set }    \mathcal Z_{\beta,\sigma}
            :=
            \{\rho\in\mathcal Z_\beta:-\sigma<\re\rho<0\}
            \text{ is finite.}
    \end{equation}
    \end{theorem}
    
  For $\beta>0$, let us define
\[
        \mathcal E_\beta
        :=
        \{\rho\in\mathcal Z_\beta:
        \rho\notin\Rb\text{ or }m_\rho>1\},
        \qquad\text{and}\qquad
        \mathcal S_\beta
        :=
        \mathcal Z_\beta\setminus\mathcal E_\beta.
\]
By Proposition~\ref{appendix:zeros:integer_beta} for integer $\beta$, and Proposition~\ref{prop:complexZero} otherwise, the set
$\mathcal E_\beta$ is finite, while every element of
$\mathcal S_\beta$ is real and simple.

\begin{theorem}[Global residue representation for $\beta>0$]
\label{thm:positive-global-results}
Let $\beta>0$. Then
\begin{equation}\label{eq:u-positive-global-results}
        U_\beta(\D x)
        =
        \eta_\beta\delta_0(\D x)
        +
        \left[
        \frac1{\mu_\beta}
        +
        \sum_{\rho\in\mathcal E_\beta}
        u_{\beta,\rho}(x)
        +
        \sum_{s\in\mathcal S_\beta}
        \frac{e^{sx}}{\phi_\beta'(s)}
        \right]\D x.
\end{equation}
The last series converges locally absolutely and uniformly for $x>0$.
Equivalently, as an identity of meromorphic functions,
\begin{equation}\label{eq:reciprocal-positive-global-results}
        \frac1{\phi_\beta(z)}
        =
        \eta_\beta
        +
        \frac1{\mu_\beta z}
        +
        \sum_{\rho\in\mathcal E_\beta}
        \mathcal P_{\beta,\rho}(z)
        +
        \sum_{s\in\mathcal S_\beta}
        \frac1{\phi_\beta'(s)(z-s)},
\end{equation}
where the last series converges locally absolutely and uniformly on
$\Cb\setminus\mathcal Z_\beta$. Consequently, for every $n\ge2$,
\begin{equation}\label{eq:EDn-positive-global-results}        \Eb D_n
        =
        \frac{h_{n-1}}{\mu_\beta}
        -\frac{\phi_\beta''(0)}{2\mu_\beta^2}
        -\sum_{\rho\in\mathcal E_\beta}
        \mathcal B_{\beta,\rho}(n)
        -\sum_{s\in\mathcal S_\beta}
        \frac{\Gamma(-s)\Gamma(n)}
        {\phi_\beta'(s)\Gamma(n-s)},
\end{equation}
and the last series converges absolutely.
\end{theorem}
The proof of the last theorem is presented in Section~\ref{sec:potential-positive-beta-proof} of the appendix.

   The following formulation isolates the finitely many zeros in a prescribed
strip and gives the error bound needed for asymptotics. Put
    \begin{equation}\label{eq:F-positive-results}
            F_\beta(z):=\frac1{\phi_\beta(z)}-\eta_\beta.
    \end{equation}
    For $\sigma>0$ such that $\phi_\beta$ has no zero on $\re z=-\sigma$, and
    for $x>0$, define
    \begin{equation}\label{eq:r-positive-results}
            r_{\beta,\sigma}(x)
            :=
            \frac1{2\pi i}
            \int_{-\sigma-i\infty}^{-\sigma+i\infty}
            e^{zx}F_\beta(z)\D z.
    \end{equation}
    
    \begin{theorem}[Finite-strip expansion for $\beta>0$]
    \label{thm:positive-strip-results}
    Let $\beta>0$ and let $\sigma>0$ be as above. Then
    \[
            r_{\beta,\sigma}(x)
            =
            O_{\beta,\sigma}(e^{-\sigma x}),
            \qquad\text{as }x\to\infty,
    \]
    and
    \begin{equation}\label{eq:u-positive-strip-results}
            U_\beta(\D x)
            =
            \eta_\beta\delta_0(\D x)
            +
            \left[
            \frac1{\mu_\beta}
            +
            \sum_{\rho\in\mathcal Z_{\beta,\sigma}}
            u_{\beta,\rho}(x)
            +
            r_{\beta,\sigma}(x)
            \right]\D x.
    \end{equation}
    Consequently, as $n\to\infty$,
    \begin{equation}\label{eq:EDn-positive-strip-results}
            \Eb D_n
            =
            \frac{h_{n-1}}{\mu_\beta}
            -
            \frac{\phi_\beta''(0)}{2\mu_\beta^2}
            -
            \sum_{\rho\in\mathcal Z_{\beta,\sigma}}
            \mathcal B_{\beta,\rho}(n)
            +
            O_{\beta,\sigma}(n^{-\sigma}).
    \end{equation}
    The nonreal terms are grouped with their complex conjugates, and the displayed
    expressions are real. Thus, by moving the contour across any prescribed finite
    collection of zeros, one obtains an asymptotic expansion to arbitrary fixed
    algebraic order.
    \end{theorem}
    
    \begin{remark}[Simple and multiple zeros]\label{rem:multiple-zeros-results}
    If $\rho$ is simple, then
    \[
            c_{\rho,1}=\frac1{\phi_\beta'(\rho)},
            \qquad
            u_{\beta,\rho}(x)=\frac{e^{\rho x}}{\phi_\beta'(\rho)},
            \qquad\text{and}\qquad
            \mathcal B_{\beta,\rho}(n)
            =
            \frac{\Gamma(-\rho)\Gamma(n)}
            {\phi_\beta'(\rho)\Gamma(n-\rho)}.
    \]
    If $\rho$ has multiplicity $m$, its contribution to $\Eb D_n$ is a linear
    combination of derivatives of $B(-\rho,n)$ and is of order
    $n^{\re\rho}(\log n)^{m-1}$. The remainder remains
    $O_{\beta,\sigma}(n^{-\sigma})$.
    \end{remark}
    
    \subsection{Maximum continuous-time height}\label{sec:max-height}
    
    Our last result concerns the maximum continuous-time height, that is, the height
    of the tree itself.  Write
    \[
            D_n^*:=\max_{1\le i\le n}D_{n,i},
    \]
    where $D_{n,i}$ is the continuous-time height of leaf $i$, and define
    \begin{equation}\label{eq:r-beta-results}
            r_\beta:=\phi_\beta(1)=B(\beta+2,\beta+2).
    \end{equation}
    If $(P_i(t))_{i\ge1}$ are the ranked fragment frequencies, following \cite{BertoinRouaultDiscretization}, we define the associated
    additive martingale
    \begin{equation}\label{eq:W-beta-results}
            W_\beta(t):=e^{r_\beta t}\sum_{i\ge1}P_i(t)^2.
    \end{equation}
    
    \begin{theorem}[Height of the sampled tree]\label{thm:max-height-results}
    For every $\beta>-2$,
    \[
            W_\beta(t)\longrightarrow W_\beta(\infty),
            \qquad\text{almost surely and in }L^1,
    \]
    where $W_\beta(\infty)>0$ almost surely and
    $\Eb W_\beta(\infty)=1$.  Moreover,
    \begin{equation}\label{eq:max-height-lln-results}
            \frac{D_n^*}{\log n}
            \longrightarrow
            \frac2{r_\beta},
            \qquad
            \text{in probability},
    \end{equation}
    and, for every $x\in\Rb$,
    \begin{equation}\label{eq:max-height-fluctuations-results}
            \P\bigl(r_\beta D_n^*-2\log n\le x\bigr)
            \longrightarrow
            \Eb\left[
            \exp\left(-\frac12e^{-x}W_\beta(\infty)\right)
            \right].
    \end{equation}
    \end{theorem}
    
    \begin{remark}\label{rem:critical-max-height-results}
    At the critical value $\beta=-1$, we have $r_{-1}=1$.  Hence
    Theorem~\ref{thm:max-height-results} resolves
    \cite[Open Problem~3]{beta2-arxiv} and identifies the constant there as
    $c=2$. In particular, the proposed value
$c=1+\mu_{-1}^{-1}+\tau_{-1}^2/(2\mu_{-1}^3)\approx1.878$ in
    \cite[equation~(41)]{beta2-arxiv} does not hold, and the centered fluctuations are instead the mixed
    Gumbel law in \eqref{eq:max-height-fluctuations-results}.
    \end{remark}

    \section{Framework}\label{sec:framework}
    
    \subsection{Exchangeable fragmentations}
    \label{sec:exchangeable}
    
    Let $\mathcal P$ be the space of partitions of $\Nb$. We say that a random partition is
    \emph{exchangeable} if its law is invariant under finite permutations of
    $\Nb$. By Kingman's paintbox theorem, conditionally on its asymptotic
    frequencies, an exchangeable partition is obtained by assigning the labels
    independently to blocks with the corresponding probabilities; see
    \cite[Chapter~2.3]{Bertoin-2006}. In the conservative binary case with
    frequencies $(x,1-x)$, every label therefore chooses one of the two daughter
    blocks independently, with probabilities $x$ and $1-x$.
    
    A \emph{homogeneous exchangeable fragmentation} is a c\`adl\`ag Markov process
    $\Pi=(\Pi(t),t\ge0)$ on $\mathcal P$, started from the one-block partition,
    whose law is exchangeable and which has the fragmentation property:
    conditionally on $\Pi(t)$, the different blocks evolve independently, each
    according to the same law. The term \emph{exchangeable} refers to the labels,
    and \emph{homogeneous} to the time scale not depending on the
    asymptotic frequency of a block; see
    \cite[Chapter~3]{Bertoin-2006}.
    
    A fragmentation without erosion is determined by its \textit{dislocation measure},
    which describes the intensity and the relative sizes of the fragments created at
    a split. If the dislocation measure $\nu$ is finite, every block waits an
    exponential time of rate $\nu([1/2,1])$ and the normalized measure
    $\nu/\nu([1/2,1))$ is the law of its relative fragment sizes. Infinite
    dislocation measures are also allowed. In this case the process is constructed
    from a Poisson point measure with intensity $\D t\otimes\nu$, and the usual
    integrability condition guarantees that every finite restriction has a finite
    rate of splits; see \cite[Section 3.1-2]{Bertoin-2006}.
    
    For the beta-fragmentation, we consider conservative binary fragmentations without erosion, that is, each fragment splits in exactly two parts and the overall mass is preserved. The
    dislocation measure is
    \[
            \nu_\beta(\D x)
            =
            x^\beta(1-x)^\beta\D x,
            \qquad \text{on }x\in[1/2,1),
    \]
    and the integrability condition needed to define a fragmentation process,
    see \cite[Chapter~3]{Bertoin-2006}, translates to
    \[
            \int_{1/2}^1(1-x)\nu_\beta(\D x)<\infty,
            \qquad
            \text{so $\beta>-2$}.
    \]
    At a dislocation with frequencies $(x, 1-x)$,
    the labels in a block are assigned independently to the two daughters.
    Consequently, a block containing $m$ labels splits into ordered daughter
    blocks of sizes $i$ and $m-i$ at rate
    \begin{equation}\label{eq:finite-split-rate-prelim}
            \widehat q_\beta(m,i)
            :=
            \frac12\binom mi
            B(i+\beta+1,m-i+\beta+1),
            \qquad \text{for }1\le i\le m-1.
    \end{equation}
    Moreover,
    \[
            \sum_{i=1}^{m-1}\widehat q_\beta(m,i)
            =
            \phi_\beta(m-1),
    \]
    and therefore, conditionally on a split,
    \begin{equation}\label{eq:q-beta-prelim}
            q_\beta(m,i)
            =
            \frac{\widehat q_\beta(m,i)}
            {\phi_\beta(m-1)}.
    \end{equation}
    Thus the restriction to $[n]$ is the canonical continuous-time beta-splitting
    process. This is the standard finite-restriction construction for homogeneous
    fragmentations as also described, for example, in
    \cite[Corollary~4]{Haas-Miermont-Pitman-Winkel-2008}.
    
    \subsubsection{The tagged fragment}\label{sec:tagged-preliminaries}
    
    An especially convenient object for analysis is the so-called
    \emph{tagged}, or \emph{typical}, fragment. Let $\Pi_1(t)$ be the block
    containing label $1$ and let $|\Pi_1(t)|$ denote its asymptotic frequency. Recall that by exchangeability, considering label $1$ is equivalent to considering any other fixed label.
    Then \cite[Theorem~3.2]{Bertoin-2006} gives
    \begin{equation}\label{eq:tagged-subordinator-prelim}
            |\Pi_1(t)|=e^{-\xi_t},
            \qquad\text{and}\qquad
            \Eb e^{-z\xi_t}=e^{-t\phi_\beta(z)}
            \quad\text{for }\Re z\ge0,
    \end{equation}
    for the subordinator $\xi$ with Laplace exponent $\phi_\beta$, defined in \eqref{eq:phi-beta-intro}. In the
    beta-fragmentation case, the L\'evy density of $\xi$ is
    \begin{equation}\label{eq:levy-density-prelim}
            \lambda_\beta(x)
            =
            e^{-(\beta+2)x}(1-e^{-x})^\beta,
            \qquad \text{for }x>0.
    \end{equation}
    Let $(P_i(t))_{i\ge1}$ be the ranked fragment frequencies. Recall the Kingman paintbox construction: the associated partition of $[n]$ at time $t$ can be constructed by dropping $n$ iid uniform random variables on $(0,1)$ and grouping the indices which fall in intervals of lengths $(P_i(t))$. Therefore symmetric quantities involving all fragments, for example the expected number of pairs in the same block, can be represented via $(P_i(t))$. They can also be connected with the \textit{typical} block and thus expressed via $\xi$. Indeed, by \cite[Proposition~2.8]{Bertoin-2006}, for every nonnegative
    Borel function $f$,
    \begin{equation}\label{eq:size-biased-identity-prelim}
            \Eb\left[\sum_{i\ge1}P_i(t)f(P_i(t))\right]
            =
            \Eb f\bigl(|\Pi_1(t)|\bigr)
            =
            \Eb f(e^{-\xi_t}).
    \end{equation}
    Equivalently, for every
    nonnegative $g$ with $g(0)=0$,
    \[
            \Eb\sum_{i\ge1}g(P_i(t))
            =
            \Eb\left[e^{\xi_t}g(e^{-\xi_t})\right].
    \]
    For example, taking $f(y)=y^{q-1}$ gives, for $q>1$,
    \begin{equation}\label{eq:power-sum-mean-prelim}
            \Eb\sum_{i\ge1}P_i(t)^q
            =
            \Eb|\Pi_1(t)|^{q-1}
            =
            e^{-t\phi_\beta(q-1)}.
    \end{equation}
    The last identity is closelyt connected with the additive martingales of Section~\ref{sec:max-height-proof}.
    
    \subsection{Meromorphic subordinators}
    \label{sec:meromorphic-subordinator-preliminaries}
    
    Following \cite{Kuznetsov-Kyprianou-Pardo-2012}, a subordinator is called
    \emph{meromorphic} if its L\'evy density is a positive discrete mixture of
    exponentials, that is,
    \begin{equation}\label{eq:meromorphic-levy-density-prelim}
            \lambda(x)
            =
            \sum_{k\ge1}c_ke^{-\rho_kx},
            \qquad \text{for some }
            c_k>0,
            \text{ and }
            0<\rho_1<\rho_2<\cdots,
    \end{equation}
    under the usual L\'evy integrability condition. Using the classical L\'evy--Khintchine representation, its Laplace exponent can then
    be written as
    \begin{equation}\label{eq:meromorphic-phi-prelim}
    \begin{split}
            \phi(z)
            =dz+\int_0^\infty(1-e^{-zx})\lambda(x)\D x
            =d z+
            \sum_{k\ge1}
            c_k\frac{z}{\rho_k(\rho_k+z)},
            \qquad\text{for } d\ge0,
    \end{split}
    \end{equation}
    which gives a meromorphic continuation to $\Cb$, with poles at $-\rho_k$, which also motivates the name of the class.
    Moreover, $\phi$ is a complete Bernstein function, and its nonzero zeros are
    real, simple and interlace with its poles; see
    \cite[Theorem~1 and Corollary~2]{Kuznetsov-Kyprianou-Pardo-2012}. This mainly stems from the fact that a complete Bernstein function is a Nevanlinna--Pick function and therefore preserves the upper and lower complex half-planes; see
    \cite[Theorem~6.2]{Schilling-Song-Vondracek-2026}. The general Nevanlinna--Pick class and its integral representation are discussed in Chapter~6 of the last reference.
    
    The importance of this class for the beta-fragmentation follows from
    \eqref{eq:levy-mixture-negative-results}: for $-2<\beta<0$, the L\'evy
    density $\lambda_\beta$ is a positive mixture of exponentials. The tagged
    subordinator is therefore meromorphic.
    
    The product and partial-fraction formulas in
    \cite{Kuznetsov-Kyprianou-Pardo-2012} are stated mainly in terms of the roots
    of $q-\Psi$ for $q>0$, where $\Psi$ is the Laplace exponent of the process. The result needed in our case, that is, $q=0$, for the potential measure of an
    unkilled subordinator is not explicitly stated there.
    We provide it below and give the proof in
    Appendix~\ref{app:meromorphic-potential-proof}.
    
    \begin{proposition}[Potential measure of a meromorphic subordinator]
    \label{prop:U-expression}
    Let $\xi$ be an unkilled meromorphic subordinator with Laplace exponent
    $\phi$ and finite mean $\mu=\phi'(0)$. Let $(-\rho_k)$ and
    $(s_k)$ denote its nonzero poles and zeros, respectively, and put $s_0=0$.
    The sequences may be finite or infinite and are indexed so that
    \[
            0>-\rho_1>s_1>-\rho_2>s_2>\cdots.
    \]
    The displayed inequalities are truncated if either sequence terminates.
    Then, for $\Re z>0$,
    \begin{equation}\label{eq:meromorphic-reciprocal-prelim}
            \frac1{\phi(z)}
            =
            \eta+
            \sum_{k\ge0}
            \frac1{\phi'(s_k)(z-s_k)},
            \qquad
            \text{where }
            \eta:=\lim_{z\to+\infty}\frac1{\phi(z)},
    \end{equation}
    and the potential measure of $\xi$ is
    \begin{equation}\label{eq:meromorphic-potential-prelim}
            U(\D x)
            =
            \eta\delta_0(\D x)
            +
            \left[
            \sum_{k\ge0}
            \frac{e^{s_kx}}{\phi'(s_k)}
            \right]\D x.
    \end{equation}
    Moreover, the residues $1/\phi'(s_k)$ are positive and the series converge
    locally uniformly in their respective domains.
    \end{proposition}
    
    \begin{remark}
    The last proposition provides a useful way of identifying the potential measure without carrying out an infinite contour inversion, which is a major technical step in \cite{Aldous-Janson-2025-Mellin}. Formally, since $1/\phi$ is the Laplace transform of $U$, one may try to apply a Bromwich-type inversion formula and move the contour across the poles of $1/\phi$, collecting their residues. This is the classical procedure described in \cite[Chapters~25-26]{Doetsch1974}, and an analogous residue method is used in \cite[Sections~6-7]{Aldous-Janson-2025-Mellin}. For an infinite collection of poles, however, one still has to justify the contour deformation, the vanishing of the outer contour and the interchange with the infinite residue sum. The point of Proposition~\ref{prop:U-expression} is that $z/\phi(z)$ is a meromorphic complete Bernstein, and hence a Nevanlinna--Pick, function for which we know that there exists a unique positive partial-fraction representation. The proof in Appendix~\ref{app:meromorphic-potential-proof}, based on the meromorphic Nevanlinna representation of \cite[Theorem~1, p.~197]{chebotarev1949routh}, identifies the candidate measure first and then uses the uniqueness of Laplace transforms. In this way, no  explicit contour arguments are needed.
    \end{remark}
    
    \subsubsection{Extended meromorphic subordinators}
    \label{sec:nevanlinna-preliminaries}
    
    For $\beta>0$, the L\'evy density in \eqref{eq:levy-density-prelim} is still
    positive, but its expansion in powers of $e^{-x}$ is no longer a positive
    mixture. Put
    \[
            a_{\beta,m}:=m+\beta+2,
            \qquad\text{and}\qquad
            c_{\beta,m}:=\frac{(-\beta)_m}{m!}.
    \]
    The binomial series gives, locally uniformly for $x>0$,
    \begin{equation}\label{eq:signed-levy-expansion-prelim}
            \lambda_\beta(x)
            =
            \sum_{m=0}^\infty
            c_{\beta,m}e^{-a_{\beta,m}x},
    \end{equation}
    and hence
    \begin{equation}\label{eq:signed-phi-expansion-prelim}
            \phi_\beta(z)
            =
            \sum_{m=0}^\infty
            c_{\beta,m}
            \frac{z}{a_{\beta,m}(a_{\beta,m}+z)},
    \end{equation}
    but the signs of $c_{\beta,m}$ are not all the same.
    For example,
    \[
            \lambda_{3/2}(x)
            =
            e^{-7x/2}
            \left(
            1-\frac32e^{-x}
            +\frac38e^{-2x}
            +\frac1{16}e^{-3x}
            +\cdots
            \right).
    \]
    This illustrates that, although the signs are not the same, they do not alternate indefinitely. More generally, if
    $\beta>0$ is not an integer and
    \[
            n_\beta:=\lfloor\beta\rfloor,
            \qquad\text{and}
            \qquad
            \varepsilon_\beta:=(-1)^{n_\beta+1},
    \]
    then
    \begin{equation}\label{eq:eventual-sign-prelim}
            \varepsilon_\beta c_{\beta,m}>0,
            \qquad\text{for } m\ge n_\beta+1.
    \end{equation}
    In this paper, we
    refer to this finite change of sign situation as the \emph{extended meromorphic}
    case.
    We also note that if $\beta\in\Nb_0$, the series has a finite number of terms, and $\phi_\beta$ is rational. 
    
    In this extended setting, the property of preserving the upper and lower complex half-planes is lost: we no longer fall in the classical Nevanlinna--Pick class which forced the poles and zeros to be real and interlacing.
    An appropriate replacement is the finite-index theory of generalised Nevanlinna functions, for which we follow
    \cite{KreinLanger1977,DahoLanger1985,RovnyakSakhnovich2004}: let $Q$ be meromorphic on $\Cb\setminus\Rb$ and satisfy
    $Q(\bar z)=\overline{Q(z)}$. Define the kernel
    \[
            N_Q(z,w)
            :=
            \frac{Q(z)-\overline{Q(w)}}{z-\bar w}.
    \]
    We say that $Q$ belongs to the \textit{Nevanlinna class of index $\kappa$}, denoted by $\mathcal N_\kappa$, if the kernel $N_Q$ has $\kappa$ negative squares. This means that, for every finite choice $z_1,\ldots,z_n$ in the domain of holomorphy of $Q$, the Hermitian quadratic form
    \[
            \sum_{i,j=1}^n N_Q(z_i,z_j)c_i\overline{c_j}
    \]
    has at most $\kappa$ negative squares, or equivalently its representing matrix has at most $\kappa$ negative eigenvalues, and that equality is attained for some choice of points. The terminology comes from the fact that, after a change of coordinates, such a quadratic form can be written as
    \[
            |u_1|^2+\cdots+|u_p|^2-|v_1|^2-\cdots-|v_r|^2,
            \qquad\text{with }r\le\kappa,
    \]
    possibly together with zero terms. Thus $\kappa$ is the maximal number of independent negative directions of the kernel. We write
    \[
            \mathcal N_{<\infty}:=\bigcup_{\kappa\ge0}\mathcal N_\kappa.
    \]
    The class $\mathcal N_0$ is the usual Nevanlinna class.
    
    The Krein--Langer factorisation, in the form of
    \cite[Corollary]{DijksmaLangerLugerShondin2000}, shows that, after removing finitely many exceptional poles and zeros, a function in $\mathcal N_{<\infty}$ reduces to a classical Nevanlinna function. In particular, a function in $\mathcal N_{<\infty}$ which is meromorphic on $\Cb$ has only finitely many nonreal poles, and all but finitely many of its real poles are simple; see also
    \cite[Theorem~2.1]{RovnyakSakhnovich2004}. Since
    \[
            N_{-1/Q}(z,w)
            =
            \frac{N_Q(z,w)}
            {Q(z)\overline{Q(w)}},
    \]
    the transformation $Q\mapsto-1/Q$ preserves the number of negative squares. The corresponding statements therefore hold also for the zeros.
    
    The finite change of sign expansion \eqref{eq:signed-phi-expansion-prelim} places the
    positive-beta case in this finite-index class. Indeed, for noninteger
    $\beta>0$, define
    \begin{equation}\label{eq:q-generalized-nevanlinna-prelim}
    \begin{split}
            \mathfrak q_\beta(z)
            :=
            -\varepsilon_\beta\frac{\phi_\beta(z)}{z}
            =-\sum_{m=0}^{n_\beta}
            \frac{\varepsilon_\beta c_{\beta,m}}
            {a_{\beta,m}(a_{\beta,m}+z)} -
            \sum_{m\ge n_\beta+1}
            \frac{\varepsilon_\beta c_{\beta,m}}
            {a_{\beta,m}(a_{\beta,m}+z)}.
    \end{split}
    \end{equation}
    By \eqref{eq:eventual-sign-prelim}, the second sum is a classical Nevanlinna function. The
    first sum is rational and its kernel has finite rank. Indeed, for $a,d\in\mathbb R$, the kernel corresponding to
    $d/(a+z)$ is
    \[
            -\frac{d}{(a+z)(a+\overline w)},
    \]
    and hence has rank one. Thus, by the subadditivity of rank,
    the kernel of the complete finite sum has rank at most
    $n_\beta+1$; see \cite[Theorem~3.8(b), equation~(3.41)]{Tian2010}. Adding a Hermitian kernel of rank at most $n_\beta+1$ to a
    positive kernel can create at most $n_\beta+1$ negative squares;
    see \cite[Theorem~3.8(b), equation~(3.39)]{Tian2010}. Applying this to every Gram matrix of
    the kernel, we conclude that it has at most $n_\beta+1$ negative
    squares. Consequently,
    \[
            \mathfrak q_\beta\in\mathcal N_{<\infty}.
    \]
    Moreover,
    \[
            N_{-1/\mathfrak q_\beta}(z,w)
            =
            \frac{N_{\mathfrak q_\beta}(z,w)}
            {\mathfrak q_\beta(z)
            \overline{\mathfrak q_\beta(w)}},
            \qquad\text{so}\quad        \varepsilon_\beta\frac{z}{\phi_\beta(z)}
            =
            -\frac1{\mathcal q_\beta(z)}
            \in\mathcal N_{<\infty}.\]
    For integer $\beta>0$, the same conclusion follows directly because
    $\phi_\beta$ is rational and the associated kernels have finite rank.
    
    It follows that $\phi_\beta$ has only finitely many nonreal zeros, counted
    with multiplicity, and that they occur in conjugate pairs. The specific form of \(\phi_\beta\) allows for more detailed information about these zeros. 
    When $n_\beta$ is even then there are precisely \(n_\beta\) complex zeros whereas if \(n_\beta\) is odd, then the complex zeroes, counted with multiplicity, are at most \(n_\beta+1\), see Proposition \ref{prop:complexZero}.
    We conjecture that in the latter case all zeros are complex but despite significant efforts we have not been able to establish this fact.

    \section{Probabilistic arguments}
    \label{sec:probabilistic-arguments}

    Before delving into the more analytic arguments which provide the series expansions for $U_\beta$ and $\Eb D_n$, we first prove the results which require mainly probabilistic arguments, namely, Theorem~\ref{thm:potential-identities-results}, Corollary~\ref{cor:Dn-renewal-results}, and Theorem~\ref{thm:max-height-results}.
    
    \subsection{The tagged fragment and the potential identities}
    
    As noted, Theorem~\ref{thm:potential-identities-results} is essentially the general-beta counterpart of \cite[Proposition~4.1]{Aldous-Janson-2025-Mellin} and the proof is in essence the same. However, it gives good intuition for the model and illustrates why the potential measure is the common object behind different model statistics, so we include a condensed proof.
    
    \begin{proof}[Proof of Theorem~\ref{thm:potential-identities-results}]
    Let
    \[
            Z_n(t):=\#\bigl(\Pi_1(t)\cap[n]\bigr)
    \]
    be the size of the clade containing the tagged label at time $t$. Conditionally on the ranked fragment frequencies $\mathbf P(t)=(P_i(t))_{i\ge1}$, the tagged label falls in the $i$th fragment with probability $P_i(t)$. Given this choice, each of the remaining $n-1$ labels falls in the same fragment independently with probability $P_i(t)$. Therefore
    \[
            \P\bigl(Z_n(t)=k\mid\mathbf P(t)\bigr)
            =
            \sum_{i\ge1}P_i(t)\binom{n-1}{k-1}P_i(t)^{k-1}(1-P_i(t))^{n-k}.
    \]
    Averaging and using the size-biased identity \eqref{eq:size-biased-identity-prelim}, we obtain
    \[
            \P\bigl(Z_n(t)=k\bigr)
            =
            \Eb\left[\binom{n-1}{k-1}|\Pi_1(t)|^{k-1}(1-|\Pi_1(t)|)^{n-k}\right]
            =
            \Eb\left[\binom{n-1}{k-1}e^{-(k-1)\xi_t}(1-e^{-\xi_t})^{n-k}\right].
    \]
    Equivalently,
    \begin{equation}\label{eq:tagged-binomial-proof}
            Z_n(t)\stackrel{d}{=}1+\operatorname{Bin}\bigl(n-1,e^{-\xi_t}\bigr),
            \qquad\text{conditionally on $\xi_t$.}
    \end{equation}
    Therefore, recalling that
    \[
            F_n(x)=1-(1-e^{-x})^{n-1},
    \]
    we have
    \[
            \P(D_n>t)=\P(Z_n(t)\ge2)=\Eb F_n(\xi_t).
    \]
    Integrating this identity over $t$ and using Tonelli's theorem gives
    \[
            \Eb D_n=\Eb\left[\int_0^\infty F_n(\xi_t)\D t\right]=\mathcal U_\beta F_n.
    \]
    Note the similarity of $F_n$ with the tail distribution of the maximum of independent exponential random variables. Indeed, if $M_{n-1}$ is the maximum of $n-1$ independent exponential random variables of rate $1$, independent of $\xi$, then
    \[
            F_n(x)=\P(M_{n-1}>x).
    \]
    Consequently,
    \[
            \P(T_{M_{n-1}}>t\mid\xi)=\P(M_{n-1}>\xi_t\mid\xi)=1-(1-e^{-\xi_t})^{n-1}.
    \]
    Averaging the last identity and comparing it with the tail of $D_n$ gives
    \[
            \P(T_{M_{n-1}}>t)=\P(D_n>t),
            \qquad\text{so}\qquad
            D_n\stackrel{d}=T_{M_{n-1}}.
    \]
    The passage-time representation is convenient because $T_x$ is the inverse of the subordinator and is a basic object in renewal theory.
    
    We are left with the expected discrete height and the expected total continuous length. For the discrete height, the strategy is to relate the probability of visiting a given clade size to the amount of continuous time spent at that size. For $2\le j\le n$, let
    \[
            O_{n,j}:=\int_0^\infty\ind{Z_n(t)=j}\D t
    \]
    be the occupation time of the tagged clade at size $j$. By \eqref{eq:tagged-binomial-proof} and Tonelli's theorem,
    \[
            \Eb O_{n,j}
            =
            \binom{n-1}{j-1}\Eb\left[\int_0^\infty e^{-(j-1)\xi_t}(1-e^{-\xi_t})^{n-j}\D t\right]
            =
            \binom{n-1}{j-1}\mathcal U_\beta G_{n,j}.
    \]
    The tagged clade can visit size $j$ at most once, since its size decreases strictly at every split. If size $j$ is visited, the clade remains there for an exponential time of mean $1/\phi_\beta(j-1)$. Therefore
    \[
            \Eb O_{n,j}=\frac{a_\beta(n,j)}{\phi_\beta(j-1)},
    \qquad\text{and hence}
\qquad
            a_\beta(n,j)=\phi_\beta(j-1)\binom{n-1}{j-1}\mathcal U_\beta G_{n,j}.
    \]
    The discrete height of the tagged leaf is exactly the number of splits along its ancestral line, or equivalently the number of non-singleton clade sizes visited before absorption. Thus
    \[
            L_n=\sum_{j=2}^n\ind{\text{the tagged clade visits size }j},
            \qquad\text{and therefore}\qquad
            \Eb L_n=\sum_{j=2}^n a_\beta(n,j).
    \]
    By exchangeability, label $1$ has the same distribution as a uniformly chosen label. Hence $\Eb L_n$ is also the expected average discrete height of the leaves of the tree.
    
    Finally, consider the total continuous-time length. At time $t$, every fragment which contains at least two of the sampled labels corresponds to an edge which is still growing and therefore contributes one unit to the instantaneous increase of the total length. Singleton fragments no longer contribute. If $N_i^{(n)}(t)$ denotes the number of sampled labels in the $i$th fragment at time $t$, then
    \[
            \Lambda_n=\int_0^\infty\sum_{i\ge1}\ind{N_i^{(n)}(t)\ge2}\D t.
    \]
    Conditionally on the fragment frequencies $\mathbf P(t)=(P_i(t))_{i\ge1}$, the marginal distribution of $N_i^{(n)}(t)$ is $\operatorname{Bin}(n,P_i(t))$. Although these variables are not independent for different $i$, only their marginal probabilities are needed, and therefore
    \[
            \Eb\left[\sum_{i\ge1}\ind{N_i^{(n)}(t)\ge2}\middle|\mathbf P(t)\right]
            =
            \sum_{i\ge1}\left[1-(1-P_i(t))^n-nP_i(t)(1-P_i(t))^{n-1}\right].
    \]
    Put
    \[
            g(y):=1-(1-y)^n-ny(1-y)^{n-1}.
    \]
    The tagged fragment is a size-biased fragment, so \eqref{eq:size-biased-identity-prelim} gives
    \[
            \Eb\sum_{i\ge1}g(P_i(t))=\Eb\left[e^{\xi_t}g(e^{-\xi_t})\right].
    \]
    By the definition of $H_n$,
    \[
            e^xg(e^{-x})
            =
            e^x\left[1-(1-e^{-x})^n\right]-n(1-e^{-x})^{n-1}
            =
            H_n(x).
    \]
    Consequently,
    \[
            \Eb\sum_{i\ge1}\ind{N_i^{(n)}(t)\ge2}=\Eb H_n(\xi_t),
    \]
    and again by integrating over $t$ and using Tonelli's theorem gives
    \[
            \Eb\Lambda_n
            =
            \Eb\left[\int_0^\infty H_n(\xi_t)\D t\right]
            =
            \mathcal U_\beta H_n.
    \]
    \end{proof}
    
    \subsection{Renewal consequences}
    \label{sec:typical-asymptotics}
    \begin{proof}[Proof of Corollary~\ref{cor:Dn-renewal-results}]
    Recall the passage time
    \[
            T_x:=\inf\{t\ge0:\xi_t\ge x\},
            \qquad
            \text{and let}\qquad
            U_\beta(x):=U_\beta([0,x)).
    \]
    Since $\{T_x>t\}=\{\xi_t<x\}$, Tonelli's theorem gives
    \[
            \Eb T_x
            =
            \int_0^\infty\P(T_x>t)\D t
            =
            \int_0^\infty\P(\xi_t<x)\D t
            =
            U_\beta(x).
    \]
    
    The Laplace transform of the renewal measure is
    \[
            \int_{[0,\infty)}e^{-\lambda x}U_\beta(\D x)
            =
            \frac1{\phi_\beta(\lambda)}
            \sim
            \frac1{\mu_\beta \lambda},
            \qquad\text{as }\lambda\downarrow0.
    \]
    Therefore Karamata's Tauberian theorem, see
    \cite[Proposition~1.5]{Bertoin-1999-Subordinators}, gives
    \begin{equation}\label{eq:inverse-mean-proof}
            U_\beta(x)\sim\frac{x}{\mu_\beta},
            \qquad\text{as }x\to\infty.
    \end{equation}
    The same Tauberian argument applies in the compound Poisson case. In particular,  there exists $C_\beta>0$ such that
    \begin{equation}\label{eq:renewal-linear-bound-proof}
            U_\beta(x)\le C_\beta(1+x),
            \qquad\text{for }x\ge0.
    \end{equation}
    We next recall the corresponding central limit theorem for the passage time. The usual central limit theorem for the L\'evy process gives
    \[
            \frac{\xi_t-\mu_\beta t}{\tau_\beta\sqrt t}
            \xrightarrow{d}\mathcal N(0,1),
            \qquad\text{as }t\to\infty,
    \]
    and a standard inversion argument, as in
    \cite[Section~10.2, Theorem~(2)]{GrimmettStirzaker2001}, gives the central limit theorem for $T_x$. Indeed, put
    \[
            b_x:=\sqrt{\tau_\beta^2\mu_\beta^{-3}x},
            \qquad\text{and}\qquad
            t_x(y):=\mu_\beta^{-1}x+yb_x,
    \]
    and note that, for every fixed $y\in\Rb$,
    \[
            \P\left(\frac{T_x-\mu_\beta^{-1}x}{b_x}\le y\right)
            =
            \P\bigl(\xi_{t_x(y)}\ge x\bigr),
            \qquad\text{and}\qquad
            \frac{x-\mu_\beta t_x(y)}{\tau_\beta\sqrt{t_x(y)}}\longrightarrow-y.
    \]
    Consequently,
    \begin{equation}\label{eq:inverse-clt-proof}
            \frac{T_x-\mu_\beta^{-1}x}
            {\sqrt{\tau_\beta^2\mu_\beta^{-3}x}}
            \xrightarrow{d}\mathcal N(0,1),
            \qquad\text{as }x\to\infty.
    \end{equation}
    
    It remains to evaluate the passage time at the random level $M_{n-1}$. Its exact distribution is
    \[
            \P(M_{n-1}\le x)=(1-e^{-x})^{n-1},
            \qquad\text{for }x\ge0,
    \]
    and standard exponential order-statistics identities,
    e.g. \cite[Section 4.14, Exercise 24]{GrimmettStirzaker2001}, give
    \[
            \Eb M_{n-1}=h_{n-1},
            \qquad\text{and}\qquad
            \Var(M_{n-1})=\sum_{k=1}^{n-1}\frac1{k^2}.
    \]
    Moreover,
    \[
            \P(M_{n-1}-\log n\le x)
            =
            \left(1-\frac{e^{-x}}n\right)^{n-1}
            \longrightarrow
            \exp\left(-e^{-x}\right),
        \qquad \text{as $n\longrightarrow\infty$}.
    \]
    It follows that
    \begin{equation}\label{eq:maximum-exponential-proof}
            M_{n-1}-\log n\xrightarrow{d}\Gumbel,
            \qquad\text{and}\qquad
            \frac{M_{n-1}}{\log n}\longrightarrow1,
            \quad\text{in }L^2.
    \end{equation}
    Using $D_n\stackrel{d}=T_{M_{n-1}}$ and the independence of $M_{n-1}$ and $\xi$, we obtain
    \[
            \Eb D_n
            =
            \Eb\left[\Eb\left(T_{M_{n-1}}\mid M_{n-1}\right)\right]
            =
            \Eb U_\beta(M_{n-1}).
    \]
    Relations \eqref{eq:inverse-mean-proof} and \eqref{eq:maximum-exponential-proof} imply
    \[
            \frac{U_\beta(M_{n-1})}{\log n}\longrightarrow\frac1{\mu_\beta},
            \qquad\text{in probability}.
    \]
    Furthermore, by \eqref{eq:renewal-linear-bound-proof},
    \[
            0\le\frac{U_\beta(M_{n-1})}{\log n}
            \le
            C_\beta\frac{1+M_{n-1}}{\log n},
    \]
    and the right-hand side is bounded in $L^2$. Therefore 
        $U_\beta(M_{n-1})/\log n$ is bounded in $L^2$ and thus uniformly integrable by the standard $L^p$-criterion, so
    \[
            \Eb D_n=\Eb U_\beta(M_{n-1})\sim\frac{\log n}{\mu_\beta}.
    \]
    Finally, since $M_{n-1}-\log n$ is tight, choose any sequence $K_n$ such that
    \[
            K_n\longrightarrow\infty,
            \qquad\text{and}\qquad
            K_n=\so(\sqrt{\log n}),
    \]
    as for example, $K_n=\log\log n$. Then
    \[
            \P\bigl(|M_{n-1}-\log n|\le K_n\bigr)\longrightarrow1,
    \]
    and, on this event, the monotonicity of $x\mapsto T_x$ gives
    \[
            T_{\log n-K_n}\le T_{M_{n-1}}\le T_{\log n+K_n}.
    \]
    Applying \eqref{eq:inverse-clt-proof} at the two deterministic levels
    $\log n-K_n$ and $\log n+K_n$, and using
    $K_n=\so(\sqrt{\log n})$, shows that both endpoint terms, centred by
    $\mu_\beta^{-1}\log n$ and divided by
    $\sqrt{\tau_\beta^2\mu_\beta^{-3}\log n}$, converge to the same standard
    normal distribution. The middle term is therefore squeezed between two
    terms with the same limit, outside an event whose probability tends to zero.
    Using \eqref{eq:Dn-passage-results}, we conclude that
    \[
            \frac{D_n-\mu_\beta^{-1}\log n}
            {\sqrt{\tau_\beta^2\mu_\beta^{-3}\log n}}
            \xrightarrow{d}\mathcal N(0,1).
    \]
    \end{proof}

    \subsection{The maximum continuous-time height}
    \label{sec:max-height-proof}
    
    We now turn to the global height of the continuous tree, whose behaviour was presented in Theorem~\ref{thm:max-height-results}. Translating the problem to the paintbox construction, we are interested in the first time when no two sampled labels fall in the same fragment. We first identify its first-order behaviour by the classical first and second moment method.
    
    For $q>1$, write
    \[
            S_q(t):=\sum_{i\ge1}P_i(t)^q.
    \]
    We keep the notation for general $q$, since $S_{p+1}$ will also appear below. Conditionally on $\mathbf P(t)$, let
    \[
            N_i(t):=\#\{1\le a\le n:a\text{ falls in fragment }i\}
    \]
    and define the number of pairwise collisions by
    \begin{equation}\label{eq:collision-count-proof}
            C_n(t):=\sum_{i\ge1}\binom{N_i(t)}2=\sum_{1\le a<b\le n}I_{ab}(t),
    \end{equation}
    where $I_{ab}(t)$ indicates that labels $a$ and $b$ lie in the same fragment. Then
    \begin{equation}\label{eq:Dstar-collision-proof}
            D_n^*\le t
            \qquad\text{if and only if}\qquad
            C_n(t)=0.
    \end{equation}
    Moreover, by linearity of expectation,
    \begin{equation}\label{eq:conditional-mu-proof}
            \mu_n(t):=\Eb\bigl(C_n(t)\mid\mathbf P(t)\bigr)=\binom n2S_2(t),
    \end{equation}
    so the martingale $W_\beta$ defined in \eqref{eq:W-beta-results} drives the behaviour of the collision count.
    
    \subsubsection{The quadratic additive martingale}
    
    \begin{lemma}[The quadratic martingale]\label{lem:quadratic-martingale-proof}
    For every $\beta>-2$,
    \begin{equation}\label{eq:W-beta-convergence-proof}
            W_\beta(t)=e^{r_\beta t}S_2(t)\longrightarrow W_\beta(\infty),
            \qquad\text{almost surely and in }L^1,
    \end{equation}
    where $W_\beta(\infty)>0$ almost surely and $\Eb W_\beta(\infty)=1$.
    \end{lemma}
    
    \begin{proof}[Proof of Lemma~\ref{lem:quadratic-martingale-proof}]
    The convergence of the martingale, together with
    $\Eb W_\beta(\infty)=1$, follows from
    \cite[Proposition~1 and Theorem~4]{BertoinRouaultDiscretization}
    once we check condition~(4) therein, namely
    $
            2\phi_\beta'(1)-\phi_\beta(1)>0.
    $
    This follows because from representation \eqref{eq:phi-beta-intro},
    \[
            2\phi_\beta'(1)-\phi_\beta(1)
            =
            \int_0^1
            \left[
            -s^2\log s-(1-s)^2\log(1-s)-s(1-s)
            \right]
            s^\beta(1-s)^\beta\D s,
    \]
    and by standard calculus we can verify that the integrand is strictly positive.
    
    It remains to note that the limit is strictly positive. By the fragmentation
    property, for every $t>0$,
    \[
            W_\beta(\infty)
            \stackrel{d}{=}
            e^{r_\beta t}
            \sum_{i\ge1}P_i(t)^2W_{\beta,i}(\infty),
    \]
    where, conditionally on $\mathbf P(t)$, the variables
    $W_{\beta,i}(\infty)$ are independent copies of
    $W_\beta(\infty)$. If
    \[
            q:=\P(W_\beta(\infty)=0)
    \]
    and $N(t)$ is the number of positive-mass fragments at time $t$, then
    \[
            q=\Eb\bigl(q^{N(t)}\bigr).
    \]
    Since the fragmentation is conservative, $N(t)\ge1$ almost surely, and since
    it is nontrivial, $\P(N(t)\ge2)>0$ for some $t>0$. Hence, for $0<q<1$,
    \[
            \Eb\bigl(q^{N(t)}\bigr)
            \le
            q\P(N(t)=1)+q^2\P(N(t)\ge2)
            <q,
    \]
    which is impossible. Therefore $q\in\{0,1\}$, and
    $\Eb W_\beta(\infty)=1$ excludes $q=1$, so $q=0$,
    that is, $W_\beta(\infty)$ is strictly positive almost surely as needed.
    \end{proof}
    
    \subsubsection{First and second moments}
    
    \begin{lemma}[Conditional first and second moments]
    \label{lem:collision-moments-proof}
    Conditionally on $\mathbf P(t)$,
    \begin{equation}\label{eq:collision-variance-proof}
            \Var\bigl(C_n(t)\mid\mathbf P(t)\bigr)=\binom n2S_2(t)\bigl(1-S_2(t)\bigr)+6\binom n3\bigl(S_3(t)-S_2(t)^2\bigr).
    \end{equation}
    In particular,
    \begin{equation}\label{eq:collision-variance-bound-proof}
            \Var\bigl(C_n(t)\mid\mathbf P(t)\bigr)\le\mu_n(t)+n^3S_3(t).
    \end{equation}
    \end{lemma}
    
    \begin{proof}[Proof of Lemma~\ref{lem:collision-moments-proof}]
    Condition on $\mathbf P(t)$ and write $I_{ab}=I_{ab}(t)$. Since
    \[
            C_n(t)=\sum_{1\le a<b\le n}I_{ab},
            \qquad\text{and}\qquad
            \P(I_{ab}=1\mid\mathbf P(t))=S_2(t),
    \]
    the sum of the individual variances is
    \[
            \sum_{a<b}\Var(I_{ab}\mid\mathbf P(t))
            =
            \binom n2S_2(t)\bigl(1-S_2(t)\bigr).
    \]
    For two distinct pairs, the corresponding indicators are conditionally
    independent if the pairs are disjoint. If they overlap, for example
    $\{a,b\}$ and $\{a,c\}$, their product is one exactly when the three labels
    $a,b,c$ fall in the same fragment. Hence
    \[
            \Cov(I_{ab},I_{ac}\mid\mathbf P(t))
            =
            S_3(t)-S_2(t)^2.
    \]
    Every set of three labels determines three unordered pairs of overlapping
    pairs, so there are $3\binom n3$ such configurations. Since the variance
    expansion contains twice the sum over unordered pairs, we obtain
    \eqref{eq:collision-variance-proof}.
    
    Finally, using \eqref{eq:conditional-mu-proof}, we have
    \[
            \binom n2S_2(t)\bigl(1-S_2(t)\bigr)\le\mu_n(t),
            \qquad\text{and}\qquad
            6\binom n3\bigl(S_3(t)-S_2(t)^2\bigr)\le n^3S_3(t),
    \]
    so \eqref{eq:collision-variance-bound-proof} holds.
    \end{proof}
    
    \begin{proof}[Proof of the first part of Theorem~\ref{thm:max-height-results}, \eqref{eq:max-height-lln-results}]
    We prove the convergence of $D_n^*/\log n$ in probability by the first and second moment method. The upper bound is the standard first moment argument, used for the critical case in \cite[Theorem~1.4]{Aldous-Pittel-2025}: fix $\varepsilon\in(0,2)$ and define
    \[
            t_n^+:=\frac{(2+\varepsilon)\log n}{r_\beta}.
    \]
    By \eqref{eq:Dstar-collision-proof}, Markov's inequality and \eqref{eq:power-sum-mean-prelim},
    \[
            \P(D_n^*>t_n^+)
            =
            \P(C_n(t_n^+)>0)
            \le
            \Eb C_n(t_n^+)
            =
            \binom n2e^{-r_\beta t_n^+}
            \longrightarrow0.
    \]
    
    For the lower bound, put
    \[
            t_n^-:=\frac{(2-\varepsilon)\log n}{r_\beta}.
    \]
    By Lemma~\ref{lem:quadratic-martingale-proof} and \eqref{eq:conditional-mu-proof},
    \[
            \mu_n(t_n^-)
            =
            \frac12\left(1-\frac1n\right)n^\varepsilon W_\beta(t_n^-)
            \longrightarrow\infty,
            \qquad\text{almost surely}.
    \]
    Moreover, by monotonicity of the $\ell^p$-norms,
    \[
            S_3(t)^{1/3}\le S_2(t)^{1/2},
            \qquad\text{and hence}\qquad
            S_3(t)\le S_2(t)^{3/2}.
    \]
    Therefore Lemma~\ref{lem:collision-moments-proof} gives
    \[
            \frac{\Var(C_n(t_n^-)\mid\mathbf P(t_n^-))}{\mu_n(t_n^-)^2}
            \le
            \frac1{\mu_n(t_n^-)}
            +
            \frac{C}{n\sqrt{S_2(t_n^-)}}.
    \]
    Since
    \[
            n\sqrt{S_2(t_n^-)}
            =
            n^{\varepsilon/2}\sqrt{W_\beta(t_n^-)}
            \longrightarrow\infty,
            \qquad\text{almost surely},
    \]
    the right-hand side tends to zero, and Chebyshev's inequality now yields
    \[
            \P\bigl(C_n(t_n^-)=0\mid\mathbf P(t_n^-)\bigr)
            \le
            \frac{\Var(C_n(t_n^-)\mid\mathbf P(t_n^-))}{\mu_n(t_n^-)^2}
            \longrightarrow0,
            \qquad\text{almost surely}.
    \]
    Taking expectations and using bounded convergence together with
    \eqref{eq:Dstar-collision-proof}, we obtain
    \[
            \P(D_n^*\le t_n^-)
            =
            \Eb\left[\P\bigl(C_n(t_n^-)=0\mid\mathbf P(t_n^-)\bigr)\right]
            \longrightarrow0.
    \]
    Equivalently,
    \[
            \P(D_n^*>t_n^-)\longrightarrow1,
    \]
    so the obtained upper and lower bounds prove \eqref{eq:max-height-lln-results}.
    \end{proof}
    
    \subsubsection{Poisson approximation and fluctuations}
    
    We now determine the behaviour of $D_n^*$ in the critical window around
    $2\log n/r_\beta$. Recall that
    \[
            D_n^*\le t
            \qquad\text{if and only if}\qquad
            C_n(t)=0,
    \]
    so we need to understand the probability that the collision count is zero.
    
    In principle, this could be done by the method of moments. Conditionally on
    $\mathbf P(t)$, the variable $C_n(t)$ is a sum of Bernoulli random variables.
    In the factorial moments, configurations consisting of disjoint pairs give
    the moments of a Poisson random variable, whereas configurations containing
    overlapping pairs give error terms involving $S_3(t)$. This suggests that
    $C_n(t)$ should be approximately Poisson with parameter
    \[
            \mu_n(t)
            =
            \Eb\bigl(C_n(t)\mid\mathbf P(t)\bigr)
            =
            \binom n2S_2(t).
    \]
    The parameter is itself random, and therefore the unconditional limit is a
    mixed Poisson distribution. Since the probability that a Poisson random
    variable of parameter $\mu$ is zero equals $e^{-\mu}$, this leads to
    the mixed Gumbel distribution in
    \eqref{eq:max-height-fluctuations-results}.
    
    Rather than expanding all factorial moments, we use
    \cite[Theorem~1]{ArratiaGoldsteinGordon1989}, which gives a direct error
    bound for a sum of dependent Bernoulli random variables and isolates exactly
    the contribution of the overlapping pairs.
    
    \begin{lemma}[Conditional Poisson approximation]
    \label{lem:poisson-collisions-proof}
    There is a universal constant $C<\infty$ such that, conditionally on
    $\mathbf P(t)$,
    \begin{equation}\label{eq:chen-stein-proof}
            d_{\rm TV}\left(\mathcal L(C_n(t)\mid\mathbf P(t)),\Pois(\mu_n(t))\right)
            \le Cn^3\left[S_2(t)^2+S_3(t)\right].
    \end{equation}
    \end{lemma}
    
    \begin{proof}[Proof of Lemma~\ref{lem:poisson-collisions-proof}]
    Condition on $\mathbf P(t)$ and let
    \[
            \mathcal I_n:=\bigl\{\{a,b\}:1\le a<b\le n\bigr\}.
    \]
    For $\alpha=\{a,b\}\in\mathcal I_n$, write $I_\alpha:=I_{ab}(t)$ and let
    \[
            B_\alpha:=\{\gamma\in\mathcal I_n:\gamma\cap\alpha\ne\varnothing\}.
    \]
    We apply \cite[Theorem~1]{ArratiaGoldsteinGordon1989}. In the notation of
    that theorem, the three error terms are denoted by $b_1,b_2,b_3$, and we treat them in this order.
    
    First,
    \[
            \Eb(I_\alpha\mid\mathbf P(t))=S_2(t),
            \qquad|B_\alpha|\le2n,
            \qquad
            \text{and}\qquad
            |\mathcal I_n|=\binom n2,
    \]
    so it follows that
    \[
            b_1
            :=
            \sum_{\alpha\in\mathcal I_n}
            \sum_{\gamma\in B_\alpha}
            \Eb(I_\alpha\mid\mathbf P(t))
            \Eb(I_\gamma\mid\mathbf P(t))
            \le Cn^3S_2(t)^2.
    \]
    If $\alpha\ne\gamma$ and $\gamma\in B_\alpha$, then the two pairs share one
    label. Both indicators are equal to one exactly when the corresponding
    three labels fall in the same fragment, so
    \[
            \Eb(I_\alpha I_\gamma\mid\mathbf P(t))=S_3(t).
    \]
    There are at most $Cn^3$ ordered overlapping pair-pairs, and therefore
    \[
            b_2
            :=
            \sum_{\alpha\in\mathcal I_n}
            \sum_{\substack{\gamma\in B_\alpha\\\gamma\ne\alpha}}
            \Eb(I_\alpha I_\gamma\mid\mathbf P(t))
            \le Cn^3S_3(t).
    \]
    
    Finally, conditionally on $\mathbf P(t)$, the sampled labels are independent.
    Therefore $I_\alpha$ is independent of the indicators $I_\gamma$ for
    $\gamma\notin B_\alpha$, and the third error term is $b_3=0$.
    The bound of \cite[Theorem~1]{ArratiaGoldsteinGordon1989} now gives
    \eqref{eq:chen-stein-proof}.
    \end{proof}
    
    We are almost ready to conclude. In what follows, the estimate used in the law of large
    numbers,
    \[
            S_3(t)\le S_2(t)^{3/2},
    \]
    is not strong enough. Indeed, at the critical time
    $t_n(x)=(2\log n+x)/r_\beta$, we have $S_2(t_n(x))$ of order $n^{-2}$, and
    this estimate only gives that $n^3S_3(t_n(x))$
    from \eqref{eq:chen-stein-proof} is bounded in probability.
    For the Poisson approximation, we need this quantity to converge to zero.
    
    We therefore use the sharper bound
    \[
            S_3(t)\le P_*(t)S_2(t),
            \quad
            \text{where }
            P_*(t):=\sup_{i\ge1}P_i(t).
    \]
    We will prove that at the critical time $t_n$, $P_*(t_n)$ 
    would be asymptotically smaller than $1/n$. Note that the scale $1/n$ is natural here: conditionally on a fragment of mass $p$, it receives on average $np$ sampled labels. Thus a fragment of mass of order $1/n$ still has a non-negligible chance of receiving more than one label. For the individual fragments to be negligible in the Poisson approximation, we need the largest fragment to be smaller than this sampling scale. The next lemma formalises this.
    
    \begin{lemma}[Largest fragment at the collision scale]
    \label{lem:largest-fragment-collision-proof}
    For every fixed $x\in\Rb$, put
    \begin{equation}\label{eq:t-n-x-proof}
            t_n(x):=\frac{2\log n+x}{r_\beta}.
    \end{equation}
    Then
    \begin{equation}\label{eq:largest-fragment-small-proof}
            nP_*(t_n(x))\longrightarrow0,
            \qquad\text{in probability}.
    \end{equation}
    \end{lemma}
    
    \begin{proof}[Proof of Lemma~\ref{lem:largest-fragment-collision-proof}]
    Put
    \[
            h(p):=\frac{\phi_\beta(p)}{p+1}.
    \]
    By the calculation in the proof of
    Lemma~\ref{lem:quadratic-martingale-proof},
    \[
            h'(1)
            =
            \frac{2\phi_\beta'(1)-\phi_\beta(1)}4
            >0,
            \qquad
            \text{so}\quad
            \frac{\phi_\beta(p)}{p+1}>\frac{r_\beta}{2}
    \]
    for some $p>1$ sufficiently close
    to $1$.
    
    Since $P_*(t)^{p+1}\le S_{p+1}(t)$, Markov's inequality and
    \eqref{eq:power-sum-mean-prelim} give, for every $\delta>0$,
    \[
            \P\bigl(nP_*(t_n(x))>\delta\bigr)
            \le
            \delta^{-(p+1)}n^{p+1}\Eb S_{p+1}(t_n(x))
            =
            \delta^{-(p+1)}e^{-x\phi_\beta(p)/r_\beta}
            n^{p+1-2\phi_\beta(p)/r_\beta}.
    \]
    The exponent of $n$ is negative by the choice of $p$ above, which proves
    \eqref{eq:largest-fragment-small-proof}.
    \end{proof}
    
    \begin{proof}[Proof of the second part of
    Theorem~\ref{thm:max-height-results},
    \eqref{eq:max-height-fluctuations-results}]
    Fix $x\in\Rb$ and put $t_n(x)$ as in \eqref{eq:t-n-x-proof}. By
    Lemma~\ref{lem:quadratic-martingale-proof},
    \begin{equation}\label{eq:mu-limit-proof}
            \mu_n(t_n(x))
            =
            \frac12\left(1-\frac1n\right)e^{-x}W_\beta(t_n(x))
            \longrightarrow
            \frac12e^{-x}W_\beta(\infty),
            \qquad\text{almost surely}.
    \end{equation}
    
    The error in the conditional Poisson approximation tends to zero, because
    \[
            n^3S_2(t_n(x))^2
            =
            n^{-1}e^{-2x}W_\beta(t_n(x))^2
            \longrightarrow0,
            \qquad\text{almost surely}.
    \]
    Moreover,
    \[
            S_3(t)\le P_*(t)S_2(t),
    \]
    and therefore
    \[
            n^3S_3(t_n(x))
            \le
            \bigl[nP_*(t_n(x))\bigr]
            \bigl[n^2S_2(t_n(x))\bigr]
            \longrightarrow0,
            \qquad\text{in probability},
    \]
    by Lemma~\ref{lem:largest-fragment-collision-proof} and the convergence
    \[
            n^2S_2(t_n(x))
            =
            e^{-x}W_\beta(t_n(x))
            \longrightarrow
            e^{-x}W_\beta(\infty),
            \qquad\text{almost surely}.
    \]
    Consequently, Lemma~\ref{lem:poisson-collisions-proof} gives
    \[
            \left|
            \P(C_n(t_n(x))=0\mid\mathbf P(t_n(x)))
            -
            e^{-\mu_n(t_n(x))}
            \right|
            \longrightarrow0,
            \qquad\text{in probability}.
    \]
    The difference is bounded by one, so it also converges to zero in $L^1$.
    Hence
    \[
            \P(C_n(t_n(x))=0)
            =
            \Eb\left[e^{-\mu_n(t_n(x))}\right]+\so(1).
    \]
    By \eqref{eq:mu-limit-proof} and dominated convergence,
    \[
            \Eb\left[e^{-\mu_n(t_n(x))}\right]
            \longrightarrow
            \Eb\left[
            \exp\left(-\frac12e^{-x}W_\beta(\infty)\right)
            \right].
    \]
    Finally, using \eqref{eq:Dstar-collision-proof},
    \[
            \P\bigl(r_\beta D_n^*-2\log n\le x\bigr)
            =
            \P(D_n^*\le t_n(x))
            =
            \P(C_n(t_n(x))=0),
    \]
    and therefore
    \[
            \P\bigl(r_\beta D_n^*-2\log n\le x\bigr)
            \longrightarrow
            \Eb\left[
            \exp\left(-\frac12e^{-x}W_\beta(\infty)\right)
            \right].
    \]
    Together with Lemma~\ref{lem:quadratic-martingale-proof} and the proof of
    \eqref{eq:max-height-lln-results}, this completes the proof of
    Theorem~\ref{thm:max-height-results}.
    \end{proof}
    
    \section{Meromorphic analysis}
    \label{sec:meromorphic-analysis}
    
    We now turn to the proofs of the meromorphic and residue expansions stated in
    Section~\ref{sec:meromorphic}. We begin with the classical meromorphic case,
    where the L\'evy density is a positive mixture of exponentials and the general
    potential representation from Proposition~\ref{prop:U-expression} applies
    directly.
    
    \subsection{The standard meromorphic case}
    \label{sec:negative-beta-proofs}
    
    \begin{proof}[Proof of Theorem~\ref{thm:negative-potential-results}]
    For $-2<\beta<0$, we have already explained in the discussion after
    \eqref{eq:levy-mixture-negative-results} that in this case $\xi$ is meromorphic. Proposition~\ref{prop:U-expression} then gives
    \eqref{eq:U-negative-results} and
    \eqref{eq:reciprocal-negative-results}, together with the positivity of the
    residues and the local uniform convergence of the series.
    
    It remains to identify the atom at zero: recall \eqref{eq:phi-beta-intro}, namely,
    \[
            \phi_\beta(z)
            =
            \int_0^1(1-s^z)s^{\beta+1}(1-s)^\beta\D s,
            \qquad
            \text{for }z>0.
    \]
    If $-1<\beta<0$, monotone convergence gives
    \[
            \phi_\beta(z)\longrightarrow
            \int_0^1s^{\beta+1}(1-s)^\beta\D s
            =
            B(\beta+2,\beta+1),
            \qquad\text{as }z\to\infty.
    \]
    If $-2<\beta\le-1$, the last integral is infinite, and monotone convergence
    instead gives $\phi_\beta(z)\to\infty$. Hence
    \[
            \eta_\beta
            =
            \lim_{z\to\infty}\frac1{\phi_\beta(z)}
            =
            \begin{cases}
            B(\beta+2,\beta+1)^{-1},
            &-1<\beta<0,\\
            0,
            &-2<\beta\le-1.
            \end{cases}
    \]
    
    Finally, the Taylor expansion at the origin gives
    \[
            \phi_\beta(z)
            =
            \mu_\beta z+\frac12\phi_\beta''(0)z^2+O_\beta(z^3),
    \]
    and therefore
    \begin{equation}\label{eq:reciprocal-origin-negative-proof}
            \frac1{\phi_\beta(z)}
            =
            \frac1{\mu_\beta z}
            -
            \frac{\phi_\beta''(0)}{2\mu_\beta^2}
            +
            O_\beta(z).
    \end{equation}
    On the other hand, subtracting $1/(\mu_\beta z)$ from
    \eqref{eq:reciprocal-negative-results} gives, for $z>0$,
    \[
            \frac1{\phi_\beta(z)}-\frac1{\mu_\beta z}
            =
            \eta_\beta
            +
            \sum_{k\ge1}
            \frac1{\phi_\beta'(s_{\beta,k})(z-s_{\beta,k})}.
    \]
    Since $s_{\beta,k}<0$ and the residues are positive, the summands increase as
    $z\downarrow0$. Therefore monotone convergence gives
    \[
            \lim_{z\downarrow0}
            \left[
            \frac1{\phi_\beta(z)}-\frac1{\mu_\beta z}
            \right]
            =
            \eta_\beta
            -
            \sum_{k\ge1}
            \frac1{s_{\beta,k}\phi_\beta'(s_{\beta,k})}.
    \]
    Comparing this with \eqref{eq:reciprocal-origin-negative-proof} proves
    \eqref{eq:finite-part-negative-results}.
    \end{proof}
    
    \begin{proof}[Proof of Theorem~\ref{thm:negative-height-results}]
    Recall that
    \[
            F_n(x)=1-(1-e^{-x})^{n-1}.
    \]
    By Theorem~\ref{thm:potential-identities-results} and
    \eqref{eq:U-negative-results}, and since $F_n(0)=1$, Tonelli's theorem gives
    \begin{equation}\label{eq:EDn-negative-integrals-proof}
            \Eb D_n
            =
            \eta_\beta
            +
            \frac1{\mu_\beta}\int_0^\infty F_n(x)\D x
            +
            \sum_{k\ge1}
            \frac1{\phi_\beta'(s_{\beta,k})}
            \int_0^\infty F_n(x)e^{s_{\beta,k}x}\D x.
    \end{equation}
    The first integral is
    \[
            \int_0^\infty F_n(x)\D x
            =
            \int_0^1\frac{1-(1-y)^{n-1}}y\D y
            =
            h_{n-1}.
    \]
    For $s<0$, the same substitution $y=e^{-x}$ gives
    \[
            \int_0^\infty F_n(x)e^{sx}\D x
            =
            \int_0^1\left[1-(1-y)^{n-1}\right]y^{-s-1}\D y
            =
            -\frac1s-B(-s,n)
            =
            -\frac1s-\frac{\Gamma(-s)\Gamma(n)}{\Gamma(n-s)}.
    \]
    Substituting these into
    \eqref{eq:EDn-negative-integrals-proof} and using
    \eqref{eq:finite-part-negative-results}, we obtain
    \[
            \Eb D_n
            =
            \frac{h_{n-1}}{\mu_\beta}
            -
            \frac{\phi_\beta''(0)}{2\mu_\beta^2}
            -
            \sum_{k\ge1}
            \frac{\Gamma(-s_{\beta,k})\Gamma(n)}
            {\phi_\beta'(s_{\beta,k})\Gamma(n-s_{\beta,k})},
    \]
    which is the first equality in
    \eqref{eq:EDn-negative-exact-results}.
    
    We are left with estimating the remainder after the first $r$ terms. Put
    \[
            a_{\beta,k}:=-s_{\beta,k}>0,
            \qquad\text{and}\qquad
            R_{\beta,r}(n)
            :=
            -\sum_{k\ge r+1}
            \frac{B(a_{\beta,k},n)}
            {\phi_\beta'(s_{\beta,k})}.
    \]
    The sequence $(a_{\beta,k})_{k\ge1}$ is increasing. Moreover, for every
    $n\ge2$,
    \[
            aB(a,n)
            =
            \frac{\Gamma(a+1)\Gamma(n)}{\Gamma(a+n)}
            =
            \frac{\Gamma(n)}{(a+1)(a+2)\cdots(a+n-1)},
    \]
    which is decreasing in $a>0$. Consequently,
    \[
            B(a_{\beta,k},n)
            \le
            \frac{a_{\beta,r+1}}{a_{\beta,k}}
            B(a_{\beta,r+1},n),
            \qquad\text{for }k\ge r+1.
    \]
    It follows that
    \[
            |R_{\beta,r}(n)|
            \le
            C_{\beta,r}B(a_{\beta,r+1},n),
            \qquad\text{where}\qquad
            C_{\beta,r}
            :=
            a_{\beta,r+1}
            \sum_{k\ge r+1}
            \frac1{a_{\beta,k}\phi_\beta'(s_{\beta,k})}.
    \]
    The constant $C_{\beta,r}$ is finite, since
    \eqref{eq:finite-part-negative-results} gives
    \[
            \sum_{k\ge1}
            \frac1{a_{\beta,k}\phi_\beta'(s_{\beta,k})}
            =
            -\frac{\phi_\beta''(0)}{2\mu_\beta^2}
            -\eta_\beta
            <\infty.
    \]
    Finally, using $1-y\le e^{-y}$,
    \[
            B(a,n)
            =
            \int_0^1y^{a-1}(1-y)^{n-1}\D y
            \le
            \int_0^\infty y^{a-1}e^{-(n-1)y}\D y
            =
            \Gamma(a)(n-1)^{-a},
            \quad\text{for }a>0\text{ and }n\ge2.
    \]
    Therefore
    \begin{equation}
        \label{eq:error_bound}
            |R_{\beta,r}(n)|
            \le
            C_{\beta,r}\Gamma(a_{\beta,r+1})
            (n-1)^{-a_{\beta,r+1}}
            =
            O_{\beta,r}\bigl(n^{s_{\beta,r+1}}\bigr),
    \end{equation}
    which proves the second equality in
    \eqref{eq:EDn-negative-exact-results}.
    \end{proof}

    \subsection{The finite-index meromorphic case}
    \label{sec:positive-beta-proofs}
    The proofs in the positive-$\beta$ case use two ingredients. The
    finite-index Nevanlinna argument from
    Section~\ref{sec:nevanlinna-preliminaries} controls the zero set of
    $\phi_\beta$, while the beta-function representation gives the decay on
    vertical lines needed for contour inversion.
    
    Let us define
    \[
            A_\beta:=B(\beta+2,\beta+1).
    \]
    Thus, for $\beta>0$,
    \[
            \phi_\beta(z)=A_\beta-B(z+\beta+2,\beta+1),
            \qquad\text{and}\qquad
            \eta_\beta=A_\beta^{-1}.
    \]
    
    We first record the vertical-line estimate used below.
    
    \begin{lemma}[Vertical-line estimate]
    \label{lem:positive-beta-estimates}
    Let $\beta>0$. For every bounded interval $I\subset\Rb$,
    \begin{equation}\label{eq:vertical-beta-estimate-proof}
            B(z+\beta+2,\beta+1)
            =
            O_{\beta,I}\bigl(|\Im z|^{-\beta-1}\bigr),
            \qquad\text{as }|\Im z|\to\infty,
    \end{equation}
    uniformly for $\re z\in I$. Consequently, if $\phi_\beta$ has no zero on
    $\re z=-\sigma$, then
    \begin{equation}\label{eq:vertical-F-integrability-proof}
            A_{\beta,\sigma}
            :=
            \frac1{2\pi}
            \int_{-\infty}^{\infty}
            |F_\beta(-\sigma+it)|\D t
            <\infty.
    \end{equation}
    \end{lemma}
    
    \begin{proof}[Proof of Lemma~\ref{lem:positive-beta-estimates}]
    The first assertion follows from
    \[
            B(z+\beta+2,\beta+1)
            =
            \Gamma(\beta+1)
            \frac{\Gamma(z+\beta+2)}
            {\Gamma(z+2\beta+3)}
    \]
    and the standard gamma-ratio estimate, uniformly on bounded vertical strips.
    Moreover,
    \[
            F_\beta(z)
            =
            \frac{B(z+\beta+2,\beta+1)}
            {A_\beta\phi_\beta(z)}.
    \]
    If $\phi_\beta$ has no zero on $\re z=-\sigma$, then its modulus is bounded
    away from zero on every bounded part of that line, while
    $\phi_\beta(-\sigma+it)\to A_\beta$ as $|t|\to\infty$. Hence
    \[
            F_\beta(-\sigma+it)
            =
            O_{\beta,\sigma}(|t|^{-\beta-1}),
    \]
    which is integrable because $\beta>0$.
    \end{proof}
    
    \begin{proof}[Proof of Theorem~\ref{thm:positive-zero-results}]
    Let $z=u+iv$ with $u\ge0$. From \eqref{eq:phi-beta-intro},
    \[
            \re\phi_\beta(z)
            =
            \int_0^1
            \left[1-s^u\cos(v\log s)\right]
            s^{\beta+1}(1-s)^\beta\D s.
    \]
    The integrand is nonnegative. If $u>0$, it is strictly positive for
    $s\in(0,1)$. If $u=0$, the integral can vanish only when
    $\cos(v\log s)=1$ for almost every $s\in(0,1)$, which implies $v=0$.
    Therefore
    \[
            \phi_\beta(z)=0,\quad \re z\ge0
            \qquad\Longrightarrow\qquad
            z=0.
    \]
    
    The discussion in Section~\ref{sec:nevanlinna-preliminaries} shows that
    \[
            \varepsilon_\beta\frac{z}{\phi_\beta(z)}
            \in\mathcal N_{<\infty}.
    \]
    The Krein--Langer factorisation quoted there implies that a meromorphic
    function in $\mathcal N_{<\infty}$ has only finitely many nonreal poles.
    The nonzero poles of $z/\phi_\beta(z)$ are precisely the nonzero zeros of
    $\phi_\beta$. Hence $\phi_\beta$ has only finitely many nonreal zeros,
    counted with multiplicity. Since
    $\phi_\beta(\bar z)=\overline{\phi_\beta(z)}$, they occur in conjugate
    pairs.
    
    If $\beta$ is a positive integer, then
    $B(z+\beta+2,\beta+1)$ is rational, and therefore $1/\phi_\beta$ is
    rational and $\mathcal Z_\beta$ is finite.
    
    Finally, fix $\sigma>0$. By
    \eqref{eq:vertical-beta-estimate-proof},
    \[
            \phi_\beta(z)\longrightarrow A_\beta,
            \qquad\text{as }|\Im z|\to\infty,
    \]
    uniformly for $-\sigma\le\re z\le0$. Thus $\phi_\beta$ has no zero in this
    strip when $|\Im z|$ is sufficiently large. Since the zeros of a meromorphic
    function are isolated, only finitely many remain in the bounded part of the
    strip, which proves that $\mathcal Z_{\beta,\sigma}$ is finite.
    
    The remaining claims concerning the zero structure follow from results in
Appendix~\ref{app:extmeromorphic-potential-proof}: Proposition~\ref{appendix:zeros:integer_beta} treats the integer case,
and Proposition~\ref{prop:complexZero} the noninteger one.
    \end{proof}
    
    \begin{proof}[Proof of Theorem~\ref{thm:positive-strip-results}]
    Fix $c>0$ and put
    \[
            \widetilde U_\beta(\D x)
            :=
            U_\beta(\D x)-\eta_\beta\delta_0(\D x).
    \]
    The finite measure $e^{-cx}\widetilde U_\beta(\D x)$ has Fourier transform
    $F_\beta(c+it)$. By \eqref{eq:vertical-beta-estimate-proof} and the fact that
    $\phi_\beta$ has no zero on $\re z=c$, this Fourier
    transform is integrable. Fourier inversion therefore shows that
    $\widetilde U_\beta$ has a continuous density $u_\beta$ on $(0,\infty)$ and
    \begin{equation}\label{eq:bromwich-positive-proof}
            u_\beta(x)
            =
            \frac1{2\pi i}
            \int_{c-i\infty}^{c+i\infty}
            e^{zx}F_\beta(z)\D z,
            \qquad\text{for }x>0.
    \end{equation}
    
    We move the contour in \eqref{eq:bromwich-positive-proof} from
    $\re z=c$ to $\re z=-\sigma$. By
    \eqref{eq:vertical-beta-estimate-proof}, the integrals over the horizontal
    parts of the rectangle tend to zero. The poles crossed are the simple pole at
    $0$ and the poles of $1/\phi_\beta$ at the zeros
    $\rho\in\mathcal Z_{\beta,\sigma}$. The poles of the beta function do not
    contribute, since they are zeros rather than poles of $1/\phi_\beta$.
    
    The pole at zero contributes $1/\mu_\beta$. If $\rho$ has multiplicity
    $m_\rho$, the residue calculation for the full principal part
    \eqref{eq:principal-part-positive-results} gives the contribution
    \[
            u_{\beta,\rho}(x)
            =
            e^{\rho x}
            \sum_{\ell=1}^{m_\rho}
            c_{\rho,\ell}\frac{x^{\ell-1}}{(\ell-1)!}.
    \]
    Consequently,
    \[
            u_\beta(x)
            =
            \frac1{\mu_\beta}
            +
            \sum_{\rho\in\mathcal Z_{\beta,\sigma}}
            u_{\beta,\rho}(x)
            +
            r_{\beta,\sigma}(x).
    \]
    Moreover, by \eqref{eq:vertical-F-integrability-proof},
    \[
            |r_{\beta,\sigma}(x)|
            \le
            A_{\beta,\sigma}e^{-\sigma x},
            \qquad\text{for }x>0.
    \]
    This proves \eqref{eq:u-positive-strip-results}.
    
    It remains to derive the expansion of the mean height. For
    $\rho\in\mathcal Z_{\beta,\sigma}$,
    \[
    \begin{split}
            \int_0^\infty F_n(x)u_{\beta,\rho}(x)\D x
            &=
            \sum_{\ell=1}^{m_\rho}
            \frac{c_{\rho,\ell}}{(\ell-1)!}
            \frac{\partial^{\ell-1}}{\partial\rho^{\ell-1}}
            \left[-\frac1\rho-B(-\rho,n)\right]\\
            &=
            \mathcal P_{\beta,\rho}(0)-\mathcal B_{\beta,\rho}(n).
    \end{split}
    \]
    Taking the Laplace transform of the density decomposition above gives, for
    $z>0$,
    \[
            F_\beta(z)
            =
            \frac1{\mu_\beta z}
            +
            \sum_{\rho\in\mathcal Z_{\beta,\sigma}}
            \mathcal P_{\beta,\rho}(z)
            +
            \int_0^\infty e^{-zx}r_{\beta,\sigma}(x)\D x.
    \]
    Subtracting $1/(\mu_\beta z)$ and letting $z\downarrow0$ gives
    \[
            \eta_\beta
            +
            \sum_{\rho\in\mathcal Z_{\beta,\sigma}}
            \mathcal P_{\beta,\rho}(0)
            +
            \int_0^\infty r_{\beta,\sigma}(x)\D x
            =
            -\frac{\phi_\beta''(0)}{2\mu_\beta^2}.
    \]
    Using Theorem~\ref{thm:potential-identities-results}, we therefore obtain
    \[
            \Eb D_n
            =
            \frac{h_{n-1}}{\mu_\beta}
            -
            \frac{\phi_\beta''(0)}{2\mu_\beta^2}
            -
            \sum_{\rho\in\mathcal Z_{\beta,\sigma}}
            \mathcal B_{\beta,\rho}(n)
            +
            R_{\beta,\sigma}(n),
    \]
    where
    \[
            R_{\beta,\sigma}(n)
            :=
            -\int_0^\infty
            (1-e^{-x})^{n-1}r_{\beta,\sigma}(x)\D x.
    \]
    Finally,
    \[
            |R_{\beta,\sigma}(n)|
            \le
            A_{\beta,\sigma}B(\sigma,n)
            \le
            A_{\beta,\sigma}\Gamma(\sigma)(n-1)^{-\sigma},
    \]
    which proves \eqref{eq:EDn-positive-strip-results}.
    \end{proof}
    \section*{Acknowledgments}
    
    The idea of using potential theory and Bernstein functions in the present
    setting arose during a talk by Svante Janson at the Introductory Workshop
    \emph{Probability and Statistics of Discrete Structures} at the Simons
    Laufer Mathematical Sciences Institute. We thank him and the organizers for this opportunity.

    Yoana R. Chorbadzhiyska was supported by the European Union’s NextGenerationEU, through the National Recovery and Resilience Plan of the Republic of Bulgaria, project No BG-RRP-2.004-0008.
    Martin Minchev was supported by the SNSF SCIEX programme, grant
    IZSF-0-235789. Mladen Savov was partially supported by the financial funds allocated to the Sofia University
“St. Kliment Ohridski”, grant No. 80-10-41/2026.
    
    \section*{Use of artificial intelligence}
    
The main ideas and results of this work were developed in 2025 without the
use of artificial-intelligence tools. At a later stage, OpenAI's ChatGPT
 5.6 was used for language editing and as an additional check of the
arguments. It suggested using the Poisson approximation of Arratia,
Goldstein, and Gordon \cite[Theorem 1]{ArratiaGoldsteinGordon1989} in place
of the factorial-moment argument initially considered for the collision
count. It also pointed us to the theory of regenerative composition
structures as a way of establishing results for $L_n$. Once this connection
had been pointed out, it was straightforward to see that our representation
is, in essence, an extension to the full beta-splitting family of Iksanov's
critical-case construction for the harmonic descent chain
\cite{Iksanov-2025-Regenerative}; the corresponding identity for the height
is made explicit in
\cite[Proposition 2.1]{Iksanov-Nikitin-Yakymiv-2026}. This led us to include
the discrete-height limit laws in Appendix \ref{app:discrete-height}, based
on the work of Gnedin and Pitman \cite{Gnedin-Pitman-2005}. All suggestions
and mathematical arguments were independently checked by the authors, who
take full responsibility for the contents of the paper.
    
    \appendix
    \section{The potential measure of a meromorphic subordinator}
    \label{app:meromorphic-potential-proof}
    
    We give here the proof of
    Proposition~\ref{prop:U-expression}. The argument is based on the
    meromorphic Nevanlinna representation and, in particular, does not rely on a
    formal term-by-term inversion of an infinite Mittag--Leffler series.
    
    \begin{lemma}[Meromorphic complete Bernstein functions]
    \label{lemma:Nevanlinna}
    Let $f$ be a nonconstant complete Bernstein function which admits a
    meromorphic extension to $\Cb$. Then there exist $\gamma,\alpha\ge0$, a
    finite or infinite strictly increasing sequence
    \[
            0<b_1<b_2<\cdots,
    \]
    with $b_n\to\infty$ in the infinite case,
    and coefficients $B_n>0$ satisfying
    $\sum_n B_n/b_n^2<\infty$, such that
    \begin{equation}\label{eq:meromorphic-CBF-series}
            f(z)
            =
            \gamma+\alpha z
            +\sum_{n\ge1}B_n
            \left(
            \frac1{b_n}-\frac1{z+b_n}
            \right)
            =
            \gamma+\alpha z
            +\sum_{n\ge1}B_n
            \frac{z}{b_n(z+b_n)}.
    \end{equation}
    The series converges locally uniformly on
    $\Cb\setminus\{-b_n:n\ge1\}$. The poles of $f$ are real and simple. All zeros of $f$ are real and
    simple, and the zeros interlace with the poles.
    \end{lemma}
    
    \begin{proof}[Proof of Lemma~\ref{lemma:Nevanlinna}]
    A complete Bernstein function extends to a Nevanlinna--Pick function and is
    nonnegative on $(0,\infty)$; see
    \cite[Theorem~7.2]{Schilling-Song-Vondracek-2026}. Since $f$ is meromorphic, the meromorphic
    Nevanlinna representation
    \cite[Theorem~1, p.~197]{chebotarev1949routh},
    available also as     \cite[Theorem 3]{Kuznetsov-2010}, applies. Because $f$ has no
    poles on $[0,\infty)$, the coefficients corresponding to poles on the
    nonnegative half-line vanish. Reindexing the remaining poles as $-b_n$, with
    $b_n>0$, gives \eqref{eq:meromorphic-CBF-series}.
    
    Since $f$ is a nonconstant Nevanlinna--Pick function,
    \[
            \Im f(z)>0,
            \qquad \text{and}
            \qquad\Im z>0.
    \]
    Indeed, otherwise the nonnegative harmonic function $\Im f$ would attain
    its minimum in the upper half-plane and would therefore be constant. By
    conjugation, $f$ has no zeros in the lower half-plane either. Hence all
    zeros of $f$ are real.
    
    Away from its poles,
    \[
            f'(x)
            =
            \alpha+\sum_{n\ge1}\frac{B_n}{(x+b_n)^2}>0.
    \]
    Thus $f$ is strictly increasing on every component of its real domain.
    Moreover,
    \[
            \lim_{x\uparrow-b_n}f(x)=+\infty,
            \qquad\text{and}
            \qquad
            \lim_{x\downarrow-b_n}f(x)=-\infty.
    \]
    Consequently, there is exactly one simple zero between consecutive
    poles, with the usual possible endpoint zero at $0$. Thus the zeros and
    poles interlace.
    \end{proof}
    
    The proof of Proposition~\ref{prop:U-expression} now follows by applying the last result to $z\mapsto z/\phi(z)$, which is again a complete Bernstein function. Dividing the resulting representation by $z$ gives the required representation of $1/\phi(z)$.
    
    \begin{proof}[Proof of Proposition~\ref{prop:U-expression}]
    Put
    \[
            \phi^*(z):=\frac{z}{\phi(z)}.
    \]
    If $\phi^*$ is constant, then $\phi(z)=dz$ for some $d>0$, so $\xi$ is a
    pure drift and the conclusion follows directly from
    $U(\D x)=d^{-1}\D x$. We therefore assume that $\phi^*$ is nonconstant.
    Since $\phi$ is a nonzero complete Bernstein function,
    \cite[Proposition~8.1]{Schilling-Song-Vondracek-2026} implies that
    $\phi^*$ is also a complete Bernstein function. Since $\phi$ is meromorphic,
    so is $\phi^*$.
    
    Applying Lemma~\ref{lemma:Nevanlinna}, we obtain
    \begin{equation}\label{eq:phi-star-series}
            \frac{z}{\phi(z)}
            =
            \gamma'
            +\alpha'z
            +\sum_{n\ge1}B_n'
            \left(
            \frac1{b_n'}-\frac1{z+b_n'}
            \right).
    \end{equation}
    The poles of $\phi^*$ are precisely the nonzero zeros of $\phi$. Writing $s_n:=-b_n'<0$, Lemma~\ref{lemma:Nevanlinna} shows that these zeros are real and simple and interlace with the
    poles of $\phi$.
    
    Dividing \eqref{eq:phi-star-series} by $z$ gives
    \begin{equation}\label{eq:reciprocal-series-proof}
            \frac1{\phi(z)}
            =
            \frac{\gamma'}{z}
            +\alpha'
            +\sum_{n\ge1}
            \frac{B_n'}{b_n'}\frac1{z+b_n'}.
    \end{equation}
    It remains to identify the constants on the right-hand side in terms of $\phi$. By
    \cite[Remark~3.3(iv)]{Schilling-Song-Vondracek-2026}, applied to the Bernstein function $\phi^*$,
    \[
            \gamma'
            =
            \lim_{z\downarrow0}\frac{z}{\phi(z)}
            =
            \frac1{\phi'(0)}
            =
            \frac1\mu,
            \qquad
            \text{and}
            \qquad
            \alpha'
            =
            \lim_{z\to+\infty}\frac1{\phi(z)}
            =
            \eta.
    \]
    Furthermore, since $s_n$ is a simple zero of $\phi$, matching residues gives
    \[
            \frac{B_n'}{b_n'}
            =
            \operatorname*{Res}_{z=s_n}\frac1{\phi(z)}
            =
            \frac1{\phi'(s_n)}.
    \]
    Substituting the identified constants, we obtain
    \begin{equation}\label{eq:meromorphic-reciprocal-proof}
            \frac1{\phi(z)}
            =
            \eta+\frac1{\mu z}
            +\sum_{n\ge1}
            \frac1{\phi'(s_n)(z-s_n)}.
    \end{equation}
    Note that the last expression is the Laplace transform of the measure
    \[
            \widetilde U(\D x)
            :=
            \eta\delta_0(\D x)
            +
            \left[
            \frac1\mu
            +\sum_{n\ge1}
            \frac{e^{s_nx}}{\phi'(s_n)}
            \right]\D x,
    \]
    since, by Tonelli's theorem, for every
    $z>0$,
    \begin{align*}
            \int_{[0,\infty)}e^{-zx}\widetilde U(\D x)
            &=
            \eta+\frac1{\mu z}
            +\sum_{n\ge1}
            \frac1{\phi'(s_n)}
            \int_0^\infty e^{-(z-s_n)x}\D x\\
            &=
            \eta+\frac1{\mu z}
            +\sum_{n\ge1}
            \frac1{\phi'(s_n)(z-s_n)}
            =
            \frac1{\phi(z)}.
    \end{align*}
    The finiteness of the last expression also shows that $\widetilde U$ is
    locally finite. The potential measure $U$ has the same Laplace transform. Fix $z_0>0$. Then the finite measures
    \[
            e^{-z_0x}U(\D x)
            \qquad\text{and}\qquad
            e^{-z_0x}\widetilde U(\D x)
    \]
    have the same Laplace transform. By uniqueness of Laplace transforms,
    \cite[Proposition~1.2]{Schilling-Song-Vondracek-2026}, they are equal, and therefore $U=\widetilde U$, which proves \eqref{eq:meromorphic-potential-prelim}.
    
    Finally, local uniform convergence of the density series on $x>0$ follows
    from positivity and the estimate
    \[
            e^{-t x}
            \le
            \frac{C_{\delta,z}}{z+t},
            \qquad \text{for }x\ge\delta>0,\text{ and } t>0,
    \]
    where $        C_{\delta,z}:=\sup_{t>0}(z+t)e^{-\delta t}<\infty,$
    combined with the convergence of
    \eqref{eq:meromorphic-reciprocal-proof}.
    \end{proof}

     \section{Further properties of the zero set for positive beta}
    \label{app:extmeromorphic-potential-proof}
   We give the detailed zero analysis used in
Theorem~\ref{thm:positive-zero-results} and prove the global representation
in Theorem~\ref{thm:positive-global-results}. The finite-strip estimate in
of Theorem~\ref{thm:positive-strip-results} is proved in Section~\ref{sec:positive-beta-proofs}.

We write $\phi_\beta$ as in
\eqref{eq:phi-beta-meromorphic-results} in terms of Gamma functions and,
with the change of variables $w=z+\beta+2$, obtain,
for $        w\in\Cb$, and
$        \beta>0$,
\begin{equation}\label{eq:phi-beta_ref}
       f_\beta(w)
        :=
        f(\beta,w)
    :=
        \frac{\phi_\beta(w-\beta-2)}{\Gamma(\beta+1)}
        =
        \frac{\Gamma(\beta+2)}{\Gamma(2\beta+3)}
        -
        \frac{\Gamma(w)}{\Gamma(w+\beta+1)}
        =:
        c(\beta)-g(\beta,w).
\end{equation}
Clearly, for every fixed $\beta>0$, the function $f_\beta$ is meromorphic
on $\Cb$; see Subsection \ref{subsec:MerbetaPos}. Next, we give information for the zeros and poles  of \(f_\beta\) when \(\beta\) is a positive integer.

\subsection{Integer values of \texorpdfstring{$\beta$}{beta}}
    \begin{proposition}
    [Zeros and poles for positive integer $\beta$]
    \label{appendix:zeros:integer_beta}
        Let \(\beta>0\)  be an integer. Then 
        \begin{enumeratei} 
            \item \label{item:1}the function \(f_\beta\) has exactly \(\beta+1\) simple zeros and \(\beta+1\) simple poles, the latter being exactly \(\{-\beta,-\beta+1,\dots, 0 \}\);
            
            \item \label{item:2}if \(\beta\) is even, then \(f_\beta\) has a unique real zero at \(w=\beta+2\), and \(\beta\) nonreal zeros, which appear in conjugate pairs;
            
            \item \label{item:3}if \(\beta\) is odd, then \(w=\beta+2\) and \(w=-2(\beta+1)\) are the real zeroes of \(f_\beta\), and its remaining \(\beta-1\) zeros are nonreal and appear in conjugate pairs;
            
            \item \label{item:4} every zero \(w_0\) of \(f_\beta\) satisfies 
             \[
                \Re(w_0)\in\left[-2(\beta+1),\beta+2\right],
                \qquad\text{and}\qquad
                \left|\Im(w_0)\right|<\left(3\beta+4\right)/{2}.
        \]    
        \end{enumeratei}
        \end{proposition}
        
        \begin{proof}[Proof of Proposition \ref{appendix:zeros:integer_beta}] The poles are clear from \eqref{eq:phi-beta_ref} which thanks to the recurrence relation of the Gamma function reduces to 
\[
        f_\beta(w)
        =
        \frac{Q(w)-Q(\beta+2)}{Q(w)Q(\beta+2)},
        \qquad\text{with}\qquad
        Q(w):=\prod_{j=0}^{\beta}(w+j).
\]
            The zeros of \(f_\beta\),
counted with multiplicity, are the \(\beta+1\)  roots of the polynomial in $w$ given by \(Q(w)-Q(\beta+2)\). Next, note that every nonreal zero $w_0$ is simple. 
Indeed, since
\begin{equation}\label{eq:sum-derivative}
f'_\beta(w_0)=\frac{Q'(w_0)}{Q(w_0)^{2}}=
        \frac{1}{Q(\beta+2)}\sum_{j=0}^\beta\frac{1}{w_0+j},
        \quad 
        \text{and}
        \quad
        \Im\lbrb{\sum_{j=0}^\beta\frac{1}{w+j}}=\sum_{j=0}^\beta\frac{-\Im(w)}{|w+j|^2},
\end{equation}
        we have that $f_
        \beta'(w_0)\neq0$ for $\Im w_0\neq 0$.
        
        Clearly, \(w=\beta+2\) is always a zero of \(f_\beta\). Also, \(Q(\beta+2)=(-1)^{\beta+1}Q(-2(\beta+1))\), so \(-2(\beta+1)\) is a zero of \(f_\beta\) iff \(\beta\) is odd.
        Moreover, a direct comparison of the factors, shows that
        \[
        \left| Q(w)\right|<Q(\beta+2),
        \qquad\text{for }
        w\in\left(-2(\beta+1),\beta+2\right),
\]
and
\[
        \left| Q(w)\right|>Q(\beta+2),
     \qquad\text{for }
        w\in\Rb\setminus[-2(\beta+1),\beta+2].
\] Therefore, the real zeros are precisely those given in  \textit{(\ref{item:2})} and \textit{(\ref{item:3})}.
At both points, \(w=\beta+2\) and \(w=-2(\beta+1)\), the terms in
\eqref{eq:sum-derivative} have the same sign, so they are at most simple zeros.  This concludes the proofs of
\textit{(\ref{item:1})}--\textit{(\ref{item:3})}.

        Regarding \textit{(\ref{item:4})}, for a zero $w_0$ of $f_\beta$, that is \( Q(w_0)= Q(\beta+2)\), we have that
        \[   \prod_{j=0}^{\beta}(\beta+2+j)=   
        Q(\beta+2)=
      \lvert Q(w_0)\rvert
      =
              \prod_{j=0}^{\beta}
        \sqrt{(\re(w_0)+j)^2+(\Im(w_0))^2}
        \ge 
\prod_{j=0}^{\beta}
\lvert\re(w_0)+j\rvert,
        \]
        which necessitates \(\Re(w_0)\in[-2(\beta+1),\beta+2]\).
        Moreover, since 
\[
        \lvert\Im(w_0)\rvert^{\beta+1}
        <
        \prod_{j=0}^{\beta}
        \sqrt{(\re(w_0)+j)^2+(\Im(w_0))^2}
        =
        \prod_{j=0}^{\beta}(\beta+2+j),
\]
 the AM--GM inequality gives
\[
      \lvert\Im(w_0)\rvert
        <
        \left[
        \prod_{j=0}^{\beta}(\beta+2+j)
        \right]^{1/(\beta+1)}\leq
        \frac{1}{\beta+1}
        \sum_{j=0}^{\beta}(\beta+2+j)
        =
        \frac{3\beta+4}{2},
\]
which proves \textit{(\ref{item:4})}.
\end{proof}
\subsection{Real zeros for noninteger \texorpdfstring{$\beta$}{beta}}
            Next, we consider the case where $\beta>0$ is noninteger. Introduce the following notation for integer and fractional part
\[
        n_\beta:=\lfloor\beta\rfloor,
        \qquad\text{and}\qquad
        \theta_\beta:=\beta-n_\beta\in(0,1),
\]
so that $\beta=n_\beta+\theta_\beta$. By
\eqref{eq:phi-beta_ref}, the zeros of $f_\beta$ are precisely the
solutions of
\begin{equation}\label{eq:zeroesfg}
        g(\beta,w)=c(\beta)>0.
\end{equation}
\begin{lemma}[Real zeros for noninteger $\beta$]
\label{lem:nozero}
Let $\beta>0$ be noninteger. Then $w=\beta+2$ is a simple zero of
$f_\beta$, and $f_\beta$ has no other zeros in the real interval
$[-n_\beta-1,\infty)$.
\end{lemma}
\begin{proof}[Proof of Lemma \ref{lem:nozero}]
Recall that we defined in \eqref{eq:phi-beta_ref},
\[
f_\beta(w)
:=
\frac{\phi_\beta(w-\beta-2)}{\Gamma(\beta+1)}
=
\frac{\Gamma(\beta+2)}{\Gamma(2\beta+3)}
-
\frac{\Gamma(w)}{\Gamma(w+\beta+1)}
=
c(\beta)-g(\beta,w).
\]
We start with analysing the positive zeros: differentiation in
\eqref{eq:phi-beta-intro} gives
\[
f_\beta'(w)
=
-\frac{1}{\Gamma(\beta+1)}
\int_0^1
\log(s)s^{w-1}(1-s)^\beta\D s
>0,
\qquad
\text{for $w>0$},
\]
so $f_\beta$ is strictly increasing on $(0,\infty)$. Moreover,
\[
f_\beta(\beta+2)=0,
\qquad\text{and}\qquad
f_\beta'(\beta+2)
=
\frac{\mu_\beta}{\Gamma(\beta+1)}>0,
\]
where $\mu_\beta$ is defined in \eqref{eq:mu-tau-results}. Therefore,
$\beta+2$ is a simple zero and the only positive zero of $f_\beta$.

It remains to consider the interval $[-n_\beta-1,0]$. Its integer
points are poles of $f_\beta$. For a noninteger
$w\in[-n_\beta-1,0]$, set
\[
j=j(w):=-\lfloor w\rfloor\in\{1,\ldots,n_\beta+1\},
\qquad\text{and}\qquad
\theta_w:=w-\lfloor w\rfloor=w+j\in(0,1).
\]
Then
$w+\beta+1
>0$,
so $\Gamma(w+\beta+1)>0$. Moreover,
\[
\Gamma(w)=(-1)^j\lvert\Gamma(w)\rvert.
\]
Therefore, if $j$ is odd, then $g(\beta,w)=\Gamma(w)/\Gamma(w+\beta +1)<0$, and thus
$f_\beta(w)>0$.

We are left with even $j$ such that
$j\leq n_\beta+1$. We show via direct estimates that
\begin{equation}\label{eq:inequality-lemmaB2}
g(\beta,w)>
\frac{1}{j!(n_\beta+2-j)!}
>
c(\beta),
\end{equation}
which entails $f_\beta(w)=c(\beta)-g(\beta,w)<0$. Our approach is to estimate the
Gamma functions in $g$ by pushing their arguments onto the positive
real line via the recurrence equation $\Gamma(z+1)=z\Gamma(z)$. We have
\[
\Gamma(w)w(w+1)\cdots(\theta_w-1)
=
\Gamma(\theta_w),
\qquad
\text{giving }
\Gamma(w)>{\Gamma(\theta_w)}/{j!},
\]
and similarly
\[
\begin{split}
\Gamma(w+\beta+1)=(w+\beta)(w+\beta-1)\dots(\theta_w+\theta_\beta)\Gamma(\theta_w+\theta_\beta)\le
(n_\beta+2-j)!\Gamma(\theta_w+\theta_\beta).
\end{split}
\]
Note that by basic properties of the Gamma function
\[
\Gamma(\theta_w)>\Gamma(\theta_w+\theta_\beta).
\]
Indeed, if $\theta_w+\theta_\beta\leq1$, this follows from its strict
decrease on $(0,1]$, and otherwise
$
\Gamma(\theta_w)>1\geq\Gamma(\theta_w+\theta_\beta).
$
Consequently, combining our estimates for $\Gamma(w)$ 
and $\Gamma(w+\beta+1)$,
\[
g(\beta,w)>
\frac{1}{j!(n_\beta+2-j)!},
\]
which gives the first inequality in \eqref{eq:inequality-lemmaB2}.
Finally, since the Gamma function is increasing on $[2,\infty)$, and again using $\Gamma(z+1)=z\Gamma(z)$,
\[
c(\beta)
=
\frac{\Gamma(\beta+2)}
{\Gamma(2\beta+3)}
<
\frac{\Gamma(n_\beta+\theta_\beta+2)}
{\Gamma(2n_\beta+\theta_\beta+3)}
<
\frac{1}{(n_\beta+2)!}
<
\frac{1}{j!(n_\beta+2-j)!}.
\]
This proves also the second inequality of \eqref{eq:inequality-lemmaB2}, and hence $f_\beta$ has no zeros
in $[-n_\beta-1,0]$, which concludes the proof.
\end{proof}
        
We have just characterised the real zeros of $f_\beta$ on
$[-n_\beta-1,\infty)$. We now partition the remaining interval as
\[
(-\infty,-n_\beta-1)
=
\bigcup_{n\geq0}
[-n-n_\beta-2,-n-n_\beta-1).
\]
The poles of $\Gamma(w+\beta+1)$ further split each of these unit
intervals: for each $n\geq0$, let
\begin{equation}\label{eq:intervals}
\mathcal I_n
:=
\lbrb{-n-n_\beta-2,-n-n_\beta-1-\theta_\beta},
\,\,\text{and}\,\,
\mathcal J_n
:=
\lbrb{-n-n_\beta-1-\theta_\beta,-n-n_\beta-1}.
\end{equation}
Thus, $\mathcal I_n$ and $\mathcal J_n$ split the unit interval
$\lbrb{-n-n_\beta-2,-n-n_\beta-1}$ at the point $-n-\beta-1$.
The sign of $g(\beta,\cdot)$ on these intervals depends on the parity
of $n_\beta$, so we consider the even and odd cases separately.

\begin{proposition}\label{prop:realzeroesEven}
Assume that $n_\beta$ is even. Then $f_\beta$ has exactly one simple
real zero in each $\mathcal J_n$ for $n\geq0$. Together with the simple
zero $w=\beta+2$, these are all the real zeros of $f_\beta$.
\end{proposition}
\begin{proof}
By Lemma~\ref{lem:nozero}, $w=\beta+2$ is the only zero of
$f_\beta$ on $[-n_\beta-1,\infty)$, and it is simple. It remains to
consider $w<-n_\beta-1$.

We first show that the equation $g(\beta,w)=c(\beta)>0$ cannot have a solution on $\mathcal I_n$ by a simple sign argument. On $\mathcal J_n$, the idea is to use the reflection formula
\begin{equation}
    \label{eq:reflection-gamma}
\Gamma(z)\Gamma(1-z)=\frac{\pi}{\sin(\pi z)}
\end{equation}
to transfer the arguments of the Gamma functions in $g(\beta,\cdot)$ to the positive half-line, where the expressions are more tractable. In particular, we can find the behaviour at the endpoints of $\mathcal J_n$, and further obtain monotonicity, which proves a unique solution of $g(\beta,w)=c(\beta)$. We now provide the details.

For $w\in\mathcal I_n$,
\[
\lfloor w\rfloor=-n-n_\beta-2,
\qquad\text{and}\qquad
\lfloor w+\beta+1\rfloor=-n-1.
\]
Since $\operatorname{sgn}\Gamma(x)=(-1)^{\lfloor x\rfloor}$ for negative
noninteger $x$,
\[
\operatorname{sgn}g(\beta,w)
=
(-1)^{\lfloor w\rfloor-\lfloor w+\beta+1\rfloor}
=
(-1)^{n_\beta+1}
=
-1.
\]
Thus, $g(\beta,w)<0<c(\beta)$ on $\mathcal I_n$, so
\eqref{eq:zeroesfg} has no solution there. We now consider $\mathcal J_n$. For $w\in\mathcal J_n$, write
\[
w=-n-n_\beta-1-t,
\qquad\text{with}\qquad
t\in(0,\theta_\beta).
\]
Applying \eqref{eq:reflection-gamma} to both Gamma functions in
$g(\beta,w)$ gives
\begin{equation}\label{eq:grealEvenJ}
g(\beta,w)
=
\frac{\Gamma(-n-n_\beta-1-t)}
{\Gamma(-n-t+\theta_\beta)}
\\
=
(-1)^{n_\beta}
\frac{\sin\lbrb{\pi(\theta_\beta-t)}}{\sin(\pi t)}
\frac{\Gamma(1+n+t-\theta_\beta)}
{\Gamma(n+n_\beta+2+t)},
\end{equation}
and since $n_\beta$ is even, we have that $g$ is strictly positive on $\mathcal J_n$ and
\[
\lim_{t\to0+}g(\beta,-n-n_\beta-1-t)
=+\infty,
\qquad\text{and}\qquad
\lim_{t\to\theta_\beta-}
g(\beta,-n-n_\beta-1-t)=0.
\]Therefore, \eqref{eq:zeroesfg} has at least one solution in $\mathcal J_n$.
We now prove that this solution is unique, which we prove is unique by considering the logarithmic derivative in $w$ of $g$: the reflection formula
for the digamma function $\psi(x):=(\log \Gamma(x))'=\Gamma'(x)/\Gamma(x)$ reads
\[
\psi(1-x)-\psi(x)=\pi\cot(\pi x),
\]
so the $\log$-derivative satisfies
\[
\frac{g'_w(\beta,w)}{g(\beta,w)}
=
\psi(n+n_\beta+2+t)-\psi(1+n+t-\theta_\beta)
+\pi\cot\lbrb{\pi(\theta_\beta-t)}
+\pi\cot(\pi t).
\]
Note that  $\psi$ is strictly increasing on $(0,\infty)$ and
\[
n+n_\beta+2+t>1+n+t-\theta_\beta>0,
\]
which gives positivity of the digamma difference. Also, by a standard trigonometric identity
\[
\cot u+\cot v
=
\frac{\sin(u+v)}{\sin u\sin v}>0,
\qquad
\text{for }u,v>0,\quad\text{and}\quad
u+v<\pi,
\]
so the logarithmic derivative is positive. Since $g>0$ on $\mathcal J_n$,
it follows that $g'_w(\beta,w)>0$ there. Hence,
\eqref{eq:zeroesfg} has exactly one solution in $\mathcal J_n$, and
at this solution,
\[
f_\beta'(w)=-g'_w(\beta,w)<0,
\]
so the zero is simple. Finally, the integer endpoints are poles of $f_\beta$, while
\[
g(\beta,-n-\beta-1)=0,
\qquad\text{and therefore}\qquad
f_\beta(-n-\beta-1)=c(\beta)>0,
\]
so there are no further real zeros.
\end{proof}
\begin{proposition}\label{prop:realzeroesOdd}
Assume that $n_\beta$ is odd. Then $f_\beta$ has exactly one simple
real zero in each $\mathcal I_n$ for $n\geq n_\beta+2$. In each interval
$\mathcal I_n$ for $0\leq n<n_\beta+2$, there is at least one real zero.
\end{proposition}

\begin{proof}
The sign argument used in the proof of Proposition
\ref{prop:realzeroesEven} gives $g<0$ on
$\mathcal J_n$, since $n_\beta$ is odd, so there are no zeros there. For $\mathcal I_n$, we need the asymptotic behaviour at its limits points similar to \eqref{eq:grealEvenJ}: for $w\in\mathcal I_n$, write
\[w=-n-\beta-1-s,\qquad
\text{with $s\in(0,1-\theta_\beta)$}.\]
Applying
\eqref{eq:reflection-gamma} to both Gamma functions in $g(\beta,w)$ gives
\begin{equation}\label{eq:grealOddI}
g(\beta,w)
=
\frac{\Gamma(-n-n_\beta-1-\theta_\beta-s)}
{\Gamma(-n-s)}
=
(-1)^{n_\beta+1}
\frac{\sin(\pi s)}
{\sin\lbrb{\pi(\theta_\beta+s)}}
\frac{\Gamma(n+1+s)}
{\Gamma(n+n_\beta+2+\theta_\beta+s)}.
\end{equation}
Since $n_\beta$ is odd, we have that $g$ is strictly positive on
$\mathcal I_n$ and
\[
\lim_{s\to0+}g(\beta,-n-\beta-1-s)
=0,
\qquad\text{and}\qquad
\lim_{s\to(1-\theta_\beta)-}
g(\beta,-n-\beta-1-s)=+\infty.
\]
Therefore, since $c(\beta)>0$, equation \eqref{eq:zeroesfg} has at least
one solution in each interval $\mathcal I_n$.

With $w=-n-n_\beta-1-\theta_\beta-s$, $s\in(0,1-\theta_\beta)$, the logarithmic derivative of $g$ is given by
\begin{equation}
\frac{g'_w(\beta,w)}{g(\beta,w)}
=
\psi(n+n_\beta+2+\theta_\beta+s)-\psi(n+1+s)
-\pi\cot(\pi s)+\pi\cot\lbrb{\pi(\theta_\beta+s)}.
\end{equation}
Denote
\[
h_1(s)
:=
\psi(n+n_\beta+2+\theta_\beta+s)-\psi(n+1+s),
\quad
\text{and}
\quad
h_2(s)
:=
\pi\cot(\pi s)-\pi\cot\lbrb{\pi(\theta_\beta+s)}.
\]
Regarding $h_1$, using the classical series representation of the
digamma function \cite[Equation (5.7.6)]{NIST-Handbook},
\[
\psi(z) = -\gammaEuler +\sum_{k=0}^\infty\left(
\frac{1}{k+1}-\frac{1}{k+z}\right),
\]
and differentiating term by term, for every $m\geq0$ we obtain
\[
(-1)^m h_1^{(m)}(s)
=
m!\sum_{k\geq0}
\left(
\frac{1}{(k+n+1+s)^{m+1}}
-
\frac{1}{(k+n+n_\beta+2+\theta_\beta+s)^{m+1}}
\right)
>0.
\]
Therefore, $h_1$ is completely monotone, and in particular positive,
decreasing, and convex. The function $h_2$ is also positive and convex, and is symmetric with respect to $(1-\theta_\beta)/2$. The logarithmic derivative is a sum of contributions with opposite signs.

Consider $n\geq n_\beta+2$. Since $x\mapsto\Gamma(x)/\Gamma(x+\beta+1)$ is strictly decreasing for $x>0$ and $n+1+s>\beta+2$, for $s\in(0,1-\theta_\beta)$ we have
\[
c(\beta)
=
\frac{\Gamma(\beta+2)}{\Gamma(2\beta+3)}
>
\frac{\Gamma(1+n+s)}
{\Gamma(n+n_\beta+2+\theta_\beta+s)}.
\]
Then, for the equality \eqref{eq:zeroesfg} to hold, it is required that $\sin(\pi s)>\sin\lbrb{\pi(\theta_\beta+s)}$, i.e. every zero of $f_\beta$ on $\mathcal I_n$ satisfies $s>(1-\theta_\beta)/2$.

Put $a_\beta:=(1-\theta_\beta)/2$ and $G_n(s):=g(\beta,-n-\beta-1-s)$. At $s=a_\beta$, the quotient of the sine functions is equal to one, and therefore
\[
G_n(a_\beta)
=
\frac{\Gamma(1+n+a_\beta)}
{\Gamma(n+n_\beta+2+\theta_\beta+a_\beta)}
<
c(\beta).
\]
On $(a_\beta,1-\theta_\beta)$, the function $h_1$ decreases, whereas $h_2$ increases. Therefore, $h_2-h_1$ is strictly increasing. Moreover,
\[
\frac{\D}{\D s}\log G_n(s)=h_2(s)-h_1(s),
\]
so the logarithmic derivative of $G_n$ is strictly increasing on this interval. Since $h_2(s)-h_1(s)\to+\infty$ as $s\to(1-\theta_\beta)-$, consequently, $G_n$ either increases throughout the interval or first decreases and then increases. Since
\[
G_n(a_\beta)<c(\beta),
\qquad\text{and}\qquad
\lim_{s\to(1-\theta_\beta)-}G_n(s)=+\infty,
\]
the equation $G_n(s)=c(\beta)$ has exactly one solution. This solution lies on the strictly increasing section of $G_n$. Thus, at the corresponding value of $w$,
\[
g'_w(\beta,w)=-G_n'(s)<0,
\qquad\text{and therefore}\qquad
f_\beta'(w)=-g'_w(\beta,w)>0,
\]
so the zero is simple.

Thus, $\mathcal I_n$ contains exactly one simple zero for $n\geq n_\beta+2$, while the endpoint behaviour gives at least one zero for $0\leq n<n_\beta+2$. There are no solutions in $\mathcal J_n$, which concludes the proof.
\end{proof}
\subsection{Nonreal zeros and global properties}
\begin{lemma}[Uniform estimates on square contours]
\label{lem:g}
Fix an integer $n\geq0$, $\gamma\in(0,1)$, and
$0\le\ell\le r\le1$ such that $n+1-r>0$,
\[
        \alpha:=n+1-r>0.
\]
For an integer $N\geq1$, set $R_N:=N-\gamma$, and let $D_N$ be the
positively oriented boundary of the square
\[
        \left\{
        w\in\Cb:
        |\Re w|\leq R_N,\quad
        |\Im w|\leq R_N
        \right\}.
\]
Then
\begin{equation}\label{eq:g-square}
        \sup_{\ell\leq\eta\leq r}
        \sup_{w\in D_N}
        |g(n+\eta,w)|
        =
        \bo(
        R_N^{-\alpha})=
        \bo(N^{-\alpha}),
        \qquad \text{as $N\to\infty$.}
\end{equation}
Consequently, for all sufficiently large $N$ and every
$\eta\in[\ell,r]$, the function $f_{n+\eta}$ does not have zeros nor
poles on $D_N$, and
\begin{equation}\label{eq:reciprocal-square}
        \sup_{\ell\leq\eta\leq r}
        \sup_{w\in D_N}
        \left|
        \frac1{f_{n+\eta}(w)}
        -
        \frac1{c(n+\eta)}
        \right|=
        \bo(N^{-\alpha}).
\end{equation}
\end{lemma}
\begin{proof}[Proof of Lemma \ref{lem:g}]
We first estimate $g$ on the right side of $D_N$. The beta-integral
representation gives, uniformly for $\eta\in[\ell,r]$ and $y\in\Rb$,
\[
\begin{aligned}
|g(n+\eta,R_N+iy)|
&=
\frac{1}{\Gamma(n+\eta+1)}
\left|
\int_0^1 t^{R_N+iy-1}(1-t)^{n+\eta}\D t
\right|
\\
&\leq
\frac{1}{\Gamma(n+\eta+1)}
\int_0^1 t^{R_N-1}(1-t)^{n+\eta}\D t
=
\frac{\Gamma(R_N)}{\Gamma(R_N+n+\eta+1)}.
\end{aligned}
\]
For $N\geq2$, we have
$R_N+n+1>2$, and since the Gamma function is increasing on $[2,\infty)$,
\[
\frac{\Gamma(R_N)}{\Gamma(R_N+n+\eta+1)}
\leq
\frac{\Gamma(R_N)}{\Gamma(R_N+n+1)}
=
\prod_{k=0}^{n}\frac1{R_N+k}
\leq R_N^{-(n+1)}.
\]
Consequently,
\[
\sup_{\ell\leq\eta\leq r}
\sup_{|y|\leq R_N}
|g(n+\eta,R_N+iy)|
\leq R_N^{-(n+1)}
\leq R_N^{-\alpha},
\]
which gives the required bound on the right vertical side of $D_N$.

We next consider the upper side. By the uniform version of Stirling's
formula, see \cite[(5.11.9)]{NIST-Handbook},
\[
        |\Gamma(x+it)|
        \sim
        \sqrt{2\pi}|t|^{x-\frac12}e^{-\pi|t|/2},
        \qquad
        \text{uniformly for $x$ in bounded real intervals as
        $|t|\to\infty$.}
\]
Applying this formula for $x\in[0,n+3]$ provides
that there
exists a constant $C>0$, depending only on $n$, such
that, for all sufficiently large $N$,
\begin{equation}\label{eq:bounded-strip-ratio}
\left|
        \frac{\Gamma(v+iR_N)}
             {\Gamma(v+n+\eta+1+iR_N)}
        \right|
        \leq
        C R_N^{-(n+\eta+1)},
        \qquad\text{for $v\in[0,1]$ and $\eta\in[\ell,r]$}.
\end{equation}
We next obtain the bounds for $g(n+\eta,x+iR_N)$
for $x \in [-R_N,R_N]$ applying the Gamma recurrence formula, pushing the quotient until we can apply \eqref{eq:bounded-strip-ratio}: first, for $0\leq x\leq R_N$,
\[
g(n+\eta,x+iR_N)
=
\frac{\Gamma(\{x\}+iR_N)}
     {\Gamma(\{x\}+n+\eta+1+iR_N)}
\prod_{k=0}^{\lfloor x\rfloor-1}
\frac{\{x\}+k+iR_N}{\{x\}+n+\eta+1+k+iR_N}.
\]
Every factor in the last product has modulus at most $1$, so by \eqref{eq:bounded-strip-ratio}
\begin{equation}
\label{eq:g-upper-line}
|g(n+\eta,x+iR_N)|
        \leq C R_N^{-(n+\eta+1)},
        \qquad
        \text{for $0\leq x\leq R_N$}.
\end{equation}
The approach for negative values is similar: for $-R_N\leq x<0$, 
\[
g(n+\eta,x+iR_N)
=
\frac{\Gamma(1-\{-x\}+iR_N)}
     {\Gamma(1-\{-x\}+n+\eta+1+iR_N)}
\prod_{k=0}^{\lfloor-x\rfloor}
\left(
1+\frac{n+\eta+1}{x+k+iR_N}
\right).
\]
Since $\lfloor-x\rfloor+1\leq R_N+1$ and $n+\eta+1\leq n+2$,
\[
\prod_{k=0}^{\lfloor-x\rfloor}
\left|
1+\frac{n+\eta+1}{x+k+iR_N}
\right|
\leq
\left(1+\frac{n+2}{R_N}\right)^{R_N+1}
\leq C.
\]
Together with \eqref{eq:bounded-strip-ratio}, this yields
\[
        \sup_{\ell\leq\eta\leq r}
        \sup_{-R_N\leq x\leq R_N}
        |g(n+\eta,x+iR_N)|
        \leq
        C R_N^{-(n+\ell+1)}
        \leq C R_N^{-(n+1)}.
\]
Note that $g$ is symmetric, in the sense that
$\overline{g(\beta,w)}=g(\beta,\overline{w})$, so the last estimate transfers to the lower horizontal part of $D_N$ as well.

It remains to consider the left vertical side of $D_N$: for it, we use the reflection formula, so we are left with a positive vertical line, which we have already seen how to treat. The reflection formula \eqref{eq:reflection-gamma} gives
\begin{equation}\label{eq:g-reflection}
        g(n+\eta,w)
        =
        \frac{\sin(\pi(w+n+\eta+1))}{\sin(\pi w)}\,
        g(n+\eta,-w-n-\eta).
\end{equation}
Using that $|\sin z|^2 = \sin^2\Re z + \sinh^2 \Im z$, we calculate, for $w=-R_N+iy=-N+\gamma+iy$,
\[
\left|
\frac{\sin(\pi(w+n+\eta+1))}{\sin(\pi w)}
\right|^2
=
\frac{\sin^2\!\bigl(\pi(\gamma+\eta)\bigr)
      +\sinh^2(\pi y)}
     {\sin^2(\pi\gamma)+\sinh^2(\pi y)}
\leq
\frac{1+\sinh^2(\pi y)}{\sin^2(\pi\gamma)+\sinh^2(\pi y)}\leq
\frac1{\sin^2(\pi\gamma)}.
\]
Note also that $        -w-n-\eta
        =
        R_N-n-\eta-iy:=x_{N,\eta}-iy.$
For all sufficiently large $N$, uniformly in $\eta\in[\ell,r]$,
\[
        x_{N,\eta}
        \ge R_N-n-r
        \ge {R_N}/{2}.
\]
Repeating the beta-integral estimate from the beginning of the proof,
now with $x_{N,\eta}$ in place of $R_N$, gives, uniformly in
$y\in\Rb$,
\[
\begin{aligned}
|g(n+\eta,x_{N,\eta}-iy)|
&\le
\frac{\Gamma(x_{N,\eta})}
     {\Gamma(x_{N,\eta}+n+\eta+1)}
\le
x_{N,\eta}^{-(n+1)}
\le
2^{n+1}R_N^{-(n+1)}.
\end{aligned}
\]
Together with the bound on the sine-quotient, this gives
\[
\sup_{\ell\leq\eta\leq r}
        \sup_{|y|\leq R_N}
        |g(n+\eta,-R_N+iy)|
        \leq
        C R_N^{-(n+1)}
        \leq C R_N^{-\alpha},
\]
which concludes the uniform estimate \eqref{eq:g-square}.

As for the non-existence of poles and zeros of
$f_{n+\eta}$ on the contour, there are no poles on the vertical parts
of $D_N$, since their real parts are noninteger, and on the horizontal, since their values are nonreal. For the zeros, since $c$ is positive and continuous, 
\[
        c_*:=
        \inf_{\ell\leq\eta\leq r}c(n+\eta)>0,
\]
so  thanks to our asymptotic result \eqref{eq:g-square}, $g$ would be sufficiently small on $D_N$, in the sense that for sufficiently large $N$,
\[
        |g(n+\eta,w)|\leq{c_*}/{2},
        \qquad
        \text{for $\eta\in[\ell,r]$ and $w\in D_N$},
\]
so $|f_{n+\eta}(w)|\geq c(n+\eta)-|g(n+\eta,w)|>0$. Finally,
\[
\left|
\frac1{f_{n+\eta}(w)}
-
\frac1{c(n+\eta)}
\right|
=
\frac{|g(n+\eta,w)|}
     {c(n+\eta)|f_{n+\eta}(w)|}
\leq
\frac{2}{c_*^2}|g(n+\eta,w)|,
\]
which proves \eqref{eq:reciprocal-square}.
\end{proof}

\begin{proposition}[Nonreal and exceptional zeros]
\label{prop:complexZero}
Let $\beta>0$ be noninteger and let $      \beta=n_\beta+\theta_\beta$ with $       n_\beta=\lfloor\beta\rfloor$.

If $n_\beta$ is even, then $f_\beta$ has exactly $n_\beta$
nonreal zeros, counted with multiplicity. They occur in conjugate pairs.

Suppose that $n_\beta$ is odd. Let $q_\beta$ be the number of zeros
of $f_\beta$ in the upper half-plane, counted with multiplicity, and,
for $0\leq k\leq n_\beta+1$, let $M_k(\beta)$ be the total
multiplicity of the real zeros of $f_\beta$ in $\mathcal I_k$. Then
each $M_k(\beta)$ is odd and
\begin{equation}\label{eq:odd-zero-count}
        2q_\beta
        +
        \sum_{k=0}^{n_\beta+1}
        \left[M_k(\beta)-1\right]
        =
        n_\beta+1.
\end{equation}
In particular, $f_\beta$ has at most $n_\beta+1$ nonreal zeros,
counted with multiplicity.

Consequently, for every noninteger $\beta>0$, only finitely many zeros
of $f_\beta$ are nonreal or multiple. All the remaining zeros are real
and simple.
\end{proposition}

\begin{proof}[Proof of Proposition \ref{prop:complexZero}]
Fix the value \(\beta=n_\beta+\theta_\beta\). We have described the
solutions of $g(\beta,w)=c(\beta)$ for positive integer values of
$\beta$ in Proposition~\ref{appendix:zeros:integer_beta}; the case
$\beta=0$ follows directly from $f_0(w)=1/2-1/w$. Our idea now is to consider
$g(\zeta,w)$ as a function of $\zeta$ on the interval joining $\beta$
to the closest even integer
\[
H:=
\begin{cases}
[n_\beta,\beta], & \text{if } n_\beta \text{ is even},\\
[\beta,n_\beta+1], & \text{if } n_\beta \text{ is odd}.
\end{cases}
\]
The choice of prioritising even endpoints is motivated by the simpler structure of $f_{n}$ for even $n$, described in
Proposition \ref{appendix:zeros:integer_beta}. In particular,
$f_{n}$ has precisely one real zero.

Once we choose a common contour $D$ such that $f_\zeta$ has
neither zeros nor poles on $D$ for every $\zeta\in H$, the argument
principle gives
\[
K_D(\zeta)
:=
\frac{1}{2\pi i}
\int_D\frac{f_\zeta'(w)}{f_\zeta(w)}\D w,
\qquad
\text{for all $\zeta\in H$}.\]
The latter equals also the number of zeros minus the number of poles enclosed by $D$,
both counted with multiplicity. Since $K_D$ is a continuous
integer-valued function on $H$, it is constant. We know its exact value
from the integer case. For reference, let $m_\beta$ be this closest even
integer: $m_\beta=n_\beta$ when \(n_\beta\) is even and
$m_\beta=n_\beta+1$ when \(n_\beta\) is odd. Then we will compare
\[
f_\zeta, \text{ for }
\zeta \in H,
\qquad
\text{with}
\qquad f_{m_\beta}.
\]
The contours we will use are the squares from Lemma~\ref{lem:g}, chosen so that their left side does not go through a zero of $f$. In the notation of the invoked lemma, pick
\[
        \gamma:=1- \theta_\beta/ 2, 
        \qquad n=n_\beta,
        \qquad
        \text{and}
        \qquad
        (\ell,r) = (0,\theta_\beta)
        \text{ or } (\theta_\beta,1),
        \qquad
\]
where the last choice depends on the fact that 
$n_\beta$ is even or odd, respectively.
In both cases $\alpha=n_\beta+1-r>0$, and $n_\beta+\eta$ ranges precisely over $H$.
With this choice, for an integer $N\geq1$, the contour $D_{N+1}$ is exactly the positively oriented boundary of 
\[
\left\{
w\in\mathbb C:
|\Re w|\leq N+\frac{\theta_\beta}{2},
\quad
|\Im w|\leq N+\frac{\theta_\beta}{2}
\right\}.
\]
Again by Lemma~\ref{lem:g}, for all sufficiently large $N$ and every
$\zeta\in H$, the function $f_\zeta$ has neither zeros nor poles on
$D_{N+1}$. Increasing $N$ if needed, we may also assume that
$D_{N+1}$ encloses all zeros and poles of $f_{m_\beta}$ and the positive
real zero $\zeta+2$ of $f_\zeta$ for every $\zeta\in H$.

Consider now\[
K_{D_{N+1}}(\zeta)
=
\frac{1}{2\pi i}
\int_{D_{N+1}}
\frac{f_\zeta'(w)}{f_\zeta(w)}\D w.
\]
The function $K_{D_{N+1}}$ is continuous and integer-valued on $H$, and
therefore constant. Moreover, at the integer point $m_\beta\in H$, for $m_\beta>0$, by Proposition \ref{appendix:zeros:integer_beta}, we know that its value is 0. 
The same is true directly for $m_\beta=0$
from $f_0(w)=1/2-1/w$. Therefore,
\[
K_{D_{N+1}}(\zeta)=0,
\qquad
\text{for }
\zeta\in H.
\]
At noninteger $\beta$, the poles of $f_\beta$ are exactly the nonpositive integers in $D_{N+1}$ and simple. 
Therefore we should have $N+1$ zeros in $D_{N+1}$, counted with multiplicity.

We now count the number of nonreal zeros. First consider the case where $n_\beta$ is even. Then the left side of the countour $D_{N+1}$ passes through
\[
        -N-{\theta_\beta}/{2}
        \in(-N-\theta_\beta,-N)
        =
        \mathcal J_{N-n_\beta-1}.
\]
Proposition~\ref{prop:realzeroesEven} gives $N-n_\beta$ negative real
zeros inside $D_{N+1}$, together with one positive zero, so the
number of nonreal zeros inside $D_{N+1}$,
counted with multiplicity, is
\[
        (N+1)-(N-n_\beta)-1=n_\beta.
\]
This holds for all sufficiently large $N$, and since the contours $D_{N}$ cover $\mathbb C$, $f_\beta$ has exactly $n_\beta$
nonreal zeros.

Similarly, suppose that $n_\beta$ is odd. Proposition~\ref{prop:realzeroesOdd} gives at least $N-n_\beta-1$ negative real zeros inside $D_{N+1}$, together with one positive zero, so the number of nonreal zeros inside $D_{N+1}$, counted with multiplicity, is at most
\[
(N+1)-(N-n_\beta-1)-1=n_\beta+1.
\]
This holds for all sufficiently large $N$, and since the contours $D_N$ cover $\mathbb C$, $f_\beta$ has at most $n_\beta+1$ nonreal zeros. If this upper bound is attained, then all real zeros are simple.

The claim about conjugacy of the zeros follows again from the symmetry $
        f_\beta(\overline w)=\overline{f_\beta(w)}$.

Let $q_\beta$ be the number of zeros in the upper half-plane. Choose
$N$ sufficiently large that $D_{N+1}$ contains all nonreal zeros and
all real zeros in the intervals
$
        \mathcal I_0,\ldots,\mathcal I_{n_\beta+1}.
$
For the remaining $\mathcal I_k$ for $k\geq n_\beta+2$, Proposition~\ref{prop:realzeroesOdd} gives
exactly one simple zero in each of them. Therefore, for sufficiently
large $N$, the $N+1$ zeros inside $D_{N+1}$ consist of

\begin{itemize}
    \item the positive zero $w=\beta+2$;
    \item the $2q_\beta$ nonreal zeros;
    \item the real zeros in the exceptional intervals, with total
          multiplicity
          $\sum_{k=0}^{n_\beta+1}M_k(\beta);
          $
    \item one simple zero in each interval $\mathcal I_k$ for
    $n_\beta+2\leq k\leq N-n_\beta-2.$
\end{itemize}
There are $N-2n_\beta-3$ intervals of the last type, so we have
\[
        N+1
        =
        1
        +
        2q_\beta
        +
        \sum_{k=0}^{n_\beta+1}M_k(\beta)
        +
        N-2n_\beta-3,
        \qquad
        \text{so}
        \qquad
                2q_\beta
        +
        \sum_{k=0}^{n_\beta+1}
        \left[M_k(\beta)-1\right]
        =
        n_\beta+1.
\]
Proposition \ref{prop:realzeroesOdd} ensures the existence of at least
one real zero in each $\mathcal I_k$, so $M_k(\beta)-1$ can be read as the
excess beyond this guaranteed zero.
It remains to note that each $M_k(\beta)$ is odd. Indeed, from
\eqref{eq:grealOddI},
\[
        \lim_{w\downarrow -k-n_\beta-2}
        f_\beta(w)
        =
        -\infty, \qquad
        \text{and}
        \qquad
        \lim_{w\uparrow-k-n_\beta-1-\theta_\beta}
        f_\beta(w)
        =
        c(\beta)>0.
\]
Thus $f_\beta$ has opposite signs at the two ends of
$\mathcal I_k$. Moving from one endpoint to the other, the sign changes
exactly when a zero of odd multiplicity is crossed, so the number of
zeros of odd multiplicity should be odd, and therefore the sum of multiplicities, that is $M_k(\beta)$, is odd.

Finally, Proposition~\ref{prop:realzeroesOdd} shows that, for
$k\geq n_\beta+2$, the interval $\mathcal I_k$ contains exactly one
zero, and this zero is simple. Hence all multiple real zeros lie in the
finitely many exceptional intervals
$\mathcal I_0,\ldots,\mathcal I_{n_\beta+1}$. Each of these intervals
contains only finitely many zeros. Together with the finiteness of the
nonreal zero set, this proves the final claim.
\end{proof}

\subsection{A potential measure expansion for positive \texorpdfstring{$\beta$}{beta}}\label{sec:potential-positive-beta-proof}
\begin{lemma}[Meromorphic contour expansion; \cite{Titchmarsh-1958}, Section~3.2, p.~110]
\label{lem:meromorphic_contours}
Let $F$ be a meromorphic function on $\Cb$ and holomorphic at $0$, with simple
poles $a_1,a_2,\ldots$,
and let $A_1,A_2,\ldots$ denote the corresponding residues. Suppose
there exists a sequence of positively oriented, piecewise smooth,
simple closed contours $(C_N)_{N\geq1}$ such that
\begin{itemize}
    \item $0$ lies inside $C_N$, no pole lies on $C_N$, and these contours cover $\Cb$;

    \item the minimum distance $d_N$
    from $C_N$ to the origin satisfies $d_N\to\infty$ as $N\to\infty$;

    \item the length of $C_N$ is $\bo(d_N)$;

    \item $F(z)=\so(d_N)$ uniformly for $z\in C_N$.
\end{itemize}
Then, locally uniformly for
$z\in\Cb\setminus\{a_1,a_2,\ldots\}$,
\[
F(z)
=
F(0)
+
\lim_{N\to\infty}
\sum_{a_n\in\operatorname{int}(C_N)}
A_n
\left(
    \frac{1}{z-a_n}+\frac{1}{a_n}
\right).
\]
\end{lemma}

\begin{proof}[Proof of Theorem~\ref{thm:positive-global-results}]
Suppose first that $\beta$ is a positive integer. In this case
$1/\phi_\beta$ is rational, so the reciprocal representation follows
from its finite partial-fraction decomposition. The formula for
$U_\beta$ then follows by Laplace inversion and uniqueness.

We now assume that $\beta=n_\beta+\theta_\beta$ is not an integer, where
$n_\beta=\lfloor\beta\rfloor$. Define
\begin{equation}
\label{app:eq:varphi-beta}
\varphi_\beta(z)
:=
z\left[
    \frac{1}{\phi_\beta(z)}
    -\eta_\beta
    -\frac{1}{\mu_\beta z}
    -\sum_{\rho\in\mathcal E_\beta}
     \mathcal P_{\beta,\rho}(z)
\right].
\end{equation}
Since
\[
        \phi_\beta(z)
        =
        \mu_\beta z+O_\beta(z^2),
        \qquad
        \text{as }z\to0,
\]
the singular part of $1/\phi_\beta$ at zero is
$1/(\mu_\beta z)$. The expression in brackets is therefore
holomorphic at zero, and $\varphi_\beta$ extends holomorphically there
with $\varphi_\beta(0)=0$.
The principal parts at the exceptional zeros in $\mathcal E_\beta$
have also been removed. The remaining poles of $\varphi_\beta$ are
therefore precisely the simple real zeros
$s\in\mathcal S_\beta$, and
\begin{equation}
\label{app:eq:varphi-residue}
\operatorname*{Res}_{z=s}\varphi_\beta(z)
=
\frac{s}{\phi_\beta'(s)}.
\end{equation}
We use directly the square contours $D_N$ from
Lemma~\ref{lem:g}. Apply that lemma with
\[
        n=n_\beta,
        \qquad
        \ell=r=\theta_\beta,
        \qquad\text{and}
        \qquad
        \gamma=1-\theta_\beta/2.
\]
Therefore $R_N = N-\gamma= N-1+{\theta_\beta}/{2}$,
and $D_N$ is the positively oriented boundary of the square
\[
        \left\{
        w\in\Cb:
        |\Re w|\leq R_N,
        \quad
        |\Im w|\leq R_N
        \right\}.
\]

Recall that
\[
        f_\beta(w)
        =
        \frac{\phi_\beta(w-\beta-2)}
        {\Gamma(\beta+1)}.
\]
We therefore define
\[
        C_N:=D_N-(\beta+2).
\]

By Lemma~\ref{lem:g}, for all sufficiently large $N$, the function
$f_\beta$ has neither zeros nor poles on $D_N$. Consequently,
$\varphi_\beta$ has no poles on $C_N$.

For sufficiently large $N$, the origin lies inside $C_N$. Also, the interiors of
the contours are increasing and cover $\Cb$. The distance to the origin is $d_N:=R_N-(\beta+2)$, and
trivially the length of $C_N$ is $\bo(d_N)$. 
Moreover, the estimate \eqref{eq:reciprocal-square} gives
\[
\sup_{w\in D_N}
\left|
        \frac1{f_\beta(w)}
        -
        \frac1{c(\beta)}
\right|
=
O_\beta(R_N^{-\alpha}),\qquad
\text{for }\alpha:=n_\beta+1-\theta_\beta>0.
\]

Since
\[
        \phi_\beta(z)
        =
        \Gamma(\beta+1)f_\beta(z+\beta+2),
        \qquad
        \text{and}
        \qquad
        \eta_\beta
        =
        \frac1{\Gamma(\beta+1)c(\beta)},
\]
it follows that
\[
\sup_{z\in C_N}
\left|
        \frac1{\phi_\beta(z)}-\eta_\beta
\right|
=
O_\beta(R_N^{-\alpha}).
\]

The set $\mathcal E_\beta$ is finite, and hence
\[
        \frac1{\mu_\beta z}
        +
        \sum_{\rho\in\mathcal E_\beta}
        \mathcal P_{\beta,\rho}(z)
        =
        O_\beta(R_N^{-1}),
        \qquad
        \text{uniformly for }z\in C_N.
\]
Since $|z|=O_\beta(R_N)$ uniformly on $C_N$, we deduce from
\eqref{app:eq:varphi-beta} that
\[
\frac1{d_N}
\sup_{z\in C_N}|\varphi_\beta(z)|
=
O_\beta(R_N^{-\alpha})
+
O_\beta(R_N^{-1})
\longrightarrow0.
\]
Thus
\[
        \sup_{z\in C_N}|\varphi_\beta(z)|
        =
        \so(d_N),
\]
and all the assumptions of
Lemma~\ref{lem:meromorphic_contours} are satisfied for all large $N$.
Applying
it, together with
\eqref{app:eq:varphi-residue} and $\varphi_\beta(0)=0$, gives
\[
\varphi_\beta(z)
=
\lim_{N\to\infty}
\sum_{s\in\mathcal S_\beta\cap\operatorname{int}(C_N)}
\frac{s}{\phi_\beta'(s)}
\left(
        \frac1{z-s}+\frac1s
\right)
=
z\lim_{N\to\infty}
\sum_{s\in\mathcal S_\beta\cap\operatorname{int}(C_N)}
\frac1{\phi_\beta'(s)(z-s)}.
\]
We next remove the contour grouping. By
Propositions~\ref{prop:realzeroesEven} and
\ref{prop:realzeroesOdd}, the residues $1/\phi_\beta'(s)$ have the same
sign for all sufficiently negative $s\in\mathcal S_\beta$. Since only
finitely many elements of $\mathcal S_\beta$ are not sufficiently
negative, evaluating the preceding contour limit at any $z_0>0$ gives
\begin{equation}\label{app:eq:absolute-residue-sum}
\sum_{s\in\mathcal S_\beta}
\frac1{|\phi_\beta'(s)|(z_0-s)}
<\infty.
\end{equation}
As in the proof of Proposition~\ref{prop:U-expression}, for every
compact set $K\subset\Cb\setminus\mathcal Z_\beta$ and every
$\varepsilon>0$, the estimates
\[
\sup_{z\in K}\frac1{|z-s|}
\leq
\frac{C_K}{z_0-s},
\qquad\text{and}\qquad
e^{sx}
\leq
\frac{C_{\varepsilon,z_0}}{z_0-s},\quad\text{for }
s<0 \text{ and } x\geq\varepsilon,
\]
show that the corresponding series converge locally absolutely and
uniformly on their respective domains. Consequently, the
contour-grouped limit agrees with the ordinary sum, and
\[
\varphi_\beta(z)
=
z\sum_{s\in\mathcal S_\beta}
\frac1{\phi_\beta'(s)(z-s)}.
\]
Substituting this into \eqref{app:eq:varphi-beta}, we obtain
\begin{equation}\label{app:phi-reciprocal}
\frac1{\phi_\beta(z)}
=
\eta_\beta
+
\frac1{\mu_\beta z}
+
\sum_{\rho\in\mathcal E_\beta}
\mathcal P_{\beta,\rho}(z)
+
\sum_{s\in\mathcal S_\beta}
\frac1{\phi_\beta'(s)(z-s)}.
\end{equation}

Now define
\[
\widetilde U_\beta(\D x)
:=
\eta_\beta\delta_0(\D x)
\quad+
\left[
        \frac1{\mu_\beta}
        +
        \sum_{\rho\in\mathcal E_\beta}
        u_{\beta,\rho}(x)
        +
        \sum_{s\in\mathcal S_\beta}
        \frac{e^{sx}}{\phi_\beta'(s)}
\right]\D x.
\]
The absolute convergence of the last series in
\eqref{app:phi-reciprocal}, together with the finiteness of
$\mathcal E_\beta$, justifies termwise Laplace transformation. Thus, for $z>0$,
\[
\int_{[0,\infty)}
e^{-zx}\widetilde U_\beta(\D x)
=
\eta_\beta
+
\frac1{\mu_\beta z}
+
\sum_{\rho\in\mathcal E_\beta}
\mathcal P_{\beta,\rho}(z)
+
\sum_{s\in\mathcal S_\beta}
\frac1{\phi_\beta'(s)(z-s)}=
\frac1{\phi_\beta(z)}.
\]
The potential measure $U_\beta$ has the same Laplace transform.
Arguing as in the proof of Proposition~\ref{prop:U-expression}, since
$\widetilde U_\beta$  has a finite exponential moment by
\eqref{app:eq:absolute-residue-sum}, we
multiply both measures by $e^{-z_0x}$ and use uniqueness of Laplace
transforms to conclude that
\[
        U_\beta=\widetilde U_\beta.
\]
This proves \eqref{eq:u-positive-global-results} and
\eqref{eq:reciprocal-positive-global-results}. We are left only with transferring the result for $U_\beta$ to $\Eb D_n$ and obtain \eqref{eq:EDn-positive-global-results}. By
\eqref{app:eq:absolute-residue-sum}, the expansion of $U_\beta$ may be
integrated term by term against $F_n$. For $s\in\mathcal S_\beta$,
\[
\int_0^\infty F_n(x)e^{sx}\D x
=
-\frac1s-B(-s,n),
\]
while, for $\rho\in\mathcal E_\beta$,
\[
\int_0^\infty F_n(x)u_{\beta,\rho}(x)\D x
=
\mathcal P_{\beta,\rho}(0)-\mathcal B_{\beta,\rho}(n).
\]
Moreover, taking the finite part at zero in
\eqref{app:phi-reciprocal}, we have
\[
\eta_\beta
+\sum_{\rho\in\mathcal E_\beta}\mathcal P_{\beta,\rho}(0)
-\sum_{s\in\mathcal S_\beta}
\frac1{s\phi_\beta'(s)}
=
-\frac{\phi_\beta''(0)}{2\mu_\beta^2}.
\]
Combining these identities with
$\Eb D_n=\mathcal U_\beta F_n$ gives exactly
\eqref{eq:EDn-positive-global-results}.
Finally, for every fixed $z_0>0$ and all $s\leq -z_0$,
\[
B(-s,n)
=
\int_0^1 t^{-s-1}(1-t)^{n-1}\D t\leq
\int_0^1 t^{-s-1}\D t
=
\frac1{-s}
\leq
\frac2{z_0-s},
\]
so the absolute convergence of the last series in
\eqref{eq:EDn-positive-global-results} follows from
\eqref{app:eq:absolute-residue-sum}. For positive integer $\beta$, the
same argument applies to the finite partial-fraction decomposition.
\end{proof}
         \section{Discrete-height asymptotics}
    \label{app:discrete-height}
    \begin{proof}[Proof of Corollary~\ref{cor:discrete-height-phase-results}]
    Put
    \[
            w_\beta(s):=s^{\beta+1}(1-s)^\beta,
            \qquad\text{for }0<s<1.
    \]
    Summing the formula for $a_\beta(n,j)$ in
    Theorem~\ref{thm:potential-identities-results}, writing $k=j-1$, and using
    \[
            \phi_\beta(k)
            =
            \int_0^1(1-s^k)w_\beta(s)\D s,
    \]
    Tonelli's theorem and the binomial formula give
    \begin{align*}
            \Eb L_n
            &=
            \int_{[0,\infty)}
            \sum_{k=1}^{n-1}
            \phi_\beta(k)\binom{n-1}{k}
            e^{-kx}(1-e^{-x})^{n-1-k}
            U_\beta(\D x)\\
            &=
            \int_0^1 w_\beta(s)
            \int_{[0,\infty)}
            \sum_{k=1}^{n-1}(1-s^k)\binom{n-1}{k}
            e^{-kx}(1-e^{-x})^{n-1-k}
            U_\beta(\D x)\,\D s\\
            &=
            \int_0^1 w_\beta(s)
            \int_{[0,\infty)}
            \left[1-\bigl(1-(1-s)e^{-x}\bigr)^{n-1}\right]
            U_\beta(\D x)\,\D s\\
            &=
            \int_0^1 w_\beta(s)
            \int_{[0,\infty)}
            F_n\bigl(x-\log(1-s)\bigr)
            U_\beta(\D x)\,\D s.
    \end{align*}
    We recall that we defined in
    \eqref{eq:def F_n}
    $F_n(x):=1-(1-e^{-x})^{n-1}=\P(M_{n-1}>x)$. Hence, for $s\in(0,1)$, 
    Tonelli's theorem gives
    \begin{align*}
            \int_{[0,\infty)}F_n(x-\log(1-s))U_\beta(\D x)
            &=
            \Eb\int_{[0,\infty)}
            \mathbf 1_{\{x<M_{n-1}+\log(1-s)\}}U_\beta(\D x)\\
            &=
            \Eb U_\beta\bigl((M_{n-1}+\log(1-s))_+\bigr),
    \end{align*}
    where, as above, $U_\beta(x)=U_\beta([0,x))$. After the change of
    variables $s=1-e^{-a}$, another application of Tonelli's theorem gives
    the exact representation
    \begin{equation}\label{eq:Ln-convolution-proof}
            \Eb L_n=\Eb V_\beta(M_{n-1}),
    \end{equation}
    where
    \[
            V_\beta(x)
            :=
            \int_0^x k_\beta(a)U_\beta(x-a)\D a,
            \qquad
            k_\beta(a)
            :=
            e^{-(\beta+1)a}(1-e^{-a})^{\beta+1}.
    \]
    The three regimes are determined by the behaviour of $k_\beta$ at
    infinity: it is integrable when $\beta>-1$, equals $1$ when
    $\beta=-1$, and grows exponentially when $-2<\beta<-1$.
    We shall use
    \[
            U_\beta(x)\sim\frac{x}{\mu_\beta},
            \qquad
            \text{and}
            \qquad
            U_\beta(x)\le C_\beta(1+x),
    \]
    from \eqref{eq:inverse-mean-proof} and
    \eqref{eq:renewal-linear-bound-proof}.
    
    Suppose first that $\beta>-1$. Then $k_\beta$ is integrable and
    \[
            \int_0^\infty k_\beta(a)\D a
            =
            B(\beta+2,\beta+1)
            =
            \phi_\beta(\infty).
    \]
    Moreover,
    \[
            \frac{V_\beta(x)}x
            =
            \int_0^\infty
            \mathbf 1_{\{a<x\}}k_\beta(a)
            \frac{U_\beta(x-a)}x\D a.
    \]
    The linear bound for $U_\beta$ provides an integrable majorant, and
    dominated convergence therefore gives
    \[
            V_\beta(x)
            \sim
            \frac{\phi_\beta(\infty)}{\mu_\beta}x,
            \qquad\text{as }x\to\infty.
    \]
    The same bound also gives $V_\beta(x)\le C_\beta(1+x)$.
    
    If $\beta=-1$, then $k_{-1}(a)=1$, and hence
    \[
            V_{-1}(x)
            =
            \int_0^xU_{-1}(u)\D u
            \sim
            \frac{x^2}{2\mu_{-1}},
            \qquad\text{as }x\to\infty.
    \]
    Moreover, $V_{-1}(x)\le C(1+x^2)$.
    
    Finally, suppose that $-2<\beta<-1$. Put $        \alpha:=-\beta-1$, so that $\alpha\in(0,1)$, and
    \[
            k_\beta(a)
            =
            e^{\alpha a}(1-e^{-a})^{-\alpha}.
    \]
    For $x\ge0$, define
    \[
            g_\beta(x)
            :=
            e^{-\alpha x}V_\beta(x).
    \]
    Changing variables  $u=x-a$ gives
    \[
            g_\beta(x)
            =
            \int_0^x
            e^{-\alpha u}
            (1-e^{-(x-u)})^{-\alpha}
            U_\beta(u)\D u.
    \]
    For $x>1$, split the integral at $x-1$. On the interval from $0$ to
    $x-1$, we have
    \[
            (1-e^{-(x-u)})^{-\alpha}
            \le
            (1-e^{-1})^{-\alpha}.
    \]
    The linear bound for $U_\beta$ therefore provides an integrable
    majorant which is a constant multiple of
    \[
            e^{-\alpha u}(1+u).
    \]
    For every fixed $u\ge0$, the remaining factor converges to $1$ as
    $x\to\infty$. Extending the integrand by zero outside the interval from
    $0$ to $x-1$, dominated convergence gives
    \[
            \int_0^{x-1}
            e^{-\alpha u}
            (1-e^{-(x-u)})^{-\alpha}
            U_\beta(u)\D u
            \longrightarrow
            \int_0^\infty e^{-\alpha u}U_\beta(u)\D u
    \]
    as $x\to\infty$.
    For the remaining interval, put $r=x-u$. The linear bound for
    $U_\beta$ gives
    \[
            \int_{x-1}^x
            e^{-\alpha u}
            (1-e^{-(x-u)})^{-\alpha}
            U_\beta(u)\D u
            \le
            C_\beta(1+x)e^{-\alpha x}
            \int_0^1
            e^{\alpha r}(1-e^{-r})^{-\alpha}\D r.
    \]
    The last integral is finite because its integrand is of order
    $r^{-\alpha}$ as $r\downarrow0$ and $\alpha<1$. Hence the right-hand
    side converges to zero as $x\to\infty$.
    
    Using the definition $U_\beta(u)=U_\beta([0,u))$ and Tonelli's theorem,
    we obtain
    \begin{align*}
            \int_0^\infty e^{-\alpha u}U_\beta(u)\D u
            &=
            \int_{[0,\infty)}
            \int_v^\infty e^{-\alpha u}\D u\,
            U_\beta(\D v)=
            \frac1\alpha
            \int_{[0,\infty)}
            e^{-\alpha v}U_\beta(\D v)=
            \frac1{\alpha\phi_\beta(\alpha)}.
    \end{align*}
    Consequently,
    \[
            g_\beta(x)
            \longrightarrow
            \frac1{\alpha\phi_\beta(\alpha)}
    \]
    as $x\to\infty$. The preceding estimates, together with the same
    integral bound for $0\le x\le1$, show that $g_\beta$ is bounded.
    Therefore, for every $x\ge0$,
    $V_\beta(x)\le C_\beta e^{\alpha x}.
    $
    
    We have obtained the asymptotics of $V_\beta(x)$ for large $x$. It remains to evaluate these asymptotics at $M_{n-1}$ in order to use them in \eqref{eq:Ln-convolution-proof}, that is, $\Eb L_n = \Eb V_\beta(M_{n-1})$. By
    \eqref{eq:maximum-exponential-proof},
    \[
            \frac{M_{n-1}}{\log n}\longrightarrow1
            \qquad\text{in }L^2.
    \]
    In particular,
    \[
            \frac{M_{n-1}^2}{(\log n)^2}\longrightarrow1
            \qquad\text{in }L^1.
    \]
    Together with the bounds for $V_\beta$, this gives the required uniform
    integrability in both cases $\beta\ge-1$. Consequently,
    \[
            \Eb V_\beta(M_{n-1})
            \sim
            \begin{cases}
            \displaystyle
            \frac{\phi_\beta(\infty)}{\mu_\beta}\log n,
            &\beta>-1,\\[1em]
            \displaystyle
            \frac1{2\mu_{-1}}(\log n)^2,
            &\beta=-1.
            \end{cases}
    \]
    Return now to the case $-2<\beta<-1$. Since $M_{n-1}$ is the maximum
    of $n-1$ independent standard exponential random variables,
    $e^{-M_{n-1}}$ is the minimum of $n-1$ independent uniform random
    variables.
    It follows that,
    as $n\to\infty$,
    \begin{align*}
            \Eb e^{\alpha M_{n-1}}
            &=
            (n-1)\int_0^1
            u^{-\alpha}(1-u)^{n-2}\D u=
            \frac{\Gamma(n)\Gamma(1-\alpha)}
                 {\Gamma(n-\alpha)}     \sim
            \Gamma(1-\alpha)n^\alpha
    \end{align*}
    Put $        \ell_\beta
            :=
            1/{\alpha\phi_\beta(\alpha)}.$
    We have shown that $g_\beta$ is bounded and converges to $\ell_\beta$
    as $x\to\infty$. Hence
    \[
            C_\beta
            :=
            \sup_{x\ge0}|g_\beta(x)-\ell_\beta|
            <
            \infty.
    \]
    Let $\varepsilon>0$ and choose $A>0$ such that
    $
            |g_\beta(x)-\ell_\beta|\le\varepsilon
            $ for $x>A.
    $
    Since
    \[
            V_\beta(M_{n-1})
            =
            e^{\alpha M_{n-1}}g_\beta(M_{n-1}),
    \]
    we have
    \begin{align*}
    \left|
            \Eb V_\beta(M_{n-1})
            -
            \ell_\beta\Eb e^{\alpha M_{n-1}}
     \right|&\leq
            \Eb\left[
            e^{\alpha M_{n-1}}
            |g_\beta(M_{n-1})-\ell_\beta|;
            M_{n-1}\le A
            \right]\\
    &\qquad+
            \Eb\left[
            e^{\alpha M_{n-1}}
            |g_\beta(M_{n-1})-\ell_\beta|;
            M_{n-1}>A
            \right]\\
    &\le
            C_\beta e^{\alpha A}
            +
            \varepsilon\Eb e^{\alpha M_{n-1}}.
    \end{align*}
    Dividing by $\Eb e^{\alpha M_{n-1}}$ gives
    \[
            \left|
            \frac{\Eb V_\beta(M_{n-1})}
                 {\Eb e^{\alpha M_{n-1}}}
            -
            \ell_\beta
            \right|
            \le
            \frac{C_\beta e^{\alpha A}}
                 {\Eb e^{\alpha M_{n-1}}}
            +
            \varepsilon.
    \]
    For fixed $A$, the first term converges to zero as $n\to\infty$.
    Letting then $\varepsilon\downarrow0$, we obtain
    \[
            \frac{\Eb V_\beta(M_{n-1})}
                 {\Eb e^{\alpha M_{n-1}}}
            \longrightarrow
            \frac1{\alpha\phi_\beta(\alpha)}.
    \]
    Therefore,
    \[
            \Eb V_\beta(M_{n-1})
            \sim
            \frac{\Gamma(1-\alpha)}
                 {\alpha\phi_\beta(\alpha)}
            n^\alpha
    \]
    as $n\to\infty$. Substituting $\alpha=-\beta-1$, proves the last part of
    \eqref{eq:discrete-height-phase-results}.
    \end{proof}

\subsection{Limit theorems for the discrete height}\label{appn: C1}

In this section we provide limit theorems for the discrete height $L_n$,
based on its interpretation through a regenerative composition generated by
the tagged subordinator.

We first describe semi-formally the connection. A composition of an integer
$m$ is an ordered sequence $(n_1,\ldots,n_k)$ of positive integers such that
$n_1+\cdots+n_k=m$; unlike for a partition, the order of the parts matters.
A family $(\mathcal C_m)_{m\ge1}$ of random compositions is called a
regenerative composition structure if, conditionally on the first part of
$\mathcal C_m$ being $r$, deleting this part leaves a composition distributed
as $\mathcal C_{m-r}$; see
\cite[Definition~1.1]{Gnedin-Pitman-2005}.

To see the connection with $L_n$, tag label~$1$ and follow its spine. At
every split visible in the restriction to $[n]$, some of the other $n-1$
labels leave the block containing label~$1$. Recording these numbers in the
order in which the splits occur gives a composition of $n-1$. Its number of
parts is exactly the discrete height of label~$1$, which, by exchangeability,
has the same distribution as $L_n$.

This can be visualised using a construction close to Kingman's paintbox.
Conditionally on the tagged fragment, we may represent the nested blocks
containing label~$1$ by intervals with a common left endpoint and lengths
\[
        |\Pi_1(t)|=e^{-\xi_t}.
\]
Place iid uniform random variables $U_1,U_2,\ldots$ on $(0,1)$, with $U_j$
corresponding to label $j+1$. In the restriction to $[n]$, label $j+1$
remains in the tagged block at time $t$ precisely when
\[
        U_j<e^{-\xi_t}, \qquad \text{for } 1\leq j\leq n-1.
\]
Thus, whenever the interval becomes shorter, the uniform points which are cut
off correspond exactly to the labels leaving the tagged block at that split.

Passing to logarithmic coordinates, put $E_j=-\log U_j$. Then
$E_1,E_2,\ldots$ are iid exponential random variables of rate $1$,
independent of $\xi$. Consider the closed range
\[
        \mathcal R_\beta
        :=
        \overline{\{\xi_t:t\ge0\}}.
\]
Its complement is a union of disjoint open intervals, called gaps, which
correspond to the jumps of $\xi$. Since $\xi$ has no drift, for example because $d=
\lim_{a\to\infty}\phi_\beta(a)/a=0$,
$\mathcal R_\beta$ has zero Lebesgue measure and, almost surely, every $E_j$
belongs to a gap. If $\xi$ jumps at time $t$ from $a=\xi_{t-}$ to
$b=\xi_t$, the corresponding gap is $(a,b)$ and the points in this gap
correspond exactly to the labels leaving the tagged block at that split.

For each $m\ge1$, list the gaps containing at least one of
$E_1,\ldots,E_m$ from left to right, which is also the order of the
corresponding splits along the tagged spine, and record the number of sample
points in each of them. The resulting ordered sequence
\[
        \mathcal C_m=(N_{m,1},\ldots,N_{m,K_m})
\]
is a composition of $m$, and $K_m$ is its number of parts, or equivalently
the number of occupied gaps. By
\cite[Theorem~5.2(i)]{Gnedin-Pitman-2005},
$(\mathcal C_m)_{m\ge1}$ is a regenerative composition structure.
Intuitively, this follows from the regenerative property of the range of a
subordinator together with the memoryless property of the exponential
sample.

For $\beta=-1$, Iksanov
\cite[p.~2]{Iksanov-2025-Regenerative} identified the harmonic descent chain
with the decrement chain of this composition. A stronger joint identity is
given in
\cite[Proposition~2.1]{Iksanov-Nikitin-Yakymiv-2026}; see also
\cite[Section~5.4.1]{beta2-arxiv} for the corresponding paintbox construction
in the critical case. The following lemma records that the same coupling
holds throughout the beta-splitting family.

    \begin{lemma}\label{lem:Ln-occupied-gaps}
    For every $n\ge2$,
    \begin{equation}\label{eq:Ln-occupied-gaps}
            L_n\stackrel{d}=K_{n-1}.
    \end{equation}
    Moreover, there is a coupling in which $K_{n-1}$ equals the discrete
    height of label $1$ almost surely.
    \end{lemma}
    
    \begin{proof}
    Tag label $1$ and put
    \[
            p_t:=|\Pi_1(t)|=e^{-\xi_t}.
    \]
    For $j\ge2$, define the first time label $j$ split from the tagged block as
    \[
            T_j:=\inf\{t\ge0:j\notin\Pi_1(t)\}.
    \]
Conditionally
on $\xi$, the variables $T_2,\ldots,T_n$ are iid with
\[
        \P(T_j>t\mid\xi)=p_t,
        \qquad \text{for }t\ge0,
\]
which follows from the paintbox construction. 
Let
$U_2,\ldots,U_n$ be iid uniform random variables, independent of $\xi$, and set
\[
        \widetilde T_j:=\inf\{t\ge0:p_t\le U_j\}.
\] 
Since $p$ is non-increasing,
\[
        \P(\widetilde T_j>t\mid\xi)
        =
        \P(U_j<p_t\mid\xi)
        =
        p_t,
\]
so, conditionally on $\xi$, the random variables
$\widetilde T_2,\ldots,\widetilde T_n$ are iid and have the same conditional
joint law as $T_2,\ldots,T_n$. We may therefore work with this version of
the departure times. In particular, simultaneously for every $t\ge0$,
\[
        j\in\Pi_1(t)
        \quad\Longleftrightarrow\quad
        U_j<e^{-\xi_t}       \quad\Longleftrightarrow\quad \xi_t<-\log U_j.
\]
    Put $E_j:=-\log U_j$. Then $E_2,\ldots,E_n$ are iid exponential random
    variables of rate $1$, independent of $\xi$. If $\xi$ jumps at time $t$
    from $a=\xi_{t-}$ to $b=\xi_t$, then
    \[
            j\in\Pi_1(t-)\setminus\Pi_1(t)
            \quad\Longleftrightarrow\quad
            e^{-b}\le U_j<e^{-a}          \quad\Longleftrightarrow\quad
            a<E_j\le b.
    \]
    Thus this jump produces a visible split on the ancestral line of label $1$
    in the restriction to $[n]$ if and only if the corresponding gap contains
    at least one of $E_2,\ldots,E_n$. The discrete height of label $1$ therefore
    equals $K_{n-1}$ almost surely in this coupling. Since label $1$ has the
    same distribution as a uniformly chosen label,
    \eqref{eq:Ln-occupied-gaps} follows.
    \end{proof}
    
\begin{corollary}[Central limit theorem for the discrete height]
\label{cor:Ln-clt-appendix}
If $\beta>-1$, then
\begin{equation}\label{eq:Ln-clt-finite-appendix}
        \frac{
        L_n-\dfrac{\phi_\beta(\infty)}{\mu_\beta}\log n
        }{
        \sqrt{
        \left(
        \dfrac{\phi_\beta(\infty)^2\tau_\beta^2}{\mu_\beta^3}
        -
        \dfrac{\phi_\beta(\infty)}{\mu_\beta}
        \right)\log n
        }}
        \xrightarrow{d}\mathcal N(0,1).
\end{equation}
At $\beta=-1$,
\begin{equation}\label{eq:Ln-clt-critical-appendix}
        \frac{
        L_n-\dfrac1{2\mu_{-1}}(\log n)^2
        }{
        \sqrt{
        \dfrac{\tau_{-1}^2}{3\mu_{-1}^3}(\log n)^3
        }}
        \xrightarrow{d}\mathcal N(0,1),
\end{equation}
where
\[
        \mu_{-1}=\zeta(2),
        \qquad\text{and}\qquad
        \tau_{-1}^2=2\zeta(3).
\]
\end{corollary}

\begin{proof}[Proof of Corollary~\ref{cor:Ln-clt-appendix}]
By Lemma~\ref{lem:Ln-occupied-gaps},
\[
        L_n\stackrel{d}=K_{n-1},
\]
so it remains to apply the corresponding limit theorems for the number of
occupied gaps.

Suppose first that $\beta>-1$ and put
\[
        c_\beta:=\phi_\beta(\infty)
        =B(\beta+2,\beta+1).
\]
The L\'evy measure of $\xi$ is finite and has total mass $c_\beta$. Note that because the
occupied gaps depend only on the jump sizes and not on the times at which
they happen, we may assume that the jump rate is one. Let $X$
be a positive random variable which has the jump distribution,
\[
        \P(X\in\D x)
        =
        \frac1{c_\beta}
        e^{-(\beta+2)x}(1-e^{-x})^\beta\D x.
\]
By the definitions of $\mu_\beta$ and $\tau_\beta^2$, see \eqref{eq:mu-tau-results},
we have that
\begin{equation}
    \label{eq: EX and VarX}
        \Eb X=\frac{\mu_\beta}{c_\beta},
        \qquad
        \text{and}
\qquad        \Var(X)
        =
        \frac{\tau_\beta^2}{c_\beta}
        -\frac{\mu_\beta^2}{c_\beta^2}.
\end{equation}
Define $W:=e^{-X}$ which measures the preserved fraction, so we also have that $1-W$ is the relative size of the gap. A calculation shows that
$     \Eb|\log(1-W)|
$ is finite so we are in the setting of \cite[Section 4]{Gnedin-Iksanov-2012}: using the second choice of centering in
\cite[Theorem~4.1(a)]{Gnedin-Iksanov-2012},
\[
        b_n
        :=
        \frac1m
        \int_0^{\log n}
        \P\bigl(|\log(1-W)|\leq z\bigr)\D z=   \frac1m
        \left(
        \log n-\Eb(|\log(1-W)\wedge\log n)
        \right)
        =
        \frac{\log n}{m}+\bo(1).
\]
Therefore the last theorem and Slutsky's lemma give, for
$m:=\Eb X$ and $\sigma^2:=\Var X$, that
\[
        \frac{
        K_N-m^{-1}\log N
        }{
        \sqrt{\sigma^2m^{-3}\log N}
        }
        \xrightarrow{d}\mathcal N(0,1).
\]
Substituting the expression for $m$ and $\sigma^2$ from \eqref{eq: EX and VarX} and using Lemma~\ref{lem:Ln-occupied-gaps} proves
\eqref{eq:Ln-clt-finite-appendix}.

At $\beta=-1$, the L\'evy density in
\eqref{eq:levy-density-prelim} is
\[
        \lambda_{-1}(x)=\frac1{e^x-1},
\]
so its total mass is infinite. Such cases are treated in \cite[Section 3]{Gnedin-Iksanov-2012} under assumptions on the function
$\Phi$ defined below. Indeed, by \cite[(6.2.3), (6.2.4), and (6.12.1)]{NIST-Handbook} for the last equality,
\[
        \Phi(t)
        :=
        \int_0^\infty
        \left(1-e^{-t(1-e^{-x})}\right)
        \lambda_{-1}(x)\D x
        =
        \int_0^1\frac{1-e^{-ty}}y\D y
       =    \log t+\gamma+\bo(t^{-1}e^{-t}),
\]
where $\gamma$ is Euler's constant. In particular,
$\varphi(t):=\Phi(e^t)\sim t$, so Condition~A in
\cite[Theorem~3.1(a)]{Gnedin-Iksanov-2012} is satisfied with index $1$.
As before, since
\[
        \Eb\xi_1=\mu_{-1},
        \qquad
        \text{and}
        \qquad
        \Var(\xi_1)=\tau_{-1}^2,
\]
the theorem gives
\begin{equation}\label{eq:Ln-critical-intermediate}
        \frac{
        K_n
        -\mu_{-1}^{-1}\displaystyle\int_1^n
        \frac{\Phi(y)}y\D y
        }{
        \Phi(n)
        \sqrt{\tau_{-1}^2\mu_{-1}^{-3}\log n}
        }
        \xrightarrow{d}
        \int_0^1 B(1-y)\D y,
\end{equation}
where $B$ is a Brownian motion. The random variable on the
right-hand side is a centred Gaussian and a calculation shows its variance is $1/3$.
Furthermore, we have seen that
$
        \Phi(n)=\log n + \gamma + \bo(n^{-1}e^{-n}),$
        so
\[
        \int_1^n\frac{\Phi(y)}y\D y
        =
        \frac12(\log n)^2+\gamma\log n+\bo(1).
\]
Substituting in
\eqref{eq:Ln-critical-intermediate}, using
Lemma~\ref{lem:Ln-occupied-gaps}, and applying Slutsky's lemma proves
\eqref{eq:Ln-clt-critical-appendix}.

Finally, a classical integral representation of the Riemann zeta
function, see e.g. \cite[Eq.~(25.5.1)]{NIST-Handbook}, gives, for $p>1$,
\[
        \int_0^\infty\frac{x^{p-1}}{e^x-1}\D x
        =
        \Gamma(p)\zeta(p),
\qquad\text{so }
        \mu_{-1}=\zeta(2)
        \text{ and }
\tau_{-1}^2=2\zeta(3).
\]
\end{proof}
\begin{remark}\label{rem:Ln-power-limit}
For $-2<\beta<-1$, the limit is not Gaussian. Let
$\alpha=-\beta-1\in(0,1)$. Then, as $x\downarrow0$,
\[
\int_x^\infty\lambda_\beta(y)\D y
       =
        \int_x^\infty
        e^{-(\beta+2)y}(1-e^{-y})^\beta\D y\sim
        \frac{x^{-\alpha}}{\alpha}.
\]
Define the exponential functional
\[
        I_{\beta,\alpha}
        :=
        \int_0^\infty e^{-\alpha\xi_t}\D t.
\]
Using the coupling in
Lemma~\ref{lem:Ln-occupied-gaps},
\cite[Corollary~5.2 and Theorem~6.3]{Gnedin-Pitman-Yor-2006}
give, as $n\to\infty$,
\begin{equation}\label{eq:Ln-power-limit}
        n^{-\alpha}L_n
        \longrightarrow
        \frac{\Gamma(1-\alpha)}{\alpha}
        I_{\beta,\alpha},
\qquad\text{almost surely and for all positive integer moments.}
\end{equation}
Note that the moments of exponential functionals can be easily computed
in terms of the Laplace exponent using the recurrence equation; see e.g.
\cite[Equation~(2.2)]{Minchev-Savov-2026} or \cite{Patie-Savov-2018}. In the present case, for every
positive integer $k$,
\[
        \Eb\left[I_{\beta,\alpha}^k\right]
        =
        \frac{k!}
        {\prod_{j=1}^k\phi_\beta(j\alpha)}.
\]
Taking $k=1$ shows directly that the mean in
\eqref{eq:Ln-power-limit} matches with the last line of
\eqref{eq:discrete-height-phase-results}.
\end{remark}

    \bibliography{Bibliography}
    \bibliographystyle{alpha}
            \end{document}